\documentclass{amsart}
\usepackage{latexsym,amsxtra,amscd,ifthen}
\usepackage{amsfonts}
\usepackage{verbatim}
\usepackage{amsmath}
\usepackage{amsthm}
\usepackage{amssymb}

\numberwithin{equation}{section}

\theoremstyle{plain}
\newtheorem{theorem}{Theorem}[section]
\newtheorem{lemma}[theorem]{Lemma}
\newtheorem{proposition}[theorem]{Proposition}
\newtheorem{corollary}[theorem]{Corollary}

\theoremstyle{definition}
\newtheorem{definition}[theorem]{Definition}
\newtheorem{example}[theorem]{Example}
\newtheorem{hypothesis}[theorem]{Hypothesis}

\newtheorem{remark}[theorem]{Remark}
\newtheorem*{remark*}{Remark}

\makeatletter              
\let\c@equation\c@theorem  
\makeatother

\newcommand{\Sh}{\mathcal S}

\newcommand{\fg}{\mathfrak g}

\DeclareMathOperator{\sign}{sign}

\DeclareMathOperator{\gldim}{gldim}

\DeclareMathOperator{\coker}{coker}

\DeclareMathOperator{\Ext}{Ext}
\DeclareMathOperator{\Cotor}{Cotor}
\DeclareMathOperator{\pdim}{projdim}

\DeclareMathOperator{\HB}{H}

\DeclareMathOperator{\gr}{gr}

\DeclareMathOperator{\Tor}{Tor}

\DeclareMathOperator{\GKdim}{GKdim}

\DeclareMathOperator{\PCdim}{PCdim}

\DeclareMathOperator{\im}{im}

\newcommand\PP{{\mathfrak P}}

\begin{document}

\title{Primitive Cohomology of Hopf algebras}

\author{D.-G. Wang, J.J. Zhang and G. Zhuang}

\address{Wang: School of Mathematical Sciences,
Qufu Normal University, Qufu, Shandong 273165, P.R.China}

\email{dgwang@mail.qfnu.edu.cn, dingguo95@126.com}

\address{Zhang: Department of Mathematics, Box 354350,
University of Washington, Seattle, Washington 98195, USA}

\email{zhang@math.washington.edu}

\address{Zhuang: Department of Mathematics, University of Southern California, Los Angeles 90089-2532, USA}

\email{gzhuang@usc.edu}

\begin{abstract}
Primitive cohomology of a Hopf algebra is defined by using a 
modification of the cobar construction of the underlying coalgebra.
Among many of its applications, two classifications are 
presented. Firstly we classify all non locally PI, pointed Hopf 
algebra domains of Gelfand-Kirillov dimension two; and secondly 
we classify all pointed Hopf algebras of rank one. The first 
classification extends some results of Brown, Goodearl and others 
in an ongoing project to understand all Hopf algebras of 
low Gelfand-Kirillov dimension. The second generalizes results of 
Krop-Radford and Wang-You-Chen which classified Hopf algebras of 
rank one under extra hypothesis. Properties and algebraic structures 
of the primitive cohomology are discussed.
\end{abstract}

\subjclass[2000]{Primary 16T05, 16E65; Secondary 16S34, 16S40}


\keywords{Hopf algebra, primitive cohomology, Gelfand-Kirillov 
dimension, pointed Hopf algebra, grouplike element, skew 
primitive element}

\maketitle


\setcounter{section}{-1}
\section{Introduction}
\label{zzsec0}
One motivation of this paper is to re-introduce cohomological invariants 
derived from the coalgebraic structure of Hopf algebras and to use 
them to study (mostly) infinite dimensional Hopf algebras. 
Recently significant progress has been made in classifying 
infinite dimensional noetherian Hopf algebras of low Gelfand-Kirillov 
dimension (which is abbreviated to GK-dimension from now on), see the 
papers by Brown, Goodearl, Lu, Wu, Wu-Liu-Ding and authors 
\cite{BZ,GZ, LWZ, WZZ2, WZZ3, WLD}. 
Another very important result is a classification of pointed Hopf 
algebra domains of finite GK-dimension with generic infinitesimal 
braiding in the work of Andruskiewitsch-Schneider \cite{AS2} and 
Andruskiewitsch-Angiono \cite{AA}. Invariants derived from the algebraic 
side of a Hopf algebra, such as GK-dimension, global dimension and 
homological integral, have been effectively used in the above mentioned 
papers. Invariants derived from the coalgebraic side of a Hopf 
algebra are also playing a crucial role in the study of infinite 
dimensional Hopf algebras, as illustrated in the present paper.

Throughout the introduction let $k$ be a base field that is algebraically
closed of characteristic zero. Our first goal is to show the 
following. We say an algebra is {\it affine} 
if it is finitely generated over the base field and say an algebra
is {\it PI} if it satisfies a polynomial identity. 

\begin{theorem}
\label{zzthm0.1} 
Let $H$ be an affine pointed Hopf domain such that $\GKdim H<3$. 
If $H$ is not PI, then $H$ is isomorphic to one of following:
\begin{enumerate}
\item[(1)]
The enveloping algebra $U(\mathfrak g)$ of the 2-dimensional solvable Lie 
algebra ${\mathfrak g}$.
\item[(2)]
The algebra $A(n,q)$ \cite[Construction 1.1]{GZ} {\rm{(}}see Example 
\ref{zzex5.4}{\rm{)}} for $n> 0$ and $q$ is not a root of unity.
\item[(3)]
The algebra $C(n)$ \cite[Construction 1.4]{GZ} {\rm{(}}see Example 
\ref{zzex5.5}{\rm{)}} for $n\geq 2$.
\end{enumerate}
\end{theorem}

Theorem \ref{zzthm0.1} was announced in \cite[Corollary 1.12]{WZZ2}
without proof. There are many PI Hopf domains of GK-dimension two, 
see \cite{GZ, WZZ2}. The statement without the ``affine'' hypothesis 
is given in Theorem \ref{zzthm7.4}.

Let $H$ be a pointed Hopf algebra and suppose that $\{C_i:=C_i(H)\}_{i\geq 0}$
is the coradical filtration of $H$ (we are not using $H_i$ here). 
Then the rank of $H$ is defined to be $\dim_k (k\otimes_{C_0}C_1) -1$.
Finite dimensional pointed Hopf algebras of rank one have been studied 
in \cite{KR, Sc} (and infinite dimensional ones in \cite{WYC}). 
In the original definition of a Hopf algebra of rank one given in 
\cite{KR, Sc, WYC}, it is required that $H$ is generated by $C_1$ 
as an algebra. Here we remove this requirement, namely, we don't 
assume that a Hopf algebra of rank one is generated by grouplikes and 
skew primitives. Our second goal is to show the following.

\begin{theorem}
\label{zzthm0.2}
Let $H$ be a pointed Hopf algebra of rank one. 
Then  $H$ is isomorphic to one of the following:
\begin{enumerate}
\item
$A_G(e, \chi)$ in Example \ref{zzex5.4} where $\chi(e)$ is either 1 or 
not a root of unity.
\item
$C_G(e, \chi, \tau)$ in Example \ref{zzex5.5}.
\item
$E_{G}(e,\chi, \ell,\lambda)$ in Example \ref{zzex8.1} 
where $\lambda$ is either 1 or $0$.
\item
$F_G(e, \chi, \ell)$ in Example \ref{zzex8.7}.
\item
$L_G(e, \chi, \ell, \eta)$ in Example \ref{zzex8.9}.
\item
$N_G(e, \chi, \ell, \xi)$ in Example \ref{zzex8.10}.
\item
$O_G(e,\chi,\ell,\eta)$ in Example \ref{zzex8.11}.
\item
$P_G(e,\chi,\ell,\eta)$ in Example \ref{zzex8.12}.
\item
$Q_G(e,\chi,\ell,\eta)$ in Example \ref{zzex8.13}.
\end{enumerate}
\end{theorem}

The Hopf algebras in Theorem \ref{zzthm0.2}(d-i) are new 
and deserve further study. 

{\it Primitive cohomology} is not new. Cohomological theory of 
coalgebras and comodules was introduced and studied by Eilenberg-Moore 
\cite{EM} in 1966, by Doi \cite{Do} in 1981, and then re-introduced by 
several others. We will be using
the definition given by Stefan-Van Oystaeyen \cite{SV} in 1998,
which is under a different name. In \cite{SV}, the 
primitive cohomology was used to study some classes of finite 
dimensional pointed Hopf algebra. A significant part of the paper 
is devoted to detail computations of the primitive cohomology of 
several families of infinite dimensional Hopf algebras. For
example, the primitive cohomology of the Hopf algebras listed in Theorem
\ref{zzthm0.1} and \ref{zzthm0.2} are computed explicitly. The proof of
these two theorems is dependent on the information of these
primitive cohomologies. Since a few years ago, the idea of the 
primitive cohomology has been used in an essential way in several 
projects, see \cite{Zh, WZZ3, W1, W2, WW, NWW}. 

The third main result is Theorem \ref{zzthm6.5}, which is computational
and very tedious. Since the statement is long, we refer the reader 
to Section 6. To further understand Hopf algebras 
of GK-dimension three, four or five, Theorem \ref{zzthm6.5} becomes
imperative. Several groups of researchers, including 
the authors, are working on Hopf algebras of low GK-dimension
and we understand that Theorem \ref{zzthm6.5} is extremely helpful.

What is the definition of the 
primitive cohomology? Given a coalgebra $C$ over $k$ 
with two grouplike elements $g$ and $h$, the $n$th primitive cohomology, 
denoted by $\PP^n_{g,h}(C)$, is defined to be the $n$th Hochschild 
cohomology of $C$ with coefficients in the special 1-dimensional 
$C$-bicomodule ${^g k^h}$, which can be defined by using a generalization 
of the cobar construction. More precise definition and elementary 
properties of $\PP^i_{g,h}(C)$ are given in Sections 
\ref{zzsec1}--\ref{zzsec3}. 
Since there will be a few other definitions based on $\PP^n_{g,h}(C)$, 
it would be beneficial to introduce a \emph{designated name} and a 
\emph{specified symbol} for this special Hochschild cohomology. 
Using our language, 
\cite[Theorem 1.2(a)]{SV} says that the first primitive cohomology 
$\PP^1_{g,h}(C)$ equals the quotient space 
$$P'_{h,g}(C):=P_{h,g}/k(g-h)$$ 
where $P_{h,g}$ is the set of $(h,g)$-primitive elements 
\cite[p. 67]{Mo}. The second primitive cohomology $\PP^2_{g,h}(C)$ has 
some interpretations (see \cite[Theorem 1.2(b) and Corollary 1.3]{SV}, 
or equivalently, Lemma \ref{zzlem2.3} and Proposition \ref{zzpro2.4}), 
which provide useful information about Hopf algebra extension. 

As demonstrated in \cite[Section 2]{SV}, it is very difficult to compute
$\PP^i_{g,h}(C)$ even for some well-studied finite dimensional Hopf 
algebras such as the Taft algebra. In Section \ref{zzsec6}, the primitive
cohomology of Hopf Ore extensions is discussed. Then we can compute 
primitive cohomology for several classes of infinite dimensional 
Hopf algebras, which leads to the proof of Theorems \ref{zzthm0.1} 
and \ref{zzthm0.2}.

Cohomology theory of modules over groups, Lie algebras and generally 
over Hopf algebras has been studied and applied to the representation 
theory for a long time. Some fundamental results on the cohomology 
ring of finite dimensional Hopf algebras opened the door to using 
geometric methods in the representation theory of finite dimensional 
Hopf algebras (see for example \cite{GK, FS}). A beautiful result of 
Mastnak-Pevtsova-Schauenburg-Witherspoon \cite{MPSW} states that, 
under some mild hypotheses, the cohomology ring of finite dimensional 
pointed Hopf algebras is finitely generated (or affine). We also consider 
primitive cohomology ring. 

Let $G(H)$ denote the group of all grouplike elements in a Hopf algebra 
$H$. Suggested by the idea that the primitive cohomology is a twisted 
version of the cobar construction, a natural multiplication exists 
in the sum of primitive cohomologies
$$\PP(H):=\bigoplus_{n\geq 0} \{\bigoplus_{g\in G(H)}\PP^n_{1,g}(H)\}$$
which is called the {\it primitive cohomology ring}. This construction
can be generalized to the case when $G(H)$ is replaced by any subbialgebra
$D\subset H$, see Definition \ref{zzdef10.2}. The {\it connected 
cohomology ring} of $H$ is defined to be
$\PP_{1}(H):=\bigoplus_{n\geq 0}\PP^n_{1,1}(H).$
Note that the cohomology ring $\PP_{1}(H)$ is related to the usual 
cohomology ring of the trivial module over its dual Hopf algebra 
(and these two cohomology rings are equal when $H$ is finite 
dimensional). We prove the following.

\begin{theorem}
\label{zzthm0.3} Let $H$ be a Hopf algebra and $G$ be the group of 
grouplike elements in $H$. Then there is a natural $G$-action 
on the graded algebras $\PP(H)$ and $\PP_{1}(H)$.
\end{theorem}

A generalized version of Theorem \ref{zzthm0.3} is given in Proposition 
\ref{zzpro10.4}. As indicated in the work of Ginzburg-Kumar \cite{GK} 
and Friedlander-Suslin \cite{FS}, the cohomology ring serves as a 
bridge between the representation theory and the associated geometry.
It is conceivable that the primitive cohomology ring contains more 
information than the cohomology ring does since $\PP(H)$ contains 
$\PP_{1}(H)$ as a graded subalgebra. The primitive 
cohomology ring is expected to link the derived category of 
comodules over $H$ to the derived category of graded modules over the 
$\PP(H)$ which should also induce a connection 
between the stable derived category of comodules over $H$ and 
noncommutative projective scheme associated to $\PP(H)$. 

\subsection*{Acknowledgments.} 
The authors thank Ken Brown and Ken Goodearl for many useful
conversations on the subject and thank John Palmieri for the 
proof of Lemma \ref{zzlem1.1}(3). 
D.-G. Wang was supported by the National Natural Science
Foundation of China (No 11471186) 
and the Shandong Provincial Natural Science Foundation of 
China (No. ZR2011AM013). J.J. Zhang was supported by the US 
National Science Foundation (Nos. DMS--0855743 and DMS--1402863). 

\section{Definition of primitive cohomology}
\label{zzsec1}

Throughout let $k$ be a commutative field, and all algebras and 
coalgebras are over $k$. Later we might assume 
that $k$ is algebraically closed of characteristic zero.

Cobar construction is a very useful tool in homological algebra \cite{FHT}. 
Let $C$ be a coalgebra with comultiplication $\Delta$ and let $g$ be 
a grouplike element. Let $\overline{C}$ be the kernel of the counit 
$\epsilon: C\to k$. Then $C=\overline{C}\oplus kg$. The reduced 
comultiplication $\overline{\Delta}: \overline{C}\to \overline{C}
\otimes \overline{C}$ is defined by
$$\overline{\Delta}(c)=\Delta(c)-(c\otimes g+g\otimes c)$$
for all $c\in \overline{C}$. It is easily checked that $\overline{\Delta}$ 
is coassociative and that $(\overline{C},\overline{\Delta})$ is a 
coalgebra without counit. Let $\Omega C$ denote the tensor algebra 
over $\overline{C}$ 
$$T(\overline{C}):=\bigoplus_{n\geq 0} (\overline{C})^{\otimes n}.$$
We use the convention that $V^{\otimes 0}=k$ for any vector space $V$. 
Define $\deg (c)=1$ for all $c\in \overline{C}$. Then the comultiplication 
$\overline{\Delta}:\overline{C}\to (\overline{C})^{\otimes 2}$ induces 
a unique differential $d$ in the tensor algebra $\Omega C$ such that 
$(\Omega C,d)$ is a dg (namely, differential graded) algebra. 
The dg algebra $(\Omega C,d)$ is called the {\it cobar construction} of 
$C$. The original definition of the cobar construction uses $C':=
\coker (kg\to C)$ instead of $\overline{C}$, but two definitions are
equivalent. The cobar construction can be defined in a more general 
setting for coaugmented dg coalgebras, see \cite[Definition, p. 271]{FHT}.

Let $g$ and $h$ be two (not necessarily distinct) grouplike elements 
in $C$. We now give the definition of primitive cohomology, which is 
equivalent to the definition of ${\bf H}^\bullet({^g k^ h},C)$ given 
by Stefan-Van Oystaeyen in \cite[Lemma 1.1]{SV} where ${^g k^h}$ denote 
the $1$-dimensional $C$-bicomodule with left and right $C$-comodule 
structure determined by $\rho_l(\alpha)= g\otimes \alpha$ and 
$\rho_r(\alpha)=\alpha\otimes h$ for all $\alpha\in k$. 
 Let $T_{g,h} C$ denote the tensor algebra over $C$, namely, 
$$T_{g,h} C:=T(C)=\bigoplus_{n\geq 0} C^{\otimes n}$$
with $\deg(c)=1$ for all $c\in C$. We define the map $\partial^0_{g,h}: 
k\to C$ by $\partial^0_{g,h}(1_k)=g-h$ where $1_k$ is the unit of the base
field $k$, and define the map 
$$\partial^1_{g,h}: C\to C\otimes C$$
by $\partial^1_{g,h}(c)=g\otimes c- \Delta(c)+c\otimes h$ for all
$c\in C$. Since $T_{g,h}C$ is a free algebra, $\partial^1_{g,h}$ extends
uniquely to a derivative of $T_{g,h}C$. But this is not what we
are going to proceed. Instead we define $\partial^n_{g,h}: C^{\otimes n}
\to C^{\otimes (n+1)}$, for every $n\geq 0$, as follows:
\begin{equation}
\label{E1.0.1}
\tag{E1.0.1}
\partial^n_{g,h}=g\otimes Id^{\otimes n}+\sum_{i=0}^{n-1} (-1)^{i+1} 
Id^{\otimes i}\otimes \Delta \otimes Id^{\otimes (n-i-1)}+(-1)^{n+1} 
Id^{\otimes n}\otimes h
\end{equation}
where $Id$ is the identity map of $C$. Let 
\begin{equation}
\label{E1.0.2}\tag{E1.0.2}
D_n=\begin{cases}
0& n=0\\
\Delta &n=1\\
\sum_{i=0}^{n-1} (-1)^i Id^{\otimes i}\otimes \Delta
\otimes Id^{\otimes (n-i-1)} &n\geq 2
\end{cases}.
\end{equation} 
Then $\partial^n_{g,h}=g\otimes 
Id^{\otimes n}-D_n+(-1)^{n+1} Id^{\otimes n}\otimes h$. 
Since $\Delta$ is coassociative,
$D_{n+1}D_n=0$ by a standard computation. Hence $(T_{g,h},D)$ is 
a dg algebra where $D$ denotes the derivation $(D_n)_{n\geq 0}$. The 
following lemma is known.

\begin{lemma}
\label{zzlem1.1} Retain the above notation.
\begin{enumerate}
\item[(1)]
\cite[Lemma 1.1]{SV}
$\partial^{n+1}_{g,h}\partial^n_{g,h}=0$ for all $n\geq 0$.
\item[(2)]
If $g=h$, then $\partial_{g,g}:=(\partial^n_{g,g})_{n}$ is a 
derivation of the graded algebra $T_{g,g}C$. Consequently, 
$(T_{g,g}C, \partial_{g,g})$ is a dg algebra.
\item[(3)]
Two dg algebras $(T_{g,g}C,\partial_{g,g})$ and $(\Omega C, d)$ are 
quasi-isomorphic.
\end{enumerate}
\end{lemma}

\begin{proof}
(3) This is well known and is a special case of a much more general
result, see for example, \cite[Theorems III.2.1 and III.2.4]{GJ}.
The following proof is given by John Palmieri.

The inclusion $\Omega C\to T_{g,g} C$ commutes with the 
differentials. For each $n$,
$$C^{\otimes n}=\overline{C}^{\otimes n} \oplus W^n$$
where $W^n$ is a direct sum of all possible tensor
products of copies of $\overline{C}$ with at least one copy of 
$k$. This decomposition gives rise to a decomposition of
complex $T_{g,g} C$:
$$T_{g,g} C=\Omega C\oplus W.$$
So is suffices to show that the complex $W$ is null homotopic.
Define a map $s$ of degree $1$ on $W$ as follows:
$$s([\cdots | c_k |g | c_{k+2}|\cdots ])=\begin{cases}
(-1)^{k} [\cdots | c_k | c_{k+2}|\cdots ]& {\text{if $c_i\in\overline{C}$
for all $i\leq k$}} \\ \quad & {\text{and $c_{j}=g$ for some $j> k+1$}}\\
\quad &\quad\\
0& {\text{otherwise.}}
\end{cases}$$
The induction shows that $sd+ds=-Id_W$. Then $Id_W$ is null homotopic
with $-s$ being the homotopy.
\end{proof}

As usual we use ${\text{H}}^n$ for $n$th cohomology of a complex, 
$H$ for a Hopf algebra and ${\bf H}$ for Hochschild (co)homology.
Following the definition given by Stefan-Van Oystaeyen \cite{SV},
we give the following definition of the primitive cohomology.

\begin{definition}
\label{zzdef1.2}
Let $C$ be a coalgebra.
\begin{enumerate}
\item [(1)]
Let $g$ and $h$ be grouplike elements in $C$. The {\it $n$th 
primitive cohomology} of $C$ based on $(g,h)$ is defined to be
$$\PP^n_{g,h}(C)={\text{H}}^n(T_{g,h}(C),\partial_{g,h})=
\ker \partial^n_{g,h}/\im \partial^{n-1}_{g,h}.$$
\item[(2)]
The {\it primitive cohomological dimension} of $C$ is defined to be
$$\PCdim C=\max\{ n \; |\; \PP^n_{g,h}(C)\neq 0, 
{\text{ for some $h,g$}}\}.$$
If $C$ has no grouplike element, we define $\PCdim C=-\infty$.
\end{enumerate}
\end{definition}

\begin{remark}
\label{zzrem1.3}
\cite[Lemma 1.1]{SV}
The primitive cohomology $\PP^n_{g,h}(C)$ is naturally isomorphic to 
the Hochschild cohomology ${\bf H}^n({^g k^h}, C)$ for all $g,h$ and
all $n$. 
\end{remark}

One advantage of presenting $\PP^n_{g,h}(C)$ as the $n$th cohomology 
group of $T_{g,h}(C)$ is that $\PP^n_{g,h}(C)$ is computable in some 
cases, which is very crucial for our applications. 

\section{Interpretations of $\PP^n_{g,h}(C)$ for $n\leq 2$}
\label{zzsec2}

Most of this section is well known, see \cite{SV}. Let $C$ be 
a coalgebra (or let $H$ be a Hopf algebra). We denote the coradical 
filtration of $C$ (or $H$) by $\{C_i\}_{i\geq 0}$, see \cite[p.60]{Mo} 
for a definition. Trying to reduce the appearance of the over-used 
letter $H$ and to avoid possible confusions, we will not use $H_i$ 
for the coradical filtration of $H$ in this paper. Let $G(C)$ be the 
set of grouplike elements in $C$. We call $C$ 
\emph{pointed} if $C_0=kG(C)$. We will use $g$ and $h$ for elements 
in $G(C)$. 

\begin{lemma}
\label{zzlem2.1} 
Let $C$ be a coalgebra and $G=G(C)$. 
\begin{enumerate}
\item[(1)]
$\PP^0_{g,h}(C)=0$ if $h\neq g$ and $\PP^0_{g,g}(C)=k$ for all $g\in G$.
\item[(2)]
\cite[Theorem 1.2(a)]{SV}
$\PP^1_{g,h}(C)=P_{h,g}(C)/(k(h-g))$ where
$P_{h,g}(C)$ is the set of $(h,g)$-primitive elements in $C$ defined
in \cite[p. 67]{Mo}. As a consequence, $\PP^1_{g,h}(C)
\cong P'_{h,g}(C)$ where $P'_{h,g}(C)$ is defined in 
\cite[p. 67]{Mo}.
\item[(3)]
If $C$ is pointed, then $C_1\cong kG\oplus \bigoplus_{g,h\in G} 
\PP^1_{g,h}(C)$.
\end{enumerate}
\end{lemma}

\begin{proof}
(1) This follows from the definition.

(3) This is \cite[Theorem 3.4.1]{Mo} together with part (2).
\end{proof}

By part (2) of the above lemma, $\PP^1_{g,h}(C)$ is the space of
$(h,g)$-primitive elements, which justifies calling
$\PP^n_{g,h}(C)$ \emph{primitive cohomology}.

Let $Z^n_{g,h}=\ker \partial^n_{g,h}$ and let $B^n_{g,h}=
\im \partial^{n-1}_{g,h}$. Then $\PP^n_{g,h}(C)=Z^n_{g,h}/B^n_{g,h}$.

\begin{proposition}
\label{zzpro2.2} Let $C$ be a pointed coalgebra. Then the following
are equivalent.
\begin{enumerate}
\item[(1)]
$C=C_0$.
\item[(2)]
$\PP^1_{g,h}(C)=0$ for all $g,h\in G$.
\item[(3)]
$\PCdim C=0$.
\end{enumerate}
\end{proposition}

\begin{proof} (2) $\Rightarrow$ (1)
If $\PP^1_{g,h}(C)=0$ for all $h,g\in G(C)$, then $C_1=C_0$ by Lemma 
\ref{zzlem2.1}(3). By the definition of the coradical filtration and 
the induction on $n$, one sees that $C_n=C_0$ for all $n$. Since $C$ 
is pointed, $C=\bigcup_n C_n =C_0$.

(3) $\Rightarrow$ (2) This is trivial.

A proof of (1) $\Rightarrow$ (3) is given after Lemma \ref{zzlem3.5}. 
\end{proof}

The following lemma is \cite[Theorem 1.2(b) and Corollary 1.3]{SV}.

\begin{lemma}
\label{zzlem2.3}  Let $C$ be a subcoalgebra of $D$ and 
$v$ be an element in $D$ such that $\Delta(v)=
v\otimes h+g\otimes v+ w$ and that $w\in C\otimes C$.
\begin{enumerate}
\item[(1)]
\cite[Theorem 1.2(b)]{SV}
$w\in Z^2_{g,h}(C)$.
\item[(2)]
$w\in B^2_{g,h}(C)$ if and only if there is an element $f\in C$
such that $v-f$ is an $(h,g)$-primitive element.
\item[(3)]
\cite[Corollary 1.3]{SV}
If $\PP^2_{g,h}(C)=0$, then there is an element $f\in C$
such that $v-f$ is an $(h,g)$-primitive.
\end{enumerate}
\end{lemma}

\begin{proof} (1) Since $w=\Delta(v)-g\otimes v-v\otimes h
=-\partial^1_{g,h}(v)$, $w\in B^2_{g,h}(D)\subset Z^2_{g,h}(D)$. 
Since $w\in C\otimes C$, $w\in Z^2_{g,h}(C)$.

(2) If $w\in B^2_{g,h}(C)$, there is $f\in C$ such that
$w=-\partial^1_{g,h}(f)$. Then $\partial^1_{g,h}(v-f)=-w+w=0$, 
namely, $v-f$ is an $(h,g)$-primitive element. The converse is 
clear. 

(3) If $\PP^2_{g,h}(C)=0$, $ Z^2_{g,h}(C)=  B^2_{g,h}(C)$.
The assertion follows from parts (1,2).
\end{proof}

We will use the following proposition several times.

\begin{proposition}
\label{zzpro2.4} Let $H$ be a pointed Hopf algebra and let
$K$ be the Hopf subalgebra of $H$ generated by all grouplike
and skew primitive elements in $H$. 
\begin{enumerate}
\item[(1)]
If $K\neq H$, then there is $c\in H\setminus K$ such that
$$\Delta(c)=c\otimes h+g\otimes c+w,$$
where $w$ represents a nonzero element in $\mathfrak{P}^2_{g, h}(K)$;
\item[(2)]
If $\PCdim K\leq 1$, 
then $K=H$.
\end{enumerate}
\end{proposition}

\begin{proof} Since $H$ is pointed, $H=\bigcup_i C_i$ where $C_i$ is 
the coradical filtration of $H$ \cite[Theorem 5.2.2(2)]{Mo}. By definition, 
$C_1\subseteq K$. 

For part (1), pick the 
minimal $n$ such that $C_n\not\subset K$ but $C_{n-1}\subseteq K$. 
Pick any $c\in C_{n}
\setminus K$. By \cite[Theorem 5.4.1(2)]{Mo}, we may assume further that
$$\Delta(c)= c\otimes h+g\otimes c+w$$
where $g,h\in G(H)$ and where $w\in C_{n-1}\otimes C_{n-1}$. So $w$ 
is in $K^{\otimes 2}$. By Lemma \ref{zzlem2.3}(1), $w\in Z^2_{g, h}(K)$. 
If $w\in B^2_{g, h}(K)$, then by Lemma \ref{zzlem2.3}(2),  there is an 
$f\in K$ such that $c-f$ is $(h,g)$-primitive.
So $c-f\in C_1\subseteq K$. Therefore $c\in K$, a contradiction.
Hence $w$ represents a nonzero element in $\mathfrak{P}^2_{g, h}(K)$.

Part (2) is a consequence of part (1).
\end{proof}

\section{Basic properties}
\label{zzsec3}
In this section we collect some elementary properties 
of the primitive cohomology. As in the previous sections, 
let $C$ be a coalgebra and $g$ and $h$ be grouplike elements in $C$.
Let $n$ be a natural number. The following lemmas are easy to verify.

\begin{lemma}
\label{zzlem3.1} 
Let $\phi: D\to C$ be a coalgebra homomorphism such 
that $h'=\phi(h)$ and $g'=\phi(g)$. Then $\phi$ naturally
induces maps $B^n_{g,h}(D)\to B^n_{g',h'}(C)$, 
$Z^n_{g,h}(D)\to Z^n_{g',h'}(C)$ and 
$\PP^n_{g,h}(D)\to \PP^n_{g',h'}(C)$.
\end{lemma}

\begin{lemma}
\label{zzlem3.2}
Let $C=\bigcup_{i\in S} D^i$ where $\{D^i\}_{i\in S}$ is a directed set 
of subcoalgebras of $C$ where $S$ is a directed index set. Suppose that 
each $D^i$ contains $g$ and $h$.
\begin{enumerate}
\item[(1)]
$B^n_{g,h}(C)=\bigcup_{i\in S} B^n_{g,h}(D^i)$.
\item[(2)]
$Z^n_{g,h}(C)=\bigcup_{i\in S} Z^n_{g,h}(D^i)$.
\item[(3)]
Fix $n$ and $i$. If the induced map $\PP^n_{g,h}(D^i)\to
\PP^n_{g,h}(D^j)$ is an isomorphism for every $j\geq i$, then 
$\PP^n_{g,h}(C)\cong \PP^n_{g,h}(D^i)$.
\item[(4)]
Fix an $n$. If $\PP^n_{g,h}(D^j)=0$
for all $j\in S$, then $\PP^n_{g,h}(C)=0$.
\item[(5)]
$\PCdim C\leq \max\{\PCdim D^i\mid  {\text{$i\in S$}}\}$.
\end{enumerate}
\end{lemma}

\begin{proof} (1,2) These assertions are clear from the definitions. 

(3) Let $V$ be a subspace of $Z^n_{g,h}(D^i)$
such that $Z^n_{g,h}(D^i)=V\oplus B^n_{g,h}(D^i)$. Then 
$V\cong \PP^n_{g,h}(D^i)$. Since the natural map $\phi_{i,j}:
\PP^n_{g,h}(D^i)\to \PP^n_{g,h}(D^j)$ is an isomorphism,
$Z^n_{g,h}(D^j)=V\oplus B^n_{g,h}(D^j)$ for every $j\geq i$.
Hence $V\cap \bigcup_j B^n_{g,h}(D^j)=\{0\}$. By parts (1,2) 
$Z^n_{g,h}(C)=V\oplus B^n_{g,h}(C)$ after taking the union.
The assertion follows. 

(4) This is an immediate consequence of (3). 

(5) This is a consequence of (4) by applying (4) to all 
pairs $(g,h)$.
\end{proof}

\begin{remark}
\label{zzrem3.3}
Let ${\mathbb G}$ be an abelian group.
If $C$ is a ${\mathbb G}$-graded coalgebra and $D^i$s are
${\mathbb G}$-graded subcoalgebras of $C$. Then the graded version of
Lemma \ref{zzlem3.2} holds.
\end{remark}

\begin{lemma}
\label{zzlem3.4}
Let $H$ be a Hopf algebra and $h_1,h_2$ and $g$ be grouplike elements.
Then the map $\Sh^n_g: H^{\otimes n}\to H^{\otimes n}$ defined by
$f\mapsto (g\otimes g\otimes \cdots \otimes g) f$ induces an isomorphism
from $\PP^n_{h_1,h_2}(H)$ to $\PP^n_{gh_1,gh_2}(H)$.
\end{lemma}

\begin{proof}
Since $\Sh_g$ commutes with differentials, it induces a map from 
$\PP^n_{h_1,h_2}(H)$ to $\PP^n_{gh_1,gh_2}(H)$. Since $\Sh_g$ is 
invertible (with the inverse $\Sh_{g^{-1}}$), the induced map 
is an isomorphism. 
\end{proof}

As a direct consequence of the previous lemma, for any Hopf algebra $H$,
$$\PCdim H=\max\{n | \PP^n_{1, g}(H)\neq 0, \,
\text{for some}\,\, g\}=\max\{n | \PP^n_{g, 1}(H)\neq 0, \,
\text{for some}\,\, g\}.$$

Next we give an interpretation of the primitive cohomology via 
certain cotorsion groups of $C$-comodules.
Let $M$ be a right $C$-comodule and $N$ a left $C$-comodule.
Then the cotorsion group $\Cotor^n_C(M,N)$ is defined to be the
$n$th derived functor of the cotensor functor $M\square_C N$
(see \cite[p. 31]{Do} for a definition). Note that, in \cite{NTZ} and some
other papers, the cotorsion group $\Cotor^n_C(M,N)$ is denoted by 
$\Tor^{C,n}(M,N)$. Below is a standard way of defining $\Cotor$ groups. 
We start with the Hochschild complex of the coalgebra $C$,
\begin{equation}
\label{E3.4.1}
(X,\partial_X)= \qquad
\cdots \to 0\to X^{-1}\to X^{0}\to X^{1}\to \cdots \to X^n\to \cdots
\tag{E3.4.1}
\end{equation}
where the $n$th term $X^{n}$ is $C^{\otimes (n+2)}$ and the 
$n$th differential is 
$$\partial_X^n=(D_{n+2}=)\sum_{i=0}^{n+1} (-1)^{i} Id^{\otimes i}
\otimes \Delta \otimes Id^{\otimes (n+1-i)}.$$
As a complex of left (or right) $C$-comodules, the Hochschild complex 
$(X,\partial_X)$ is homotopic to the zero complex. As a complex of 
$C$-bicomodules, $(X,\partial_X)$ is a free (or injective in the category
of $C$-bicomodules) resolution of $C$. So, for any right $C$-comodule $M$, 
the cotensor product $M\square_C X$ is homotopic to the zero complex of 
$k$-vector spaces. Consequently, ${\text{H}}^n(M\square_C X)=0$ for all $n$
(or equivalently, we say, $M\square_C X$ is acyclic) 
\cite[Proof of Corollary A1.2.12]{Rav}. Since we are mainly
interested in primitive cohomology, let $M$ be the 
right $C$-comodule $kg$ where $g$ is a grouplike element. Note that
$kg\square_C C\cong kg$, and hence $kg \square_C X$ becomes 
the following complex
\begin{equation}
\label{E3.4.2}
(Y,\partial_Y)= \qquad
\cdots \to 0\to Y^{-1}\to Y^{0}\to Y^{1}\to \cdots \to Y^n\to \cdots
\tag{E3.4.2}
\end{equation}
where $Y^{-1}=kg$, $Y^n=kg\otimes C^{\otimes (n+1)}$ for all $n\geq 0$, 
and where the $n$th differential is defined by
$$\partial^n_Y(g\otimes w)=g\otimes g\otimes w+g\otimes 
(\sum_{i=0}^{n} (-1)^{i+1} Id^{i}\otimes \Delta\otimes Id^{n-i})(w)$$
for all $g\otimes w\in Y^{n}$. Since $(Y,\partial_Y)$ is acyclic
(but may not be homotopic to zero as a complex of right $C$-comodules), 
the truncated complex $Y^{\geq 0}$ is an injective resolution of the 
right $C$-comodule $kg$. Hence $\Cotor^n_C(kg,kh)$ is the
$n$th cohomology of the cotensor product complex $Z:=Y^{\geq 0}\square_C kh$.
Using the fact that $C\square_C kh\cong kh$, one sees that
\begin{equation}
\label{E3.4.3}
(Z,\partial_Z)= \qquad
\cdots \to 0\to 0 \to Z^{0}\to Z^{1}\to \cdots \to Z^n\to \cdots
\tag{E3.4.3}
\end{equation}
where $Z^{0}=kg\otimes kh$, $Z^n=kg\otimes C^{\otimes n}\otimes kh$ 
for all $n\geq 1$, and where the $n$th differential is defined by
$$\begin{aligned}
\partial^n_Z&(g\otimes w\otimes h)=g\otimes g\otimes w\otimes h\\
&+g\otimes (\sum_{i=0}^{n-1} (-1)^{i+1} Id^{i}\otimes 
\Delta\otimes Id^{n-i-1})(w)\otimes h+(-1)^{n+1} g\otimes w\otimes h
\otimes h
\end{aligned}
$$
for all $g\otimes w\otimes h\in Z^n$. Using the isomorphism
of $kg\otimes C^{\otimes n}\otimes kh\cong C^{\otimes n}$, 
we obtain a natural isomorphism from $(Z,\partial_Z)$ to
$(T_{g,h}(C),\partial)$. Thus 
\begin{equation}
\label{E3.4.4}
\PP^n_{g,h}(C):={\text{H}}^n(T_{g,h}(C),\partial)\cong 
{\text{H}}^n(Z,\partial_Z)=\Cotor^n_C(kg,kh).
\tag{E3.4.4}
\end{equation}
A standard argument along this line shows that, if $M$ is a right 
$C$-comodule and $N$ is a left $C$-comodule, then 
\begin{equation}
\label{E3.4.5}
\Cotor^n_C(M,N)={\text{H}}^n(M\square_C X^{\geq 0}\square_C N)
\tag{E3.4.5}
\end{equation}
for all $n$. 
When $C$ is finite dimensional over $k$, \eqref{E3.4.4} is known
following by early results. For example, combining \cite[Theorem 3.4]{AW} 
with \cite[Proposition 1.4]{SV}, we have,
$$\Cotor^n_C(kg,kh)\cong {\text{\bf H}}^n(C^\circ,({^g k^h})^\circ)
\cong {\text{\bf H}}^n({^g k^h},C)\cong \PP^n_{g,h}(C)$$
where $(-)^\circ$ means the $k$-linear dual. 
As a summary, we have

\begin{lemma}
\label{zzlem3.5} Let $X$ be the complex defined in \eqref{E3.4.1}.
Let $\Cotor^n_C(-,-)$ be the $n$th derived functor of the cotensor 
functor $-\square_C -$.
\begin{enumerate}
\item[(1)]
Let $M$ {\rm{(}}resp. $N${\rm{)}} be a right {\rm{(}}resp. left{\rm{)}} 
$C$-comodule. Then 
$\Cotor^n_C(M,N)={\rm{H}}^n(M\square_C X^{\geq 0}\square_C N)$.
\item[(2)]
If $g,h$ are grouplike elements in $C$, then 
$\PP^n_{g,h}(C)\cong \Cotor^n_C(kg,kh)$.
\end{enumerate}
\end{lemma}

Now Proposition \ref{zzpro2.2} {\rm{(1) $\Rightarrow$ (3)}} is
easy. 

\begin{proof}[Proof of Proposition \ref{zzpro2.2} 
{\rm{(1) $\Rightarrow$ (3)}}]
When $C=C_0$, $\Cotor^n_C(M,N)=0$ for all $n\geq 1$ and all 
$C$-comodules $M,N$ (since $C$ is cosemisimple). Let $M=kg$ and $N=kh$, 
then $\Cotor^n_C(kg,kh)=0$ for all $n\geq 1$. By Lemma \ref{zzlem3.5}(2),
$\PP^n_{g,h}(C)=0$ for all $g,h$. This means that $\PCdim C=0$.
\end{proof}

Suppose now that $C$ is an ${\mathbb N}^d$-graded coalgebra
which is locally finite ${\mathbb N}^d$-graded. All grouplikes in $C$
must have degree 0. Then $C^\circ$ is 
a locally finite ${\mathbb N}^d$-graded algebra and $(C^\circ)^\circ\cong C$
where $(-)^\circ$ denotes the graded vector space dual. 
We extend the ${\mathbb N}^d$-grading from $C$ to $C^{\otimes n}$ 
by defining $\deg(a_1\otimes \cdots \otimes a_n)=\sum_{i=1}^n \deg(a_i)$.
In this setting,
\begin{enumerate}
\item[(i)]
the differentials $\partial^n_{g,h}$ preserve the grading, and 
\item[(ii)]
$C^{\otimes n}$ is locally finite for every $n\geq 0$. 
\end{enumerate}

For any grouplike element $g\in C$, it can be viewed as an
algebra map $g:C^\circ\to k$. We denote the 1-dimensional
$C^\circ$-module $C^\circ/(\ker g)$ by $k\xi_g$. The left (and right) 
$C^\circ$-module structure on $k\xi_g$ is given by $ a\cdot\xi_g=
g(a) \xi_g$ for all $a\in C^\circ$. 

\begin{lemma}
\label{zzlem3.6} Let $C$ be a locally finite ${\mathbb N}^d$-graded
coalgebra and let $g,h\in G(C)\subseteq C_0$. 
\begin{enumerate}
\item[(1)]
$\PP^n_{g,h}(C)$ is locally finite ${\mathbb N}^d$-graded for every $n$.
\item[(2)]
For each $n$,
$$(\PP^n_{g,h}(C))^\circ\cong (\Cotor^n_C(kg,kh))^\circ\cong 
\Tor_n^{C^\circ}(k\xi_g,k\xi_h).$$
\end{enumerate}
\end{lemma}

\begin{proof} (1) Since $C^{\otimes n}$ is locally finite for
all $n$, each term in the complex $Z$ (see \eqref{E3.4.3}) is 
locally finite. By \eqref{E3.4.4}, $\PP^n_{g,h}(C)$ is locally finite. 

(2) By \eqref{E3.4.4}, 
$$(\PP^n_{g,h}(C))^\circ\cong 
(\Cotor^n(kg,kh))^\circ=({\text{H}}^n(Z,\partial_Z))^\circ.$$
Since every term in the complex $Z$ is locally finite 
${\mathbb N}^d$-graded,
$$({\text{H}}^n(Z,\partial_Z))^\circ\cong 
{\text{H}}^n(Z^\circ,\partial_{Z^\circ}).$$
It remains to show that ${\text{H}}^n(Z^\circ,\partial_{Z^\circ})
\cong \Tor_n^{C^\circ}(k\xi_g,k\xi_h)$. Note that $(kg)^\circ\cong k\xi_g$.
Since $C^{\otimes n}$ is locally finite, we have
$$(Z^\circ)^{-n}=(Z^n)^\circ\cong k\xi_g\otimes 
(C^\circ)^{\otimes n}\otimes k\xi_h,$$
and the differential $(\partial_{Z^\circ})^{-n}$ maps an element
$\xi_g\otimes a_1\otimes \cdots \otimes a_n\otimes \xi_h$ to
$$ \begin{aligned}
\xi_g g(a_1)\otimes a_2\otimes & \cdots \otimes a_n\otimes \xi_h\\
& +\sum_{i=1}^{n-1} (-1)^{i} 
\xi_g \otimes a_1\otimes \cdots \otimes a_{i-1}\otimes a_ia_{i+1}
\otimes \cdots \otimes a_n\otimes \xi_h\\
&\qquad\qquad \qquad +(-1)^{n} \xi_g \otimes a_1\otimes \cdots 
\otimes a_{n-1}\otimes h(a_n)\xi_h.
\end{aligned}
$$
Using a statement similar to the discussion after Lemma \ref{zzlem3.4},
$(X^\circ,\partial_{X^\circ})$ is the Hochschild complex over the 
algebra $A:=C^\circ$,
which is a free resolution of the $A$-bimodule $A$. The complex
$(Y^\circ,\partial_{Y^\circ})=k\xi_g\otimes_A(X^\circ,\partial_{X^\circ})$ 
gives 
rise to a free resolution of the right $A$-module $k\xi_g$, and 
the complex $(Z^\circ,\partial_{Z^\circ})=(Y^\circ,\partial_{Y^\circ})
\otimes_A k\xi_h$
computes the torsion group $\Tor^{A}_{*}(k\xi_g,k\xi_h)$. 
Therefore ${\text{H}}^n(Z^\circ,\partial_{Z^\circ})\cong 
\Tor_n^{C^\circ}(k\xi_g,k\xi_h)$.
\end{proof}

The following proposition is a translation of some algebraic results.

\begin{proposition}
\label{zzpro3.7}
Let $C$ be a pointed coalgebra.
\begin{enumerate}
\item[(1)]
If $C$ is a locally finite ${\mathbb N}^d$-graded coalgebra, 
then $\PCdim C=\gldim C^\circ$ where $C^\circ$ is the graded dual algebra 
of $C$. 
\item[(2)]
For every $n\leq \PCdim C$, there are $g, h\in G(C)$ such that 
$\PP^n_{g,h}(C)\neq 0$.
\end{enumerate}
\end{proposition}

Every ${\mathbb Z}^d$-graded vector space can be viewed as a 
${\mathbb Z}$-graded space by setting ${\mathbb Z}^d$-degree 
$(a_1,\cdots,a_d)$ to ${\mathbb Z}$-degree $\sum_{i=1}^d a_i$.
We say a ${\mathbb Z}^d$-graded vector space $M$ is 
\emph{bounded below} if it is \emph{bounded below} when viewed 
as a ${\mathbb Z}$-graded vector space, which means, by definition, 
that there is an $n_0$ such that $M_{n}=0$ for all $n\geq n_0$.

\begin{proof}[Proof of Proposition \ref{zzpro3.7}] 
(1) Let $A=C^\circ$. It follows from Lemma \ref{zzlem3.6}(2) that 
$\PCdim C \leq \gldim C^\circ=\gldim A$. It remains to show 
$\PCdim C \geq \gldim A$.

By combining \cite[Theorem 7.6.17]{MR} with \cite[Corollary I.2.7]{NV2},
the global dimension of $A$ can be computed by using the category
of graded modules over $A$. 

Since $C$ is pointed, $A$ is basic (meaning that
every simple graded right (or left) $A$-module is 1-dimensional).
There is a one-to-one correspondence between the set of 
the grouplike elements $\{g\}$ in $C$ and the isomorphism classes 
of 1-dimensional graded right $A$-modules $\{k\xi_g\}$. 

Since $A$ is locally finite ${\mathbb N}^d$-graded, $A$ is graded 
semi-local and complete with respect to the graded Jacobson 
radical $J(A)$ \cite[p.52]{NV2}. Further $A/J(A)$ is a finite 
direct sum of 1-dimensional $A$-modules. 
Therefore the projective cover exists for every locally finite 
bounded below left and right graded $A$-modules. This implies that 
every locally finite and bounded below left graded $A$-module has a 
minimal projective resolution with each term being locally 
finite and bounded below. 

A standard result in ring theory, for example, a graded version of
\cite[Theorem 4.1.2]{We}, states that
$$\gldim A=\max\{\pdim (A/I)| {\text{ for all left graded ideals $I\subset A$}}\}.$$
Let $n$ be any integer no more than $\gldim A$. Then there is
a left $A$-module $M=A/I$ such that $\pdim M\geq n$.
Since $M$ is locally finite bounded below, $M$ has a minimal
projective resolution
\begin{equation}
\label{E3.7.1}
\tag{E3.7.1}
P:= \qquad
\cdots \to P^{-i}\to \cdots \to P^{-1}\to P^0\to M\to 0
\end{equation}
where, for each $i$, $P^{-i}$ is a locally finite and bounded 
below projective module. Since $\pdim (M)\geq n$, $P^{-n}\neq 0$,
whence, there is a simple module $k\xi_g$
such that $k\xi_g \otimes_A P^{-n}\neq 0$. Since \eqref{E3.7.1} is 
a minimal projective resolution of $M$, the differential in the 
complex $k\xi_g \otimes_A P$ is zero, which implies that
$\Tor_n^A(k\xi_g, M)\neq 0$. Consequently, $\pdim k\xi_g\geq n$.
Repeating the above argument for the right $A$-module $k\xi_g$ shows 
that there is a simple module $k\xi_h$ such that 
$\Tor_n^A(k\xi_g, k\xi_h)\neq 0$. By Lemma \ref{zzlem3.6}(2), 
$\PP^n_{g,h}(C)\neq 0$ and hence $\PCdim C\geq n$, as desired.

(2) If $C$ is finite dimensional, this follows from part (1) and Lemma 
\ref{zzlem3.6}(2). The next is a proof for general $C$. 

Let $n\leq \PCdim C$. Then there are $g,h\in G(C)$ and $m\geq n$ 
such that $\PP^m_{g,h}(C)\neq 0$. Consider the minimal injective
resolution of the simple left $C$-comodule $kh$:
\begin{equation}
\label{E3.7.2}\tag{E3.7.2}
I:= \qquad 0\to kh \to I^0\to I^1\to \cdots \to I^{i}\to \cdots .
\end{equation}
Since $\Cotor_C^m(kg,kh)\cong \PP^m_{g,h}(C)\neq 0$, $I^m\neq 0$. 
So $I^{n}\neq 0$ as $n\leq m$. Then $I^n$ contains a simple 
$C$-subcomodule, say $kw$, for some $w\in C(G)$. So 
$kw\square_{C} I^n\neq 0$. Since $I$ is 
minimal, the differential of the complex $kw\square_{C} I$ is zero.
This implies that $\Cotor_C^n(kw,kh)\neq 0$. By \eqref{E3.4.4},
$\PP^{n}_{w,h}(C)=\Cotor_C^n(kw,kh)\neq 0$.
\end{proof}

The global dimension of a coalgebra was defined in \cite{NTZ}.
The global dimension of $C$ equals  $\PCdim C$ when $C$ 
is a pointed coalgebra \cite[Corollary 3]{NTZ}. Proposition 
\ref{zzpro3.7} has an easy consequence.

\begin{corollary}
\label{zzcor3.8}
Let $H$ be a finite dimensional pointed Hopf algebra that is not 
a group algebra. Then, for every $n\geq 0$, there are grouplike 
elements $g$ and $h$ such that $\PP^n_{g,h}(H)\neq 0$.
\end{corollary}

\begin{proof} Since $H$ is pointed and is not a group algebra, 
it is not cosemisimple. This implies that $H^\circ$ is not semisimple, 
and hence it has infinite global dimension. If we define $\deg(f)=0$
for all $f\in H$, then $H$ is locally finite ${\mathbb N}$-graded. 
By Proposition \ref{zzpro3.7}(1), $\PCdim H=\infty$. The 
assertion now follows from  Proposition \ref{zzpro3.7}(2). 
\end{proof}

We state one more proposition.

\begin{proposition}
\label{zzpro3.9} 
For parts (1, 2, 3) we let $\{C_i\}_{i\in S}$ be a set of 
coalgebras and let $C=\bigoplus_{i\in S} C_i$.
\begin{enumerate}
\item[(1)]
If $g,h$ are grouplike elements in $C_i$ for some $i$, then 
$\PP^n_{g,h}(C)=\PP^n_{g,h}(C_i)$ for all $n$.
\item[(2)]
If $g\in C_i$ and $h\in C_j$ for $i\neq j$,  then $\PP^n_{g,h}(C)=0$.
\item[(3)]
$\PCdim C=\max\{\PCdim C_i \mid i\in S\}$.
\item[(4)]
Let $C_1$ and $C_2$ be two coalgebras and $g_i,h_i$ be grouplike
elements in $C_i$ for $i=1,2$. Then, for all $n$, 
$$\PP_{g_1\otimes g_2,h_1\otimes h_2}^n(C_1\otimes C_2)\cong 
\bigoplus_{s=0}^n
\PP_{g_1,h_1}^s(C_1)\otimes \PP_{g_2,h_2}^{n-s}(C_2).$$ 
In general, if $M_i$ are right $C_i$-comodules, and $N_i$ are left 
$C_i$-comodules, for $i=1,2$, then, for all $n$, 
$$\Cotor^n_{C_1\otimes C_2}(M_1\otimes M_2,N_1\otimes N_2)\cong 
\bigoplus_{s=0}^n \Cotor^s_{C_1}(M_1,N_1)\otimes 
\Cotor^{n-s}_{C_2}(M_2,N_2).$$
\item[(5)]
Let $H$ be a Hopf algebra and $K$ be a Hopf subalgebra of $H$ such that
$H=K[x^{\pm 1};\sigma]$ for some grouplike element $x\in H$.
Then $\PP^n_{g,h}(H)\cong \PP^n_{1,g^{-1}h}(K)$ if $g^{-1}h\in G(K)$ and
$\PP^n_{g,h}(H)=0$ if $g^{-1}h\not\in G(K)$. As a consequence, 
$\PCdim H=\PCdim K$.
\end{enumerate}
\end{proposition}

\begin{proof} (1,2) These follow from \eqref{E3.4.4} and
the definition of $\Cotor^n_C(kg, kh)$ by using the minimal
injective resolution of 1-dimensional comodule $kh$.

(3) The assertion follows from parts (1,2).

(4) These assertions follow from the K{\" u}nneth formula and standard argument.

(5) Note that $H=\bigoplus_{i\in {\mathbb Z}} x^i K$ where each $x^i K$ 
is isomorphic to $K$ as coalgebras. If $g,h\in x^i K$ (or equivalently, 
$g^{-1}h\in G(K)$), then $\PP^n_{g,h}(H)\cong \PP^n_{1,g^{-1}h}(H)=
\PP^n_{1,g^{-1}h}(K)$ by part (1). If $g,h$ are not in the same $x^i K$ 
(or equivalently, $g^{-1}h\not\in G(K)$), then by part (2) 
$\PP^n_{g,h}(H)=0$. The consequence is obvious.
\end{proof}

\section{Examples, Part I}
\label{zzsec4}

In this paper we will present quite a few examples, and 
some are new. In this section we give two easy and well-known examples. 
Let $\fg$ be a Lie algebra over $k$ and let $U(\fg)$ be the 
universal enveloping algebra of $\fg$. The only grouplike element
in $U(\fg)$ is 1. Let $\dim$ denote $\dim_k$. Again $n$ is a natural
number. 

\begin{example}
\label{zzex4.1} Suppose ${\rm{char}}\; k=0$ and let $H=U(\fg)$.
Then $\PP^n_{1,1}(H)$ is isomorphic to $\Lambda^{ n} \fg$ as 
vector spaces. As a consequence, $\PCdim U(\fg)=\dim \fg$.
\end{example}

\begin{proof} 
The consequence is clear from the main assertion.

Fix a totally ordered $k$-linear basis $\Phi=\{t_i\}$
of the Lie algebra $\fg$.
By PBW theorem, $U(\fg)$ has a $k$-linear basis
$$\{ t_1^{d_1}t_2^{d_2} \cdots t_n^{d_n}\;|\; n\geq 0,\;
t_1<t_2<\cdots < t_n,\;  d_s\geq 0, \; 
\forall s\}.$$
Clearly, 
$$\Delta(t_1^{d_1}t_2^{d_2} \cdots t_n^{d_n})
=\sum_{i_s} {d_1\choose i_1} {d_2\choose i_2}\cdots{d_n\choose i_n}
t_1^{i_1}t_2^{i_2} \cdots t_n^{i_n}\otimes t_1^{d_1-i_1}t_2^{d_2-i_2} 
\cdots t_n^{d_n-i_n}.$$
Therefore the coalgebra structure of $U(\fg)$ does not dependent on 
the Lie algebra structure of $\fg$. Consequently, we may just assume 
that $\fg$ is abelian.

Define $\deg(t_1)=(1,0, \cdots)$, $\deg(t_2)=(0,1,\cdots)$ and 
so on, and extend this ${\mathbb Z}^{\Phi}$-grading to 
$H^{\otimes m}$ for all $m\geq 1$. Note that the comultiplication
$\Delta$ preserves the grading.

We first assume that $\fg$ is finite dimensional (and abelian).
Say $d=\dim \fg$. By Lemma \ref{zzlem3.6}, $(\PP^{n}_{1,1}(H))^\circ\cong
\Tor^{H^\circ}_n(k,k)$ where $H^\circ$ is the graded dual of the coalgebra 
$H$, which is isomorphic to the commutative polynomial ring of 
$d$ variables. It is well-known that $\dim \Tor_n^{H^\circ}(k,k)
={d\choose n}$. Further, $\Tor_n^{H^\circ}(k,k)$ is 
${\mathbb N}^d$-graded with 1-dimensional component in each degree
$(e_1,\cdots, e_d)$ where $e_s$ is either 
0 or 1 and $\sum_{s=1}^d e_s=n$. 
Now we would like to identify a nonzero element in the 
$(e_1,\cdots, e_d)$-homogeneous 
space $\PP^{n}_{1,1}(H)_{(e_1,\cdots,e_d)}$. 
By a reduction from $\fg$ to a Lie subalgebra 
$\fg':=\bigoplus_{s: e_s=1} k t_{s}$, we may assume that 
$e_s=1$ for all $s$ (or equivalently, $d=n$). 
By Lemma \ref{zzlem1.1}(3) we can use 
the dg algebra $(\Omega C,d)$ to compute $\PP^n_{1,1}$. Consider the 
element
$$f=\sum a_{\sigma} t_{\sigma(1)}\otimes t_{\sigma(2)}\otimes 
\cdots \otimes t_{\sigma(d)}$$ 
where the sum is over permutations $\sigma$ in the symmetric group
$S_d$. Since 
$$
\begin{aligned}
\partial^{n-1}(t_{\sigma(1)}\otimes &t_{\sigma(2)}\otimes \cdots
\otimes t_{\sigma(i)}t_{\sigma(i+1)}\otimes \cdots
\cdots \otimes t_{\sigma(d)})\\
&=t_{\sigma(1)}\otimes t_{\sigma(2)}\otimes \cdots
\otimes t_{\sigma(i)}\otimes t_{\sigma(i+1)}\otimes \cdots
\cdots \otimes t_{\sigma(d)}\\
&\quad + t_{\sigma(1)}\otimes t_{\sigma(2)}\otimes \cdots
\otimes t_{\sigma(i+1)}\otimes t_{\sigma(i)}\otimes \cdots
\cdots \otimes t_{\sigma(d)},
\end{aligned}
$$
we have, modulo $B^{n}_{1,1}(H)_{(1,\cdots,1)}$, that 
$$f\equiv a t_1\otimes t_2\otimes \cdots \otimes t_d, 
\quad a=\sum (-1)^\sigma a_\sigma$$
which is in $Z^d_{1,1}(H)_{(1,\cdots,1)}$. 
The above computation also shows that 
$t_1\otimes t_2\otimes \cdots \otimes t_d$ is not in 
$B^n_{1,1}(H)_{(1,\cdots,1)}$. 
Therefore the class of $t_1\otimes t_2\otimes \cdots \otimes t_d$
is a nonzero homogeneous element in $\PP^n_{1,1}(H)_{(1,\cdots,1)}$. 
This ends the proof when $\fg$ is finite.

Since we may assume that $\fg$ is abelian, $\fg$ is a union 
of finite dimensional Lie subalgebras. 
Next we will apply Lemma \ref{zzlem3.2}(3). Let $\{D_i\}$ be the family
of Hopf subalgebras of $H$ generated by various finite subsets of the
chosen basis of the Lie algebra $\fg$. Fix any degree
in ${\mathbb N}^{\Phi}$. The hypotheses in the graded version of 
Lemma \ref{zzlem3.2}(3) holds by the finite dimensional case. 
Hence the assertion follows from the graded version of 
Lemma \ref{zzlem3.2}(3).
\end{proof}

\begin{remark}
\label{zzrem4.2} 
Let $\{t_i\}_{i\in I}$ be any fixed $k$-linear basis of 
${\mathfrak g}$ with a total order in $I$. The above proof shows 
that $\PP_{1}^n(U({\mathfrak g}))$ has a $k$-linear basis
$$\{ t_{i_1}\otimes \cdots \otimes t_{i_n} \;|\; i_1<i_2<\cdots <i_n,
\; {\text{and}}\; i_s\in I\}.$$
After the cohomology ring is defined in Section \ref{zzsec10},
one can easily show that $\bigoplus_{n\geq 0} \PP^n_{1,1}(H)$ is 
isomorphic to $\Lambda (\fg):=\bigoplus_{n\geq 0}\Lambda^{ n} \fg$ 
as graded algebras. 
\end{remark}

The following example is easy and follows from Lemma \ref{zzlem3.5}(2).

\begin{example}
\label{zzex4.3}
Let $C$ be a cosemisimple coalgebra. Then $\PCdim C=0$ if
$C$ has a grouplike element and $\PCdim C=-\infty$ if $C$
has no grouplike element.
\end{example}

\section{Hopf Ore extensions}
\label{zzsec5}
One way of constructing new examples is to use Hopf Ore extension.
The Hopf Ore extension was studied by many researchers, see 
\cite{BDG, Pa, W2, WYC, BOZZ}, for example. The most general one is given 
in \cite{BOZZ}. We use a version in between \cite{Pa} and \cite{BOZZ}.

Let $H$ be a Hopf algebra and $K$ be a Hopf subalgebra of $H$. 
Suppose that 
\begin{enumerate}
\item[(i)]
$\sigma$ is an algebra automorphism of $K$,
\item[(ii)]
$\delta$ is an $\sigma$-derivation of $K$, and 
\item[(iii)]
$z\in H\setminus K$. 
\end{enumerate}
We say $H$ is a {\it Hopf Ore extension} (or {\it HOE}) of $K$ if 
$H=K[z;\sigma, \delta]$ as an algebra and if 
$$\Delta(z)=z\otimes 1+e\otimes z+z_0$$ where $e$ 
is a grouplike element in $K$ and $z_0\in K^{\otimes 2}$. By the
antipode axiom, we have $S(z)=-e^{-1} z+\beta$ for some $\beta\in K$.

In most of the papers, for example, see \cite{Pa}, it is assumed that $z_0=0$. 
In \cite{BOZZ}, authors consider a slightly more general
setting by assuming that
$$\Delta(z)=z\otimes 1+e\otimes z+z_0+v(z\otimes z)$$
for some $v\in K^{\otimes 2}$. In this paper we always assume
that $v=0$ (but $z_0$ can be nonzero). 
Recall that the algebra structure of the Ore extension 
$H=K[z;\sigma, \delta]$ is determined by relation
$$zf=\sigma(f)z+\delta(f), \quad \forall \quad f\in K.$$

Both $C(n)$ \cite[Construction 1.4]{GZ} and 
$A(n,q)$ \cite[Construction 1.1]{GZ} are HOEs of 
$K=k {\mathbb Z}$. At the end of this section we will give two 
examples which are generalizations of $C(n)$ and $A(n,q)$. 
Let $k^{\times}$ be $k\setminus \{0\}$.

\begin{lemma} 
\label{zzlem5.1}
Retain the above notation and let $H=K[z;\sigma,\delta]$ be a
HOE of $K$.
\begin{enumerate}
\item[(1)]
$\sigma(e)=q e$ for some  scalar $q\in k^\times$. If $z_0=0$, then 
$\delta(e)\in Z^1_{e^2,e}(K)$. 
\item[(2)]
If $K$ is commutative {\rm{(}}or $e$ commutes with $\delta(e)${\rm{)}}
and $q\neq 1$ and $z_0=0$, we may assume 
that $ze=qez$ after 
replacing $z$ by $z+a$ for some $a\in K$.
\item[(3)]
\cite[Proposition 2.5]{BOZZ} 
$G(H)=G(K)$.
\end{enumerate}
\end{lemma}

\begin{proof} (1) By \cite[Theorem 2.4(d)]{BOZZ}, there is a 
character $\chi: K\to k$ such that $\sigma(r)=\chi(r_1)\; r_2$.
So $\sigma(e)=q e$ where 
\begin{equation}
\label{E5.1.1}\tag{E5.1.1}
q=\chi(e).
\end{equation} 
Applying \cite[Theorem 2.4(e)]{BOZZ} to element $r=e$, 
we obtain that $\Delta(\delta(e))=\delta(e)\otimes e+e^2\otimes
\delta(e)$. Hence $\delta(e)\in Z^1_{e^2,e}(K)$. 

(2) Suppose $q\neq 0$ and let $a=(q-1)^{-1}e^{-1}\delta(e)$. 
If $e$ commutes with $\delta(e)$, then $e$ commutes with $a$
and $(z+a)e=qe(z+a)$. It is easy to see that $a\in Z^1_{e,1}(H)$,
whence $z':=z+a\in Z^1_{e,1}(H)$ or $\Delta (z')=z'\otimes 1+
e\otimes z'$. 

(3) This is \cite[Proposition 2.5]{BOZZ}  when $v=0$.
\end{proof}

Let $q\in k^\times$.
Recall that $q$-integer $[n]_q$
is $1+q+q^2+\cdots + q^{n-1}$ for any positive integer
$n$. The $q$-factorial $[n]_q!$ is defined to be $\prod_{i=0}^n [i]_q$.
Then $q$-binomial coefficients are defined to be, for $n\geq m\geq 0$, 
$${n\choose m}_{q}:=\frac{[n]_q!}{[m]_q! [n-m]_q !}.$$
When $q=1$ we recover the usual binomial coefficients. 
All these definitions and the following lemma can be found in 
\cite[Section 6.5]{LR}.

\begin{lemma}
\label{zzlem5.2}
Assume ${\rm{char}}\; k=0$. Suppose $0\leq m\leq n$.
\begin{enumerate} 
\item[(1)]
If $q$ is either 1 or not a root of unity, then 
${n\choose m}_{q}\neq 0$.
\item[(2)]
\cite[Proposition 6.5.1(b)]{LR}
Assume that $q$ is a primitive $\ell$th root of unity. Let $n=q_n \ell+r_n$
and $m=q_m\ell +r_m$ where $0\leq r_n, r_m \leq \ell$. Then 
$${n\choose m}_{q}={r_n\choose r_m}_{q}{q_n\choose q_m}_{1}.$$
\item[(3)]
${n\choose m}_{q}=0$ if and only if $r_n<r_m$.
\end{enumerate}
\end{lemma}

Let $H:=K[z;\sigma,\delta]$ be a HOE of $K$. Define $\deg(z)=1$ and 
$\deg(a)=0$ for all $a\in K$. If $z_0=0$ and $\delta=0$, then $H$ 
is an ${\mathbb N}$-graded Hopf algebra with respect to this
grading. In general, for each nonzero element $f\in H$, we define 
the degree of $f$ to be $n$ if $f=\sum_{i=0}^n f_i z^i$ where
$f_i\in K$ and $f_n\neq 0$. Then $H$ becomes an ${\mathbb N}$-filtered 
algebra. Any nonzero element $f\in H^{\otimes 2}$ 
can be written as $f=\sum_{i=0}^n (\sum_{a+b=i} f_{a,b} z^a\otimes z^b)$ 
with $f_{a,b}\in K^{\otimes 2}$ and  $\sum_{a+b=n} f_{a,b} z^a\otimes 
z^b\neq 0$. In this case the degree of $f$ is defined to be $n$. 
Similarly, $H^{\otimes 2}$ is an ${\mathbb N}$-filtered algebra. 
The following lemma implies that $\deg (\Delta(f))=\deg(f)$
for any nonzero $f\in H$.

\begin{lemma} 
\label{zzlem5.3}
Retain the notation as above. Then 
\begin{equation}
\label{E5.3.1}\tag{E5.3.1}
\Delta(z^n)=\sum_{i=0}^n {n \choose i}_q e^{n-i} z^i\otimes z^{n-i}
+ldt
\end{equation}
where $ldt$ means a linear combination of elements in $H^{\otimes 2}$ of
lower degree, namely, of  degree strictly  less than $n$.
If $z_0=0$ and $ze=qez$, then $ldt=0$. 
\end{lemma}

Next we give some well-known examples of HOE, which were  introduced by
several researchers in  \cite{KR, WYC, GZ}. 

\begin{example}
\label{zzex5.4} 
Let $K=kG$ where $G$ is a group and let $\chi:G\to k^{\times}$ be 
a character. Define an algebra automorphism $\sigma_{\chi}:
K\to K$ by $\sigma_{\chi}(g)=\chi(g)g$ for all $g\in G$.
Let $\delta=0$. Then it is easy to check that 
$H:=K[z;\sigma_{\chi}]$ is a HOE of $K$ with 
$\Delta(z)=z\otimes 1+e\otimes z$ for any $e$ in the center of 
$G$. By \cite[Theorem 2.4(d)]{BOZZ}, $\sigma_{\chi}$ is the only 
possible automorphism to make $H$ a HOE. This Hopf algebra is 
denoted by $A_G(e,\chi)$.

A special case is when $G={\mathbb Z}$, and we write 
$K=k[x^{\pm 1}]$ for a grouplike generator $x$. Then 
$A_{\mathbb Z}(e,\chi)$ is determined by $e:=x^n$ for some 
$n\in{\mathbb Z}$ and $q=\chi(x)\in k^{\times}$. The Hopf 
algebra $A_{\mathbb Z}(e,\chi)$ is isomorphic to 
$k\langle x^{\pm1},z \mid xz=qzx\rangle$. By \cite[Construction 1.1]{GZ}, 
this Hopf algebra is denoted by $A(n,q)$.
If $m\in{\mathbb Z}$ and $r\in k^{\times}$, then $A(m,r)\cong A(n,q)$ 
if and only if either $(m,r)= (n,q)$ or $(m,r)= (-n,q^{-1})$.
\end{example}

\begin{example} 
\label{zzex5.5}
Let $K=kG$ where $G$ is a group and $e$ be an element in the
center of $G$. Let $\chi: G\to k^\times$ be a character of 
$G$ such that $\chi(e)=1$. Define an algebra automorphism 
$\sigma_{\chi}$ of $K$ by $\sigma_{\chi}(g)=\chi(g) g$ for all
$g\in G$, as in Example \ref{zzex5.4}. Let $\tau: G\to k$ be a map satisfying
\begin{equation}
\label{E5.5.1}\tag{E5.5.1}
\tau(gh)=\tau(g)+\chi(g)\tau(h), \quad \forall\;\; g, h\in G.
\end{equation}
Define a $k$-linear map $\delta: K\to K$ by
$\delta_{\tau}(g)=\tau(g)g(e-1)$ for all $g\in G$. Then 
\eqref{E5.5.1} is equivalent to the condition that 
$\delta$ is a $\sigma_{\chi}$-derivation. 
It is not hard to show that 
$H:=K[z;\sigma_{\chi},\delta_{\tau}]$ is a HOE of $K$ with 
$\Delta(z)=z\otimes 1+e\otimes z$. This Hopf algebra is denoted 
by $C_G(e,\chi, \tau)$.

A special case is when $\chi$ is the trivial character and 
$\tau:G\to (k,+)$ is an additive character. 
In this case $\sigma_{\chi}=Id_{K}$, and  
the Hopf algebra is denoted by $C_G(e,\tau)$.

A further special case is when $G={\mathbb Z}$, write 
$K=k[x^{\pm 1}]$ for a grouplike 
generator $x$. If $\tau(x)=0$, then this is isomorphic to $A(n,1)$.
If $\tau(x)\neq 0$, we may assume that $\tau(x)=1$ after 
replacing $z$ by a scalar multiple. In this case $C_{\mathbb Z}(e,\tau)$ 
is determined by $e=x^n$. This algebra is isomorphic to $C(n)$ 
\cite[Construction 1.4]{GZ} for a different $n$ due to the change
of definition. 

Let us recall the definition of the algebra $C(n)$ next.
Let $n$ be an integer, and set $C=k[y^{\pm1}] \bigl[ x;
(y^n-y)\frac{d}{dy} \bigr]$. There is a Hopf
algebra structure on $C$ such that $\Delta(y)=y\otimes y$ 
and $\Delta(x)= x\otimes y^{n-1}+ 1\otimes x$. This 
Hopf algebra is denoted by $C(n)$.
For $m,n\in{\mathbb Z}_{>0}$, the algebras $C(m)$ and $C(n)$
are isomorphic if and only if $m=n$ \cite[Construction 1.4]{GZ}.
Clearly $C(n)\cong C_{\mathbb Z}(x^{1-n},\tau)$ where $\tau(x^s)=s$ for
all $s\in {\mathbb Z}$. 
\end{example} 

In the next section we will use a special element, denoted 
by $[z]^{\ell}$, in $H^{\otimes 2}$. Let
$\ell$ be an integer larger than $1$ such that $[i]_{q}
\neq 0$ for all $i<\ell$ and let
\begin{equation}
\label{E5.5.2}\tag{E5.5.2}
[z]^{\ell}:=\sum_{i=1}^{\ell-1} \frac{[\ell-1]_q !}{[i]_q ! [\ell-i]_q !} 
e^{\ell -i}z^i\otimes z^{\ell-i}.
\end{equation}
Note that $[\ell]_q [z]^{\ell}=\sum_{i=1}^{\ell-1}{\ell\choose i}_q 
e^{\ell -i}z^i\otimes z^{\ell-i}$.

\begin{lemma}
\label{zzlem5.6} Let $H$ be a Hopf algebra and $e, z$ be elements in 
$H$ such that $ze=qez$ for some $q\in k^\times$. If $q=1$, we 
further assume that ${\rm{char}}\; k=0$. Suppose that 
$e$ is a grouplike and $\Delta(z)=z\otimes 1+e\otimes z$.
Let $\ell\geq 2$. 
\begin{enumerate}
\item[(1)]
$[\Delta(z), [z]^{\ell}]=(q-1)ez^{\ell}\otimes z+(1-q)
e^{\ell}z\otimes z^{\ell}$.
\item[(2)]
$[z]^\ell\in Z^2_{e^{\ell},1}(H)$.
\end{enumerate}
\end{lemma}

\begin{proof}  The proof for the case $q=1$ is similar, so we 
assume that $q\neq 1$.

(1) Let $A_{\mathbb Z}(e,\chi)$ be the Hopf algebra defined in Example 
\ref{zzex5.4} where $e$ is the generator of ${\mathbb Z}$ and where 
$\chi$ is determined by $\chi(e)=q$. Clearly there is a Hopf algebra 
map from $A_{\mathbb Z}(e,\chi)\to H$ sending $e$ to $e$ and $z$ to 
$z$. So we might assume that $H=A_{\mathbb Z}(e,\chi)$. Replacing 
$\Delta(z)$ by $z\otimes 1+e\otimes z$,
the assertion is independent of the coalgebra structure of 
$A_{\mathbb Z}(e,\chi)$, but its coalgebra structure is helpful in 
the proof below. Since all coefficients of $[z]^{\ell}$ are rational 
functions of $q$, we might further assume that the base field $k$ 
is ${\mathbb Q}(q)$ and $q$ is transcendental over ${\mathbb Q}$
(one can specialize $q$ to a particular value when necessary). 
In this setting, we have $[\ell]_q\neq 0$. 

By Lemma \ref{zzlem5.3}, we have  
\begin{equation}
\label{E5.6.1}\tag{E5.6.1}
\Delta(z^{\ell})=z^{\ell}\otimes 1+e^{\ell}\otimes z^{\ell}+
[\ell]_{q} [z]^{\ell}.
\end{equation}
Since $\Delta(z)$ commutes with $\Delta(z^{\ell})$, we obtain
$$\begin{aligned}
0&= [\Delta(z), \Delta(z^{\ell})]=[\Delta(z), z^{\ell}\otimes 1]
+[\Delta(z), e^{\ell}\otimes z^{\ell}]+[\ell]_{q}[\Delta(z),[z]^{\ell}]\\
&=(1-q^{\ell})ez^{\ell}\otimes z+(q^{\ell}-1)e^{\ell}z\otimes z^{\ell}
+[\ell]_{q}[\Delta(z),[z]^{\ell}]\\
\end{aligned}
$$
The assertion follows by dividing $[\ell]_{q}$. 

(2) By the argument given in the proof of part (1), we might assume
that $[\ell]_q\neq 0$. By \eqref{E5.6.1}, $[\ell]_{q} [z]^{\ell}
=-\partial^{1}_{e^{\ell},1}(z^\ell)$. So $[\ell]_{q} [z]^{\ell}
\in Z^2_{e^{\ell},1}(H)$ and the assertion follows
by dividing $[\ell]_{q}$.
\end{proof}


\section{Computations of $\PP^{1,2}_{g,h}(H)$ for HOEs}
\label{zzsec6}

In this section we try to understand $\PP^{i}_{g,h}(H)$ for $i=1,2$ 
for $H$ being a HOE. Information about $\PP^{1,2}_{g,h}(H)$ is one 
of the key ingredients in proving Theorems \ref{zzthm0.1} and 
\ref{zzthm0.2}.  We first give a general set-up.

\begin{hypothesis}
\label{zzhyp6.1}
Throughout this section, let $H:=K[z;\sigma,\delta]$ be a HOE of $K$
with $\Delta(z)=z\otimes 1+e\otimes z+z_0$ for some $z_0\in 
K^{\otimes 2}$. By Lemma \ref{zzlem5.1}(1),
$ze=q ez+\delta(e)$ for some $q\in k^{\times}$  and $\delta(e)\in K$. 
Let $\ell\geq 2$ denote the order of $q$ 
when $q$ is a root of unity, but not 1; and let $\ell=\infty$ if $q=1$ or 
$q$ is not a root of unity. When $q=1$, we further assume that 
${\rm{char}}\; k=0$.  
\end{hypothesis}

\begin{remark}
\label{zzrem6.2}
If $q\neq 1$,  then the equation $ze=q ez+\delta(e)$ implies
$e\neq 1$. When $\ell$ is finite, we have 
$$z e^{\ell}=q^{\ell} e^{\ell} z+
\delta(e^{\ell})=e^{\ell} z+
\delta(e^{\ell}),$$
which implies that $e^{\ell}\neq e$. It is possible that
$e^{\ell}=1$.
\end{remark}

The following lemma is easy to check (and also follows from the fact
that $\Cotor^{n-1}_C(C,kg)=\Cotor^{n-1}_C(kg,C)=0$ 
for all $n\geq 2$).  Recall from \eqref{E1.0.2} that
$$
D_n=\begin{cases}
0& n=0\\
\Delta &n=1\\
\sum_{i=0}^{n-1} (-1)^i Id^{\otimes i}\otimes \Delta
\otimes Id^{\otimes (n-i-1)} &n\geq 2
\end{cases}.
$$

\begin{lemma}
\label{zzlem6.3}
Let $C$ be a coalgebra and $g\in G(C)$.
\begin{enumerate}
\item[(1)]
Let $\Lambda^n_{g}=D_n+(-1)^n Id^{\otimes n}\otimes g$ be the map from
$C^{\otimes n}\to C^{\otimes (n+1)}$. Then 
$\ker \Lambda^n_g=\im \Lambda^{n-1}_g$ for all $n\geq 2$.
\item[(2)]
Let ${_g \Lambda^n}=g\otimes Id^{\otimes n}- D_n$ be the map from
$C^{\otimes n}\to C^{\otimes (n+1)}$. Then 
$\ker {_g \Lambda^n}=\im {_g \Lambda^{n-1}}$ for all $n\geq 2$.
\end{enumerate}
\end{lemma}

Recall that the definition of the degree of an element
in $H$ or in $H^{\otimes 2}$ are given before Lemma \ref{zzlem5.3}.

\begin{lemma}
\label{zzlem6.4}
Let $f:=\sum_{i=0}^{n} f_i z^i$ be in $H$ with $f_i\in K$ for all $i$, 
$f_n\neq 0$ and $n\geq 1$. 
\begin{enumerate}
\item[(1)]
If $\partial^1_{g,1}(f)\in K^{\otimes 2}$,
then, up to a scalar, either
\begin{enumerate}
\item[(1i)]
$f=z+f_0$ and $g=e$, or
\item[(1ii)]
$\ell$ is finite, and $f=z^{\ell}+\sum_{i=0}^{n-1}f_i z^i$ and 
$g=e^{\ell}$. 
\end{enumerate}
In the second case, $\{f_1,\cdots,f_{n-1}\}$ are unique, and 
$f_i=0$ for all $i=1,\cdots,n-1$ if $z_0=0$ and $\delta(e)=0$.
\item[(2)]
If $\deg (\partial^1_{g,1}(f))<\deg f$, then, up to a scalar, either
\begin{enumerate}
\item[(2i)]
$f=z+f_0$ and $g=e$, or
\item[(2ii)]
$\ell$ is finite, and $f=z^{\ell}+\sum_{i=0}^{n-1}f_i z^i$ and 
$g=e^{\ell}$. These $f_i$s may not be unique. 
\end{enumerate}
\end{enumerate}
\end{lemma}

\begin{proof} (1) If $f$ has degree $1$, then $f=f_1 z+f_0$
for some $f_1,f_0\in K$ and $f_1\neq 0$. Then $\partial^1_{g,1}(f) 
\in K^{\otimes 2}$ means that 
$$\Delta(f)=g\otimes f+f\otimes 1+w_0$$
for some $w_0\in K^{\otimes 2}$, or equivalently, 
\begin{equation}
\label{E6.4.1}
\tag{E6.4.1}
\Delta(f_1) (z\otimes 1+e\otimes z+z_0)+\Delta(f_0)=
g\otimes (f_1z+f_0)+(f_1z+f_0)\otimes 1+w_0.
\end{equation}
This implies that
\begin{equation}
\label{E6.4.2}
\tag{E6.4.2}
\Delta(f_1)=f_1\otimes 1, \quad
{\text{and}}\quad \Delta(f_1)(e\otimes 1)=g\otimes f_1.
\end{equation}
So $f_1\in k^{\times}$ and $e=g$. Since $f_1\neq 0$, we may assume $f_1=1$
after replacing $f$ by $f_1^{-1}f$.  This is case (1i).

If $\deg(f)=n\geq 2$, then $g\otimes f+f\otimes 1$ has no terms of 
the form $w z^{i}\otimes z^{j}$ where $w\in K^{\otimes 2}$ except 
for $(i,j)=(0,n)$ or $(n,0)$. Recall that $\partial^1_{g,1}(f)=
-\Delta(f)+g\otimes f+f\otimes 1$. The equation \eqref{E5.3.1} 
shows that if $\partial^1_{g,1}(f)\in K^{\otimes 2}$, then 
$$\Delta(f_n) {n\choose i}_q e^{n-i} z^i\otimes z^{n-i}=0$$ 
for all $0<i<n$. Since $f_n\neq 0$, this forces that ${n\choose i}_q=0$
and hence $q$ is a primitive $n$th root of unity 
\cite[Lemma 7.5]{GZ}. Therefore $q$ is a root of unity (and not 1)
and $n=\ell$. So $\ell$ is finite. 
Repeating the argument as before (see \eqref{E6.4.2}), one has that 
$f_{\ell}\in k^{\times}$ and $g=e^{\ell}$. So this is case (1ii).

To see the uniqueness of $f_i$ for $i=1,\cdots,n-1$, we consider
another $f'=z^{\ell}+\sum_{i=0}^{n-1} f'_i z^i$ such that
$\partial^1_{e^{\ell},1}(f')\in K^{\otimes 2}$. Then 
$\partial^1_{e^{\ell},1}(f-f')\in K^{\otimes 2}$. Since 
$e^{\ell}\neq e$, $f-f'$ is in case (1ii). Since $f-f'$ has
degree strictly less than $\ell$, it must be in $K$. 
So $f'_i=f_i$ for all $i=1,\cdots,n-1$.

If $z_0=0$ and $ze=qez$, it is easy to see that $z^{\ell}$
is in $Z^{1}_{e^{\ell},1}(H)$ (namely, $\partial^1_{e^{\ell},1}(z^{\ell})=0$). 
Therefore $f=z^{\ell}+f_0$ for some $f_0\in K$.
This finished the proof of (1).

The proof of (2) is similar and omitted. 
\end{proof}

To prove the main theorem of this section we need a few more 
lemmas, but we delay the proof of these lemmas.

Let $\phi: K\to H$ denote the inclusion map and let $\phi^n_{g,h}$
be the map $\PP^n_{g,h}(K)\to \PP^n_{g,h}(H)$, for all $n$ and 
$g,h\in G(K)$, naturally induced by $\phi$ [Lemma \ref{zzlem3.1}]. 
Since $\phi^n_{g,h}= {\mathcal S}_{h}\phi^n_{h^{-1}g,1}{\mathcal S}_{h^{-1}}$ 
[Lemma \ref{zzlem3.4}], we only study $\phi^n_{g,1}$. For an element 
$f$ in $Z^{n}_{g,1}(H)\subseteq H^{\otimes n}$, $\overline{f}$ denotes the 
cohomology class of $f$ in $\PP^n_{g,1}(H)$. 

The following theorem is very useful in computation.

\begin{theorem}
\label{zzthm6.5}
Assume Hypothesis \ref{zzhyp6.1}.
\begin{enumerate}
\item[(1)]
The map $\phi^1_{g,1}$ is injective, so it is viewed as an 
inclusion. 
\item[(2)]
The cokernel of $\phi^1_{g,1}$ is determined by the 
following.
\begin{enumerate}
\item[(2i)]
$\coker \phi^1_{g,1}=0$ for all $g\neq e, e^{\ell}$. 
\item[(2ii)]
$$\coker \phi^1_{e,1}=\begin{cases}
k\; \overline{z+f_0} & {\text{if}} \quad 
z_0=\partial^1_{e,1}(f_0)\in B^2_{e,1}(K),\\ 
0 &{\text{if}}\quad z_0\not\in B^2_{e,1}(K).
\end{cases}
$$
\item[(2iii)]
If $\ell$ is finite, then
$$\coker \phi^1_{e^{\ell},1}=\begin{cases}
k\; \overline{z^{\ell}} & {\text{if}} \quad 
z_0=0, \; {\text{and}} \; ze=qez,\\ 
{\text{having $\dim \leq 1$}} &{\text{otherwise}}.
\end{cases}
$$
If $\ell=\infty$, $e^{\ell}$ does not exist.
\end{enumerate}
\item[(3)]
The kernel of $\phi^2_{g,1}$ is determined by the following.
\begin{enumerate}
\item[(3i)]
$\ker \phi^2_{g,1}=0$ if $g\neq e, e^{\ell}$.
\item[(3ii)]
$$\ker \phi^2_{e,1}=\begin{cases}
0 & z_0\in B^2_{e,1}(K)\\
k\overline{z_0}& z_0\not\in B^2_{e,1}(K).
\end{cases}$$
\item[(3iii)]
When $\ell$ is finite, $\dim \ker \phi^2_{e^{\ell},1}\leq 1$. 
If, further, $z_0=0$ and $ze=qez$, then $\ker \phi^2_{e^{\ell},1}=0$.
\end{enumerate}
\item[(4)]
The cokernel of $\phi^2_{g,1}$ is determined partially by the following.
\begin{enumerate}
\item[(4a)]
If $q$ is either 1 or not a root of unity, for any $g$, a nonzero 
$\coker \phi^2_{g,1}$ is formed by the classes
of $f= a_0\otimes z+f_0$ where $a_0\in Z^1_{g,e}(K)$ and $f_0\in 
K^{\otimes 2}$ satisfy the following conditions
\begin{enumerate}
\item[(4i)] 
$\overline{a_0}\neq 0$ in $\PP^1_{g,e}(K)$ and 
$\overline{a_0\otimes z_0}=0$ in $\PP^3_{g,1}(K)$, 
and 
\item[(4ii)]
$\partial^2_{g,1}(f_0)=-a_0\otimes z_0$.
\end{enumerate}
\item[(4b)]
Suppose that $q$ is a primitive $\ell$th root of unity {\rm{(}}where
$\ell\geq 2${\rm{)}} and that $z_0=0$ and $ze=qez$. Then,
considering $\coker \phi^2_{g,1}$ as an ${\mathbb N}$-graded vector 
space, we have the following.
\begin{enumerate}
\item[(4iii)] 
$(\coker \phi^2_{g,1})_1\cong \PP^1_{g,e}(K)\otimes k z$ 
for all $g\in G(H)$.
\item[(4iv)]
$$\coker (\phi^2_{g,1})_{\ell}\cong
\begin{cases}\PP^1_{g,e^{\ell}}(K)\otimes z^{\ell}\oplus k
{[z]^{\ell}} & {\text{if $g=e^{\ell}$,}}\\
\PP^1_{g,e^{\ell}}(K)\otimes z^{\ell} & {\text{if $g\neq e^{\ell}$.}}
\end{cases}$$
\item[(4v)]
$(\coker \phi^2_{g,1})_{\ell+1}\cong 
\begin{cases} k {e^{\ell}z\otimes z^{\ell}} & g= e^{\ell+1},\\
0& g\neq e^{\ell+1}.
\end{cases}$
\item[(4vi)]
$(\coker \phi^2_{g,1})_{w}=0$ for all $g$ and all $w\neq 1, 
\ell, \ell+1$.
\end{enumerate}
\end{enumerate}
\end{enumerate}
\end{theorem}

\begin{proof} Recall that the definition of the degree of an element
in $H$ or in $H^{\otimes 2}$ is given before Lemma \ref{zzlem5.3}.
For our convenience, sometimes we do not distinguish an element 
$x\in H^{\otimes n}$ from its cohomology class $\overline{x}$ in 
$\PP^n_{g,1}(H)$ (for appropriate $g$).

(1) As a consequence of Lemma \ref{zzlem5.1}(3), 
$B^1_{g,h}(H)=k(g-h)=B^1_{g,h}(K)$. By definition, 
$\PP^1_{g,1}(K)\subseteq \PP^1_{g,1}(H)$, so $\phi^1_{g,1}$ is
injective. 

(2) We would like to see exactly when there are elements in 
$\PP^1_{g,1}(H)\setminus \PP^1_{g,1}(K)$.

Let $f\in Z^1_{g,1}(H)$ and we write $f=\sum_{i=0}^n f_i z^n$ where
$f_i\in K$.
 
If $\deg(f)=0$, then $f\in Z^1_{g,1}(K)$ and $\overline{f}
\in \PP^1_{g,1}(K)$. This means that there is no element of 
degree zero in $\PP^1_{g,1}(H)\setminus \PP^1_{g,1}(K)$. 

Next assume that $\deg f=1$. Since $\partial^1_{g,1}(f)=0\in K^{\otimes 2}$,
Lemma \ref{zzlem6.4}(1) tells us that if $g\neq e$, then there 
is no element $f$ of degree 1 in  $Z^1_{g,1}(H)$. So $\phi^1_{g,1}$ 
is surjective. The other case is when $g=e$ and $f=z+f_0$. We don't
need to consider the case in Lemma \ref{zzlem6.4}(1ii) since $\deg f=1$.
Then \eqref{E6.4.1} implies that
$z_0=\partial^{1}_{e,1}(f_0)$. So we have
\begin{enumerate}
\item[(a)]
If $z_0\not\in B^2_{e,1}(K)$, then there is no element $f$ of degree 1 
in $Z^1_{e,1}(H)$.
\item[(b)]
If $z_0\in B^2_{e,1}(K)$ and write $z_0=\partial^1_{e,1}(f_0)$,
then, up to a scalar, $f:=z+f_0$ is the only element of degree 1 in 
the quotient $\PP^1_{e,1}(H)/\PP^1_{e,1}(K)$. 
\end{enumerate}
Combining (a) and (b) we obtain the formula in (2ii).

If $\deg f=n\geq 2$, by Lemma \ref{zzlem6.4}(1ii), 
$f=z^{\ell}+\sum_{i=0}^{n-1}f_i z^i$. So 
$\coker \partial^1_{e^{\ell},1}$ is generated by this
$f$ (which is unique in $\PP^1_{e^{\ell},1}(H)/\PP^1_{e^{\ell},1}(K)$
by Lemma \ref{zzlem6.4}(1)).  This
shows that $\dim \coker \partial^1_{e^{\ell},1}\leq 1$.

If $z_0=0$ and $ze=qez$, it is easy to see that $z^{\ell}$
is in $Z^{1}_{e^{\ell},1}(H)$ by the proof of
Lemma \ref{zzlem6.4}(1). Therefore (2iii) follows.

(3) Recall that the map $\phi^2_{g,1}$ is from 
$$\PP^2_{g,1}(K):=
Z^{2}_{g,1}(K)/B^{2}_{g,1}(K)\longrightarrow \PP^2_{g,1}(H):=
Z^{2}_{g,1}(H)/B^{2}_{g,1}(H),$$ 
and hence $\ker \phi^2_{g,1}$ is isomorphic to 
$(Z^{2}_{g,1}(K)\cap B^{2}_{g,1}(H))/B^{2}_{g,1}(K)$.

By Lemma \ref{zzlem6.4}(1), if $\partial^{1}_{g,1}(f)
\in (Z^{2}_{g,1}(K)\cap B^{2}_{g,1}(H))\setminus 
B^{2}_{g,1}(K)=:L$, then, up to a scalar, $f=z+f_0$ and $g=e$ or 
$f=z^{\ell}+\sum_{i=0}^{n-1}f_i z^i$ and $g=e^{\ell}$. This implies 
that $\dim \ker \phi^2_{g,1}\leq 1$, and 
$\ker \phi^2_{g,1}=0$ if $g\neq e, e^{\ell}$. So we finish
the proof of (3i).  

If $g=e$ (and so $e^{\ell}\neq e$), then the only possible element 
in $L$ is $\partial^1_{e,1}(z+f_0)$. By definition, 
$\partial^1_{e,1}(z+f_0)=\partial^1_{e,1}(f_0)-z_0$.
If $z_0\in B^2_{e,1}(K)$, then $L$ is empty and $\ker \phi^2_{e,1}=0$.
If $z_0\not\in B^2_{e,1}(K)$, then $\overline{z_0}\neq 0$ in 
$\ker \partial^{2}_{e,1}$. This is (3ii). 

The last case is when $\ell$ is finite and $n=\ell$. By Lemma 
\ref{zzlem6.4}(1), we have $g=e^{\ell}$ and 
$f=z^{\ell}+\sum_{i=0}^{n-1}f_i z^i$.
Again $\ker \partial^{2}_{e^{\ell},1}$ is generated by the class of 
$\partial^{1}_{e^{\ell},1}(f)$.
So $\dim \ker \partial^{2}_{e^{\ell},1}\leq 1$. 
If, further $z_0=0$ and $ze=qez$, 
then $f=z^{\ell}+f_0$ [Lemma \ref{zzlem6.4}(1)]. In this case 
$\overline{\partial^{1}_{e^{\ell},1}(f)}=0$ in $\PP^2_{e^{\ell},1}(H)$ 
and the kernel is zero. So we prove (3iii).

(4) For part (4) we need  to understand when the map 
$$Z^{2}_{g,1}(K)+B^2_{g,1}(H) \to Z^{2}_{g,1}(H)$$ 
is surjective.
We have to use a few more lemmas which will be proved later.

Let $f$ be an element in $Z^2_{g,1}(H)$ for some 
$g$. We will use induction on $\deg f$ to show that 
$f\in B^2_{g,1}(H)+Z^2_{g,1}(K)$ except for a few interesting cases.

If $f$ has degree $0$, then we have $f\in Z^2_{g,1}(K)$. So
the class $\overline{f}$ is in the image of the map $\PP^2_{g,1}(K)\to
\PP^2_{g,1}(H)$. 

Now we assume that $\deg(f)=w>0$ (and $f\neq 0$).
Recall that $f\in Z^2_{g,1}(H)$ means that $f$ satisfies the 
equation
\begin{equation}
\label{E6.5.1}
f\otimes 1+(\Delta \otimes 1)f-g\otimes f-(1\otimes \Delta) f=0.
\tag{E6.5.1}
\end{equation}

Write 
$$f=\sum_{n=0}^w A_{n}z^n\otimes z^{w-n}+ldt$$
where $A_n\in K^{\otimes 2}$ for all $n$. As always $ldt$ denotes
some element in $H^{\otimes d}$ of degree less than 
the degree of the leading term (and $w$ in this case).
Expanding and re-arranging terms in equation \eqref{E6.5.1} we obtain 
the following 
\begin{align}
\label{E6.5.2}\tag{E6.5.2}
&\sum(A_n\otimes 1)(z^n\otimes z^{w-n}\otimes 1)+
\qquad\qquad\qquad\qquad\qquad\qquad\\
&\qquad\qquad \sum (\Delta\otimes 1) A_n 
(\sum {n \choose m}_q e^{(n-m)} z^m\otimes z^{n-m}\otimes z^{w-n})
+ldt\notag\\
&=g\otimes \sum A_n z^n\otimes z^{w-n}+
\sum (1\otimes \Delta)A_n (z^n\otimes (\sum {w-n\choose m}_q 
e^{m} z^{w-n-m}\otimes z^m))\notag\\
&\qquad\qquad +ldt\notag
\end{align}
Considering the coefficients of term $z^w\otimes 1\otimes 1$, 
we have
$$A_w\otimes 1+(\Delta\otimes 1)A_w=(1\otimes \Delta) A_w$$
which means that $A_w\in \ker \Lambda^2_1$ where $\Lambda^2_1$ is
defined in Lemma \ref{zzlem6.3}(1). By Lemma \ref{zzlem6.3}(1),
$A_w\in \im \Lambda^1_1$, or equivalently, $A_w=-b\otimes 1+\Delta(b)$
for some $b\in K$.  A direct computation shows that
$$\partial^1_{g,1}(bz^w)=(b\otimes 1-\Delta(b)) (z^w\otimes 1)+
\sum_{i=0}^{w-1} f'_i z^i\otimes z^{w-i}+
ldt$$
for some $f'_i\in K^{\otimes 2}$.
After replacing $f$ by $f+\partial^1_{g,1}(bz^w)$ we may
assume that $A_w=0$. Some other computation is given in Lemmas 
\ref{zzlem6.6} -- \ref{zzlem6.8} below. For example, by Lemma 
\ref{zzlem6.6}(1), $A_n=a_n\otimes 1$ for some $a_n\in K$ for all $n$, 
and by Lemma \ref{zzlem6.6}(2), $a_n=b_n e^{w-n}$ for $b_n\in k$ 
for all $1\leq n \leq w-1$. In summary we might assume that
$$f=a_0\otimes z^w+\sum_{n=1}^{w-1} b_n e^{w-n}z^n \otimes z^{w-n} +ldt.$$

Case (4a): Suppose that $q$ is either 1 or not a root of unity.

If $\deg(f)=1$, then $f=A_0(1\otimes z)+f_0=a_0\otimes z+f_0$.
Equation \eqref{E6.5.1} implies that 
$$\begin{aligned}
(a_0\otimes z+f_0)\otimes 1+&\Delta(a_0)\otimes z+
(\Delta\otimes 1)f_0\\
&=g\otimes (a_0\otimes z+f_0)+a_0\otimes
(z\otimes 1+e\otimes z+z_0)+(1\otimes \Delta)(f_0)
\end{aligned}$$
which gives rise to the following
$$\begin{aligned}
\Delta(a_0)&=g\otimes a_0+a_0\otimes e,\\
\partial^2_{g,1}(f_0)&=-a_0\otimes z_0.
\end{aligned}
$$
If $a_0\in \im \partial^0_{g,e}$, then $a_0=\alpha(g-e)$, and by 
replacing $f$ by $f-\partial^1_{g,1}(\alpha z)$, we may further assume
$a_0=0$, which goes back to the case $\deg(f)=0$. As a consequence,
$\overline{f}$ is in the image of $\phi^2_{g,1}$. Conversely, 
if $\overline{f}$ is in the image of $\phi^2_{g,1}$, then 
$f=\partial^1_{g,1}(x)+y$ where $x\in H$ and $y\in Z^{2}_{g,1}(K)$.
By Lemma \ref{zzlem6.4}(2), $x=x_1 z+x_0$. Now the equation
$$a_0\otimes z+f_0=f=\partial^1_{g,1}(x)+y
=(x_1 z+x_0)\otimes 1+g\otimes (x_1 z+x_0)
-\Delta(x_1z+x_0)+y$$
implies that $x_1\in k$ and $a_0=x_1(g-e)\in \im \partial^1_{g,e}$. 
Combining these two statements, we obtain that the class of 
$f:=a_0\otimes z+f_0$ is in $\PP^2_{g,1}(H)\setminus \im \PP^2_{g,1}(K)$ 
for some $a_0, f_0$ if and only if $\overline{a}_0\in \PP^1_{g,e}(K)$
is nonzero and $\overline{a_0\otimes z_0}=0$ in $\PP^3_{g,1}(K)$.
Restricted to elements $f$ of degree at most 1, $\coker \phi^2_{g,1}$ 
is generated by the classes of $f:=a_0\otimes z+f_0$ such that 
$\overline{a_0}\neq 0$ in $\PP^1_{g,e}(K)$ and $\partial^{2}_{g,1}(f_0)
=-a_0\otimes z_0$.

Next we assume that $\deg f=w\geq 2$.
Since  ${w\choose w-n}_q\neq 0$ for all $0\leq n\leq w$, by Lemma 
\ref{zzlem6.6}(3), $a_0=b(g-e^w)$. By Lemma \ref{zzlem6.7}(1),
we may assume $A_0=A_w=0$ up to a coboundary.
 By Lemma \ref{zzlem6.6}(3), $g=e^w$. 
By Lemma \ref{zzlem6.7}(2), after adding a coboundary, $b_{w-1}=0$.
Then \eqref{E6.6.5} implies that $b_n=0$ 
for all $n$. Hence $f=f'+f''$ where $f'\in B^2_{g,1}(H)$ and $\deg f''<w$. 
Applying the argument to $f''$ and by induction, one might assume 
that $\deg(f'')\leq 1$. Therefore the assertion follows by the 
previous paragraph.

Case (4b): Suppose that $q$ is a primitive $\ell$th root of unity and
$\ell\geq 2$. As hypotheses, we further assume that $z_0=0$ and 
that $ze=qez$ 
(or $\delta(e)=0$). If $\delta(K)=0$, then $H$ is a graded Hopf algebra
with $\deg (a)=0$ for all $a\in K$ and $\deg(z)=1$. If $\delta(K)\neq 0$,
$H$ is not a graded Hopf algebra, but it is still a graded coalgebra
with $\deg (az^i)=i$ for all $i\in{\mathbb N}$ and all $0\neq a\in K$.
Therefore $\PP^2_{g,1}(H)$ is ${\mathbb N}$-graded. In this setting,
the degree $w$ component of a graded vector space $V$ is denoted by
$V_w$. 

It is trivial that $(\coker \phi^2_{g,1})_0=0$ since 
elements of degree 0 are in $K^{\otimes 2}$. 
Now we consider $(\coker \phi^2_{g,1})_1$.
Let $f\in H^{\otimes 2}$ be of degree 1.
By the argument at the beginning of Case (4a),
we have, up to a scalar (and modulo $B^2_{g,1}(H)$), 
$$\begin{aligned}
f&= a_0 \otimes z, \;\; {\text{where $a_0\in K$, and}},\\
\Delta(a_0)&=g\otimes a_0+a_0\otimes e.\\
\end{aligned}
$$
If $a_0\in \im \partial^0_{g,e}$, then $a_0=\alpha(g-e)$, and
$f=\partial^1_{g,1}(\alpha z)\in B^2_{g,1}(H)$. Conversely, 
we assume that $\overline{f}$ is in the image of $\phi^1_{g,1}$.
Write $f=\partial^1_{g,1}(x)$ where $x\in H$. 
Since $\deg (f)=1$, $\deg(x)= 1$.  Then the argument 
of Case (4a) is valid, and we obtain that $f=a_0\otimes z$ 
where $a_0\in \im \partial^0_{g,e}$. 
This means that the class of $f:=a_0\otimes z$ 
is in $\PP^2_{g,1}(H)\setminus \im \PP^2_{g,1}(K)$ 
if and only if $\overline{a}_0\in \PP^1_{g,e}(K)$ is nonzero.
Therefore $(\coker \phi^2_{g,1})_1\cong \PP^1_{g,e}(K)$, 
which is (4iii).

Next consider degree $w\geq 2$.  Recall that $[z]^{\ell}$ denotes 
the element 
$$\sum_{i=1}^{\ell-1} \frac{[\ell-1]_q !}{[i]_q ! [\ell-i]_q !} 
e^{\ell -i}z^i\otimes z^{\ell-i}.$$ By the discussion after 
\eqref{E6.5.2}, we may assume
that $A_w=0$. We divide our discussion into four subcases.
We might also omit the words ``modulo $B^2_{g,1}(H)$"
when it is understood.

Case (4b.1): $w=\ell$. 
By Lemma \ref{zzlem6.8}(1), $f=a_0\otimes z^{\ell}+b_1 [z]^{\ell}$.
By \eqref{E6.6.4}, $a_0\in Z^1_{g,e^{\ell}}(K)$. 
If $b_1\neq 0$, then by \eqref{E6.6.3}, $g=e^w=e^{\ell}$. Therefore
we have the following:
\begin{enumerate}
\item
If $g=e^{\ell}$, then there is an injective map
\begin{equation}
\label{E6.5.3}\tag{E6.5.3}
\coker (\phi^2_{e^{\ell},1})_{\ell}\to 
(\PP^1_{g,e^{\ell}}(K)\otimes z^{\ell})\oplus k [z]^{\ell}
\end{equation}
sending $\overline{f}:=\overline{a_0\otimes z^{\ell}+b_1 [z]^{\ell}}$ to 
$\overline{a_0}\otimes z^{\ell}+ b_1[z]^{\ell}$. The injectivity follows
easily.
\item
If $g\neq e^{\ell}$, then $b_1=0$ and 
then there is an injective map
\begin{equation}
\label{E6.5.4}\tag{E6.5.4}
\coker (\phi^2_{g,1})_{\ell}\to 
\PP^1_{g,e^{\ell}}(K)\otimes z^{\ell}
\end{equation}
sending $\overline{f}:=\overline{a_0\otimes z^{\ell}}$ to 
$\overline{a_0}\otimes z^{\ell}$. The injectivity follows
from Lemma \ref{zzlem6.8}(1) (details are easy to provide).
\end{enumerate}
We want to remark that (a) and (b) hold even without the
assumption that $z_0=0$ and $\delta(e)=0$. When $z_0=0$
and $\delta(e)=0$, the surjectivity of \eqref{E6.5.3}-\eqref{E6.5.4}
follows from Lemma \ref{zzlem6.8}(2,3). Therefore, we have 
$$\coker (\phi^2_{g,1})_{\ell}\cong
\begin{cases}\PP^1_{g,e^{\ell}}(K)\otimes z^{\ell}\oplus k\overline{
[z]^{\ell}} & {\text{if $g=e^{\ell}$,}}\\
\PP^1_{g,e^{\ell}}(K)\otimes z^{\ell} & {\text{if $g\neq e^{\ell}$.}}
\end{cases}$$

For the rest of the proof we assume that $w\neq \ell$. By 
Lemma \ref{zzlem6.7}(3), $a_0=b(g-e^w)$ for some $b\in k$ 
and by Lemma \ref{zzlem6.7}(1), we may further assume that $A_0=A_w=0$. 
Since we are mostly considering $\partial_{g,1}^1$ for
$g=e^w$, we will automatically assume that $g=e^w$ in appropriate cases.

Case (4b.2): $w=s\ell$ for some $s\geq 2$. If $n>\ell$ and $n$ is not a 
multiple of $\ell$, then by \eqref{E6.6.5}
$${n\choose \ell}_q b_n={s\ell-\ell\choose s\ell-n}_qb_{\ell}=0$$
and ${n\choose \ell}_q\neq 0$. Hence $b_n=0$. If $0<n<\ell$,
$${s\ell-n\choose s\ell-\ell}_q b_n={\ell\choose n}_q b_{\ell}=0$$
and ${s\ell-n\choose s\ell-\ell}_q\neq 0$. Hence $b_n=0$. Using 
\eqref{E5.3.1}, one sees that 
$\partial^1_{e^{s\ell},1}((z^{\ell})^s)=\sum_{i=1}^{s-1}
{s\choose i} e^{(s-i)\ell}(z^{\ell})^i\otimes (z^{\ell})^{s-i}$.
By replacing $f$ by $f+\partial^1_{e^{s\ell},1}(\frac{b_{\ell}}{s}
(z^{\ell})^s)$, one has $b_{\ell}=0$. For any $a>1$, 
\eqref{E6.6.5} and Lemma \ref{zzlem5.2}(2) induce  that
$${a\choose 1} b_{a\ell}={a\ell \choose \ell}_q b_{a\ell}
={s\ell-\ell\choose s\ell-a\ell}_q b_{\ell}=
{s-1\choose s-a} b_{\ell}=0$$
which implies that $b_{a\ell}=0$. Therefore $b_n=0$ for
all $n$, or equivalently, $f=0$. Thus $(\coker \phi^2_{g,1})_{s\ell}=0$
for all $s\geq 2$.

Case (4b.3): $w=\ell+1$. Applying \eqref{E6.6.5} to $n=w-1=\ell$, one
has $b_m=0$ for all $m=2,\cdots,w-2$. By Lemma \ref{zzlem6.7}(2),
one may assume further $b_{w-1}=0$. In other words, 
$f=b_1 e^{\ell}z\otimes z^{\ell}$. Thus
$(\coker \phi^2_{g,1})_{\ell+1}=0$ if $g\neq e^{\ell+1}$. If 
$g=e^{\ell+1}$, then there is an injective map 
$$\coker (\phi^2_{e^{\ell+1},1})_{\ell+1}\to k {e^{\ell}z\otimes z^{\ell}}$$
sending $\overline{b_1 e^{\ell}z\otimes z^{\ell}}$ to $b_1 
{e^{\ell}z\otimes z^{\ell}}$. This holds even without
the hypotheses that $z_0=0$ an $\delta(e)=0$. When 
$z_0=0$ and $\delta(e)=0$, we have 
$(\coker \phi^2_{e^{\ell+1},1})_{\ell+1}\cong 
k{e^{\ell}z\otimes z^{\ell}}$ by Lemma \ref{zzlem6.8}(4). 

Case (4b.4): $w\neq \ell+1, s\ell$ for any $s\geq 1$ and $w\geq 2$. By
Lemma \ref{zzlem6.7}(2) we may assume that $b_{w-1}=0$. Applying
\eqref{E6.6.5} to $(w-1,m)$ we have
$$0={w-1\choose m}_q b_{w-1}={w-m\choose 1}_q b_m.$$
For any $m$ such that $\ell\nmid w-m$, ${w-m\choose 1}_q \neq 0$,
and hence $b_m=0$. 

If $\ell\mid w-m$, then $m=w-a\ell$ for some $a\geq 0$.
Since $\ell\nmid w-\ell$, $b_{\ell}=0$ (if $w>\ell$). 
If $\ell\leq w-\ell$, applying
\eqref{E6.6.5}, we have 
$${w-\ell\choose \ell}_q b_{w-\ell}={w-\ell\choose \ell}_q b_l=0.$$
Since ${w-\ell\choose \ell}_q \neq 0$, $b_{w-\ell}=0$. 
If $\ell>w-\ell$, applying
\eqref{E6.6.5}, we have 
$${\ell\choose w-\ell}_q b_{w-\ell}={\ell\choose w-\ell}_q b_l=0.$$
Since ${\ell\choose w-\ell}_q \neq 0$, $b_{w-\ell}=0$. 
Applying \eqref{E6.6.5} to $(m,m-\ell)$, we have 
$${m\choose m-\ell}_q b_{m}={w-m+\ell\choose w-m}_q b_{m-\ell}.$$
Since both ${m\choose m-\ell}_q$ and ${w-m+\ell\choose w-m}_q$ are
nonzero, $b_{m}=0$ if and only if $b_{m-\ell}=0$. Since $b_{w-\ell}=0$,
we have $b_{w-a\ell}=0$ for all $a$. Thus we have shown that
$b_{m}=0$ for all $m$ and therefore $f=0$. Or equivalently,
$(\coker \phi^2_{g,1})_{w}=0$ for all $w\neq \ell+1, s\ell$ for 
any $s\geq 1$ and $w\geq 2$

The assertions in (4b) follows by combining all these cases.
\end{proof}

\begin{lemma}
\label{zzlem6.6}
Retain the notation as in the proof of Theorem \ref{zzthm6.5}.
Assume \eqref{E6.5.2} and $A_w=0$. Then the following hold. 
\begin{enumerate}
\item[(1)]
\begin{equation}
\label{E6.6.1}\tag{E6.6.1}
A_n=a_n\otimes 1 
\end{equation}
where $a_n\in K$ for all $n\geq 0$.
\item[(2)]
\begin{equation}
\label{E6.6.2}\tag{E6.6.2}
a_n=b_n e^{w-n} 
\end{equation}
where $b_n\in k$ for all $n\geq 1$.
\item[(3)]
\begin{equation}
\label{E6.6.3}\tag{E6.6.3}
b_n e^w=b_ng+a_0{w\choose w-n}_{q}
\end{equation}
for all $n\geq 1$.
\item[(4)]
\begin{equation}
\label{E6.6.4}\tag{E6.6.4}
\Delta(a_0)=g\otimes a_0+a_0\otimes e^w.
\end{equation}
\item[(5)]
\begin{equation}
\label{E6.6.5}\tag{E6.6.5}
{n \choose m}_q b_n={w-m\choose w-n}_q b_m
\end{equation}
for all $1\leq m\leq n\leq w-1$.
\end{enumerate}
\end{lemma}

\begin{proof} (1) 
Considering the coefficients of the term $z^{n}\otimes z^{w-n}\otimes 1$
in \eqref{E6.5.2} and using the fact that $A_w=0$, we obtain that
$$(1\otimes \Delta) A_n=A_n\otimes 1$$
for all $n$. This implies that $A_n=a_n\otimes 1$ for some $a_n\in K$,
for all $n$. 

(2) Considering the coefficients of the term $z^{n}\otimes 1\otimes z^{w-n}$
for $0<n<w$ in \eqref{E6.5.2} we obtain that
$$(\Delta\otimes 1)A_n = (1\otimes \Delta)A_n 
(1\otimes e^{w-n}\otimes 1)$$
for all $0<n<w$, which implies that
$a_n=\epsilon(a_n) e^{w-n}=b_n e^{w-n}$ for all $0<n<w$.

(3) Considering the coefficients of the term $1\otimes z^{n}\otimes z^{w-n}$
for $n<w$ in \eqref{E6.5.2} we obtain that
$$(\Delta\otimes 1)A_n (e^n\otimes 1\otimes 1)=
g\otimes A_n+{w\choose w-n}_q
(1\otimes \Delta)A_0 (1\otimes e^{w-n}\otimes 1)$$
for all $n<w$. Applying $1\otimes \epsilon\otimes 1$
to the above equation and using the fact that $A_n=a_n\otimes 1$ 
we have
$$a_n e^n\otimes 1=g\otimes \epsilon(a_n)+
{w\choose w-n}_q a_0\otimes 1,$$
for all $n\geq 0$, 
which is equivalent to \eqref{E6.6.3} when $n\geq 1$.

(4) Considering the coefficients of the term $1\otimes 1\otimes z^w$
in \eqref{E6.5.2}, we have 
$$g\otimes A_0+(1\otimes \Delta)A_0 (1\otimes e^w\otimes 1)=
(\Delta\otimes 1)A_0.$$
The above equation together with $A_0=a_0\otimes 1$ implies 
\eqref{E6.6.4}.

(5) 
Considering the coefficients of the term $z^{m}\otimes z^{n-m} 
\otimes z^{w-n}$ for $1\leq m\leq n\leq w-1$ in the equation 
\eqref{E6.5.2} we obtain 
$$(\Delta\otimes 1)A_n {n \choose m}_q(e^{n-m}\otimes 1\otimes 1)
=(1\otimes \Delta)A_m {w-m\choose w-n}_q (1\otimes e^{w-n}\otimes 1)
$$
which is equivalent to \eqref{E6.6.5}.
\end{proof}

\begin{lemma}
\label{zzlem6.7} Retain the hypotheses of Lemma \ref{zzlem6.6}.
In particular, $A_w=0$. 
\begin{enumerate}
\item[(1)]
If $a_0= b(g-e^w)$ for some $b\in k$, then, after replacing
$f$ by $f-\partial_{g,1}(b z^w)$, we may assume that $A_0=A_w=0$.
\item[(2)]
Suppose $A_0=0$.
If $[w]_q\neq 0$ and $g=e^{w}$, then, after replacing
$f$ by $f+\partial_{g,1}(\frac{b_{w-1}}{[w]_q} z^w)$, we may assume 
that $A_0=A_w=A_{w-1}=0$.
\item[(3)]
Suppose $q$ is a primitive $\ell$th root of unity and $w\neq \ell$.
Then $a_0=b(g-e^w)$ for some $b\in k$.
\end{enumerate}
\end{lemma}

\begin{proof} (1) Let $f'=f-\partial^1_{g,1}(b z^w)$ and let $A'_n$
be the corresponding $A_n$ for this new element. An easy 
computation shows that
$$\partial^1_{g,1}(b z^w)=0(z^w\otimes 1)+b(g-e^w)\otimes z^w+
\sum_{i=1}^{w-1}g' e^{w-i}z^{i} \otimes z^{w-i}+ldt$$
which implies that $A'_w=A_w=0$ and $A'_0=A_0-b(g-e^w)\otimes 1=0$.

(2) Let $h=\frac{b_{w-1}}{[w]_q} z^w$. Then 
$$\begin{aligned}
\partial^1_{e^w,1}(h)&=-\Delta(h)+h\otimes 1+e^w\otimes h\\
&=-\sum_{n=1}^{w-1} 
\frac{b_{w-1}}{[w]_q}{w\choose n}_q e^{w-n} z^n\otimes z^{w-n}
+ldt\\
&=-b_{w-1}e z^{w-1}\otimes z-\sum_{n=1}^{w-2} 
\frac{b_{w-1}}{[w]_q}{w\choose n}_q e^{w-n} z^n\otimes z^{w-n}\\
& \qquad\qquad\qquad +ldt.\\
\end{aligned}
$$
Replacing $f$ by $f+\partial^1_{e^w,1}(h)$, one sees that $b_{w-1}=0$. 

(3) Pick $1\leq n\leq w-1$ such that ${w\choose w-n}_{q}\neq 0$ 
(this is possible when $w\neq \ell$). Then the assertion follows
from \eqref{E6.6.3}.
\end{proof}

\begin{lemma}
\label{zzlem6.8} Retain the hypotheses of Lemma \ref{zzlem6.6}.
Assume that $q$ is a primitive $\ell$th root of unity. 
\begin{enumerate}
\item[(1)]
If $f:=a_0\otimes z^{\ell}
+\sum_{i=1}^{\ell-1} b_i  e^{\ell -i} z^i \otimes z^{\ell-i}+ldt$
is in $Z^2_{g,1}(H)$ and $b_{i_0}\neq 0$ for some $0<i_0<\ell$, 
then $g=e^{\ell}$ and $b_i\neq 0$ for all $0<i<\ell$. In fact, 
$$b_i=b_1\frac{[\ell-1]_q !}{[i]_q ! [\ell-i]_q !}$$
for all $0<i<\ell$. 
\end{enumerate}
In parts {\rm{(}}2,3,4{\rm{)}} below,
assume that $z_0=0$ {\rm{(}}namely, $\Delta(z)=z\otimes 1+
e\otimes z${\rm{)}} and 
$\delta(e)=0$ {\rm{(}}namely, $ze=qez${\rm{)}}.
\begin{enumerate}
\item[(2)]
The class of 
$$[z]^{\ell}:=\sum_{i=1}^{\ell-1} \frac{[\ell-1]_q !}{[i]_q ! [\ell-i]_q !} 
e^{\ell -i}z^i\otimes z^{\ell-i}$$
is nonzero  in $\PP^2_{e^{\ell},1}(H)$.
\item[(3)]
If $a_0\in Z^1_{g,e^{\ell}}(K)$, then 
$a_0\otimes z^{\ell}\in Z^2_{g,1}(H)$. 
\item[(4)]
$\overline{e^{\ell}z\otimes z^{\ell}}$ is nonzero  in 
$\PP^2_{e^{\ell+1},1}(H)$
\end{enumerate}
\end{lemma}

\begin{proof} (1) By \eqref{E6.6.3}, $b_{i_0}e^{\ell}=b_{i_0} g+0$
as ${\ell\choose i_0}=0$; so $g=e^{\ell}$. 
Since both ${n \choose m}_q$ and ${w-m\choose w-n}_q$ 
are nonzero for $1\leq m\leq n\leq w-1$, the equation \eqref{E6.6.5}
implies that $b_i\neq 0$ for all $i>0$. The last assertion follows from 
\eqref{E6.6.5}. 

(2) By Lemma \ref{zzlem5.6}(2), $[z]^{\ell}\in Z^2_{e^{\ell},1}(H)$. If remains
to show that $[z]^{\ell}$ is not of the form $\partial^1_{e^\ell,1}(F)$
for any $F\in H$. If there were $F$ such that $\partial^1_{e^\ell,1}(F)=
[z]^{\ell}$, then $\deg(F)=\ell$. Write $F=F_0 z^{\ell}$ for $F_0\in K$.
Then the fact $\Delta(z^{\ell})=z^{\ell}\otimes 1+e^{\ell}\otimes z^{\ell}$
implies that $\partial^1_{e^{\ell},1}(F)\neq [z]^{\ell}$. Therefore the class
of $[z]^{\ell}$ is nonzero. 

(3,4) These follow from a direct computation and the fact that
$z^{\ell}\in Z^1_{e^{\ell},1}(H)$.
\end{proof}

Now we are able to compute the primitive cohomology for two HOEs
$C_G(e,\chi,\tau)$ [Example \ref{zzex5.5}] and $A_G(e,\chi)$
[Example \ref{zzex5.4}].

\begin{corollary}
\label{zzcor6.9}
Let $H$ be $C_G(e,\chi,\tau)$ and let $g,h\in G(H)=:G$. 
\begin{enumerate}
\item[(1)]
$\PP^1_{g,h}(H)=k \; \overline{hz}$ if $gh^{-1}=e$
and $\PP^1_{g,h}(H)=0$ otherwise.
\item[(2)]
$\PCdim H=1$.
\end{enumerate}
\end{corollary}

\begin{proof} Let $K=kG=C_0$. Since $K$ is a group algebra,
$\PP^i_{g,h}(K)=0$ for all $g,h$ and all $i>0$ [Example 
\ref{zzex4.3}]. By Example 
\ref{zzex5.5}, $H$ is a HOE $K[z;\sigma_{\chi},\delta_{\tau}]$ 
with $q=1$ and $z_0=0$ (by the notation of Theorem \ref{zzthm6.5}).
By Theorem \ref{zzthm6.5}, the $\ker \phi^i_{g,1}$ and $\coker 
\phi^i_{g,1}$, for $i=1,2$, are zero except for 
$\coker \phi^1_{e,1}=k\overline{z}$. The assertion in (1) follows. 
For part (2), we see that $\PP^2_{g,1}(H)=0$ for all $g\in G$.
It is clear that $\PCdim H\ge 1$. 
The assertion in (2) follows Proposition \ref{zzpro3.7}(2).
\end{proof}

\begin{corollary}
\label{zzcor6.10}
Let $H$ be $A_G(e,\chi)$ and assume that $\chi(e)$ is either 1 or not a 
root of unity. 
\begin{enumerate}
\item[(1)]
$\PP^1_{g,h}(H)=k \; \overline{hz}$ if $gh^{-1}=e$
and $\PP^1_{g,h}(H)=0$ otherwise.
\item[(2)]
$\PCdim H=1$.
\end{enumerate}
\end{corollary}

\begin{proof} The proof is similar to the proof of Corollary
\ref{zzcor6.9}.
\end{proof}

\begin{corollary}
\label{zzcor6.11}
Let $H$ be $A_G(e,\chi)$ and assume that $\chi(e)$ is a primitive 
$\ell$th root of unity for some $\ell\geq 2$. Set $h=1$.
\begin{enumerate}
\item[(1)]
$$\PP^1_{g,1}(H)=\begin{cases} k \; \overline{z} &{\text{ if $g=e$,}}\\
k\; \overline{z^{\ell}} &{\text{ if $g=e^{\ell}$,}}\\
0 &{\text{otherwise.}}
\end{cases}
$$
\item[(2)]
$$\PP^2_{g,1}(H)=\begin{cases} k \; \overline{[z]^{\ell}} 
&{\text{ if $g=e^{\ell}$,}}\\
k\; \overline{e^{\ell}z\otimes z^{\ell}} &{\text{ if $g=e^{\ell+1}$,}}\\
0 &{\text{otherwise.}}
\end{cases}
$$
\end{enumerate}
\end{corollary}

\begin{proof} This follows from Theorem \ref{zzthm6.5} in the
case when $\ell\geq 2$ is finite.
\end{proof}

\section{Proof of Theorem \ref{zzthm0.1}}
\label{zzsec7}

In this section we assume that the base field $k$ is algebraically
closed of characteristic zero.
First recall some notations introduced in \cite{WZZ1}. Let 
$P_{g, \xi,n}$ denote the $k$-linear span of skew primitive elements 
$y$ satisfying
$$\Delta(y)=y\otimes 1+g\otimes y, \quad {\text{and}}\quad 
(T_{g^{-1}}-\xi Id_H)^n (y)\in C_0:=kG(H)$$ 
where $T_{g^{-1}}$ is the inverse 
conjugation by $g$, namely, $T_{g^{-1}}:a \to g^{-1}a g$ for
all $a\in H$. Let 
$P_{g, \xi,*}=\bigcup_{n}P_{g, \xi,n}$. We say a skew primitive 
{\it nontrivial} if it is not in $C_0$. 
Let $Z$ be the $k$-linear space spanned by all nontrivial skew 
primitive elements and let $Y_*$ be the $k$-linear space spanned by 
all nontrivial skew primitive elements in $P_{g, \xi,*}$ for all 
$\xi$ not being a primitive $\ell$th root of unity for any $\ell\geq 2$.

We refer to \cite{KL} for the definition and basic properties of
GK-dimension. 

\begin{theorem}
\label{zzthm7.1}
Let $H$ be a pointed Hopf algebra domain.
Suppose that $C_0:=kG(H)$ is commutative and that 
$$\GKdim C_0<\GKdim H<\GKdim C_0+2<\infty.$$ 
Suppose $H$ does not contain $A(1,1)$ [Example \ref{zzex5.4}] as a 
Hopf subalgebra. Then $H$ is isomorphic to either $A_G(e,\chi)$ 
or $C_G(e,\tau)$ where $G=G(H)$. 
\end{theorem}

\begin{proof} Note that $C_0=kG(H)$. Since $H\neq C_0$, $H$ contains 
a nontrivial skew primitive element, say $y$. By 
\cite[Lemma 3.7(a)]{WZZ1}, $y$ is a linear combination of skew 
primitives in $P_{g, \xi,*}$ for different pairs $(g,\xi)$. 
For at least one pair $(g,\xi)$, $P_{g,\xi,*}$ is not a subspace of 
$C_0$. Then there is a nontrivial skew primitive $y\in P_{g,\xi,1}$, 
namely $g^{-1}yg-\xi y= \alpha(g-1)$. Now we claim that $\xi$ is
either 1 or not a root of unity.

Suppose otherwise that $\xi$ is a primitive $\ell$th root of unity for 
some $\ell\geq 2$. Then, after replacing $y$ by 
$y+\frac{\alpha}{\xi-1}(g-1)$, we may assume that $g^{-1}yg-\xi y=0$. 
Then the subalgebra $H'$ generated $g^{\pm 1}$ and $y$ is a 
noncommutative Hopf subalgebra domain. By \cite[Lemma 4.5]{GZ}, 
$\GKdim H'\geq 2$. Clearly, $H'$ is a quotient algebra of 
$A(1,\xi)$ [Example \ref{zzex5.4}] by sending $e\to g$
and $z\to y$. It is known that $A(1,\xi)$ is a domain of 
GK-dimension 2 \cite[Construction 1.1]{GZ}. Therefore $H'\cong A(1,\xi)$. 
Note that $A(1,\xi)$ contains 
$A(1,1)$ as a Hopf subalgebra, a contradiction to the hypotheses. 
Therefore we proved the claim that $\xi$ is not a primitive 
$\ell$th root of unity for any $\ell\geq 2$.

By the notation introduced before the theorem, the claim is equivalent
to say that $Z=Y_*$. Since $H\neq C_0$, we have $Y_*\not\subset C_0$. 
By \cite[Theorem 3.10]{WZZ1}, 
$$\dim Y_*/(Y_*\cap C_0)\leq \GKdim H-\GKdim C_0<2$$
which implies that $\dim Y_*/(Y_*\cap C_0)=1$. By definition (and also
see \cite[Lemma 3.2(b)]{WZZ2}), 
$Y_*/(Y_*\cap C_0)=\bigoplus_{(g,\xi)} P_{g,\xi,*}/(P_{g,\xi,*}\cap C_0)$.
Hence there is only one pair, say $(e,\xi)$, such that 
$P_{e,\xi,*}/(P_{e,\xi,*}\cap C_0)\neq 0$ and 
$P_{e,\xi,*}/(P_{e,\xi,*}\cap C_0)$ is 1-dimensional. This implies
that $P_{e,\xi,*}=P_{e,\xi,1}=ky \oplus k(e-1)$. Combining these observations, 
we have
$$Z=Y_*=P_{e,\xi,*}=P_{e,\xi,1}=ky\oplus k(e-1)$$
for some $e\in G(H)$ and some $\xi\in k^{\times}$ being 1 or not 
a root of unity. 

Let $K$ be the subalgebra of $H$ generated by $Z$ and $C_0$, or
equivalently, by $y$ and $C_0$. It is clear that $K$ is a Hopf subalgebra. 
Applying \cite[Lemma 2.2(c)]{WZZ1} to 
$V=ky\oplus k(e-1)$, there is a $z\in V\setminus k(e-1)$ such that 
either
\begin{enumerate}
\item[(i)]
there is a character $\chi: G\to k^{\times}$ such that
$h^{-1}z h=\chi(h) z$ for all $h\in G$, or
\item[(ii)]
there is an additive character $\tau: G\to (k,+)$ such that
$h^{-1}z h=z+\tau(h)(e-1)$ for all $h\in G$. 
\end{enumerate}

In the first case $K$ is a quotient of $A_G(e,\chi)$ and in the
second case $K$ is a quotient of $C_G(e,\tau)$. We claim that
$K\cong A_G(e,\chi)$ in the first case and that $K\cong C_G(e,\tau)$
in the second case. We only prove the claim for the first case.
Consider the natural Hopf map $f: A_G(e,\chi)\to H$ which is 
injective on $C_0+C_0 z=C_1(H)$ by definition. By 
\cite[Theorem 5.3.1]{Mo}, $f$ is injective. Consequently, 
$K\cong A_G(e,\chi)$ (and $K\cong C_G(e,\tau)$ in the second case).

By definition, $K$ is generated by all grouplikes and
skew primitive elements of $H$. By the above paragraph, $K$ 
is isomorphic to either $A_G(e,\chi)$ or $C_G(e,\tau)$. In the
first case $q:=\chi(e)$ is not a primitive $\ell$th root of unity
for any $\ell\geq 2$. By Corollaries \ref{zzcor6.9} and \ref{zzcor6.10},
$\PCdim K=1$. Finally, by 
Proposition \ref{zzpro2.4}, $H=K$ as desired.
\end{proof}

An algebra $A$ is called \emph{locally PI} if every affine
subalgebra of $A$ is PI. 

\begin{theorem}
\label{zzthm7.2}
Let $H$ be a pointed Hopf domain such that $\GKdim H<3$ and that 
$C_0=k{\mathbb Z}$. If $H$ is not locally PI, then $H$ is isomorphic 
to either 
\begin{enumerate}
\item[(1)]
$A(n,q)$ of Example \ref{zzex5.4} where $n>0$ and
$q$ is not a root of unity, or
\item[(2)]
$C(n)$ of Example \ref{zzex5.5} for $n\geq 2$.
\end{enumerate}
As a consequence, $H$ is affine and noetherian.
\end{theorem}

\begin{proof} By Examples \ref{zzex5.4} and \ref{zzex5.5}, 
$A_{G}(e, \chi)$ becomes $A(n,q)$  and $C_G(e,\tau)$ becomes 
$C(n)$ when $G={\mathbb Z}$. 

Note that $A(1,1)$ is a commutative Hopf algebra of GK-dimension two.
If $H$ is not a locally-PI domain of GK-dimension two, by 
\cite[Theorem 1.1]{Be}, $H$ does not contain $A(1,1)$ as a subalgebra.
Since $H$ is not commutative, $\GKdim H\geq 2$ \cite[Lemma 4.5]{GZ}. 
Hence all hypotheses of Theorem \ref{zzthm7.1} hold.
Therefore the assertion follows from Theorem \ref{zzthm7.1}. 
\end{proof}

\begin{proof}[Proof of Theorem \ref{zzthm0.1}]
Since $\GKdim C_0\leq \GKdim H<3$, $\GKdim C_0$ is either 0, or 1, or 2.
If $\GKdim C_0=0$, by \cite[Theorem 1.9]{WZZ2} $H$ is isomorphic
to the algebra in part (1).

If $\GKdim C_0=2$, by \cite[Theorem 1.7]{WZZ2}, $H$ is isomorphic to 
the group algebra $k\Gamma$ where $\Gamma$ is either ${\mathbb Z}^{2}$
or ${\mathbb Z}\rtimes {\mathbb Z}$ as given in \cite[Theorem 1.4(1)]{WZZ2}.
In both cases $k\Gamma$ are PI, a contradiction.

If $\GKdim C_0=1$, by \cite[Lemma 1.10]{WZZ2}, $C_0=kG$ is affine. 
Since $C_0$ is an affine domain of GK-dimension one, 
$G$ is finitely generated abelian \cite[Lemma 4.5]{GZ} and torsionfree 
of rank 1, namely, $G={\mathbb Z}$. By Theorem \ref{zzthm7.2},
$H$ is isomorphic to algebras in part (2) or part (3).
\end{proof}

\begin{corollary}
\label{zzcor7.3} 
Let $H$ be a pointed Hopf algebra domain.
Then $\GKdim H$ can not be strictly between 2 and 3.
\end{corollary}

\begin{proof} By the definition of GK-dimension, we can assume
that $H$ is affine.
If $H$ is PI or if $\GKdim H\leq 2$, then $\GKdim H$ is 
an integer. So the assertion follows. If $H$ is 
not PI, the assertion follows from Theorem \ref{zzthm0.1}. 
\end{proof}

Next we remove the hypothesis ``affine'' from Theorem \ref{zzthm0.1}.

\begin{theorem}
\label{zzthm7.4} 
Let $H$ be a pointed Hopf domain of GK-dimension 2. 
If $H$ is not locally PI, then $H$ is isomorphic to one of following:
\begin{enumerate}
\item[(1)]
$U(\mathfrak g)$ of the 2-dimensional {\rm{(}}non-abelian{\rm{)}} 
solvable Lie algebra ${\mathfrak g}$.
\item[(2)]
The algebra $A_{G}(e, \chi))$ of Example \ref{zzex5.4} where $G$ is a 
nontrivial subgroup of $({\mathbb Q},+)$, $\chi: G\to k^{\times}$ 
is a character and  $q:=\chi(e)$ is not a root of unity.
\item[(3)]
The algebra $C_G(e, \eta)$ of Example \ref{zzex5.5} where $G$ is a 
nontrivial subgroup of $({\mathbb Q},+)$ and $\eta:G\to (k,+)$
is a nonzero additive character of $G$.
\end{enumerate}
\end{theorem}

\begin{proof} 
Let $K$ be any affine (non PI) Hopf subalgebra of $H$ of GK-dimension 
$2$. By Theorem \ref{zzthm0.1} we have three cases.

Case 1: $K=U(\fg)$ where $\fg$ is the 2-dimensional 
solvable Lie algebra. So $\GKdim K=\GKdim H=2$. It follows from
\cite[Lemma 7.2]{Zh} that $H=K=U(\fg)$.

In other two cases, $K$ is generated by (nontrivial) grouplikes 
and skew primitives. In particular, $G(K)={\mathbb Z}$. Consequently,
$\GKdim kG=1$ where $G=G(H)$. 
Since $kG$ is a domain of GK-dimension one,
$G$ is a nontrivial subgroup of the abelian group ${\mathbb Q}$. 
Since $H$ is not locally PI, 
$H$ does not contain the commutative Hopf subalgebra 
$A(1,1)$. So $H$ satisfies the hypotheses of
Theorem \ref{zzthm7.1}. By Theorem \ref{zzthm7.1}, 
$H$ is isomorphic to either $A_G(e,\chi)$ or $C_G(e,\eta)$.

Case 2: $H=A_G(e,\chi)$. Since $H$ is not locally
PI, $q:=\chi(e)$ is not a root of unity.

Case 3: $H=C_G(e,\eta)$. Since $H$ is not locally PI, $\eta$ is nonzero.
\end{proof}

\begin{remark}
\label{zzrem7.5} 
\begin{enumerate}
\item[(1)]
In an unpublished note, a result similar to Theorem \ref{zzthm7.4}
was proved by Goodearl and the second-named author when they studied
non-noetherian Hopf domains of GK-dimension two that satisfy
the extra condition $\Ext^{1}_H(k,k)\neq 0.$
\item[(2)]
Connected Hopf algebras of GK-dimension three
were classified by the third-named author in \cite{Zh}. It would be  
interesting to classify all affine, non-PI, pointed Hopf 
domains of GK-dimension three. We believe that this is an attackable 
project once we understand the theory of Hopf double Ore extension.
The double Ore extension was introduced in \cite{ZZ1, ZZ2} and Hopf 
double Ore extension should be a useful construction for 
Hopf algebras.
\end{enumerate}
\end{remark}

\section{Examples, Part II}
\label{zzsec8}

In this section we give more examples which will be used in 
a classification result in the next section.

\begin{example}
\label{zzex8.1} \cite{KR, WYC}
Let $K$ be the Hopf algebra $A_G(e, \chi)$ defined in Example 
\ref{zzex5.4} and assume that 
$\chi(e)$ is a primitive $\ell$th root of unity for some 
$\ell\ge2$. Let $I$ be the algebra ideal generated by the 
element $z^{\ell}-\lambda(e^{\ell}-1)$ for some $\lambda\in k$.
Since $z^{\ell}-\lambda(e^{\ell}-1)$
is $(e^{\ell},1)$-primitive, $I$ is also a Hopf ideal of $K$.
Let $E_G(e, \chi,\ell, \lambda)$ denote the Hopf quotient $K/I$.
This is the Hopf algebra $H_{\mathcal D}$ studied in \cite{KR} and
the Hopf algebra described in \cite[Section 2]{WYC}.
Note that $\chi(e^{\ell})=1$. 
\end{example}

Since $K$ is generated by grouplikes and skew primitives, so is
$E_G(e, \chi,\ell, \lambda)$. 

\begin{lemma}
\label{zzlem8.2} Let $K$ be the Hopf algebra $A_G(e, \chi)$ 
and let $E=E_G(e, \chi,\ell, \lambda)$. Let $w_{\lambda}
=z^{\ell}-\lambda(e^{\ell}-1)$.
\begin{enumerate}
\item[(1)]
$w_0:=z^{\ell}$ is normal in $K$. If $\lambda(e^{\ell}-1)
\neq 0$, then $w_{\lambda}$ is normal in $K$ if and only
if $\chi^{\ell}$ is the trivial character.
\item[(2)]
If $w_{\lambda}$ is normal, then $G(E)=G$ and
$E=\oplus_{i=0}^{\ell-1} z^i (kG)=\oplus_{i=0}^{\ell-1} (kG)z^i$.
\item[(3)]
Suppose that $G(E)=G$ and that $kG\subsetneq E$. Then 
$z^{\ell}-\lambda(e^{\ell}-1)$ is a normal element in $K$ and 
$E=\oplus_{i=0}^{\ell-1} z^i (kG)=\oplus_{i=0}^{\ell-1} (kG)z^i$.
\end{enumerate}
\end{lemma}

\begin{proof} (1,2) Follow by direct computation. 

(3) Note that $w_{\lambda} z= z w_{\lambda}$ and
$$\lambda (\chi^{\ell}(g)-1) g (e^{\ell}-1)= 
w_{\lambda} g- \chi^{\ell}(g) gw_{\lambda}\in I$$
for all $g\in G$. If $w_{\lambda}$ is not normal, 
then the above expression has to be nonzero for 
at least one $g$, hence $0\neq e^{\ell}-1\in I$, 
which shows that the map $kG \to E$ is not injective. 
Therefore, if $kG$ embeds in $E$, then $w_{\lambda}$ 
has to be normal. The direct sum decomposition in 
this case has been already stated in part (2), 
and so the inequality $kG \neq E$ holds automatically.
\end{proof}

\begin{remark}
\label{zzrem8.3}
Clearly the previous lemma asserts that $G(E)=G$ and 
$kG\subsetneq E$ if and only if $w_{\lambda}$ is 
normal in $K= A_G(e, \xi)$.
From now on we require that, in the definition of 
$E_G(e,\xi,\ell, \lambda )$ in Example $\ref{zzex8.1}$, 
the element $w_{\lambda}$ is normal
in $K$ and therefore $E=\oplus_{i=0}^{\ell-1} z^i (kG)=\oplus_{i=0}^{\ell-1} (kG)z^i$. 
\end{remark}

We need some general lemmas before we compute the primitive 
cohomology of $E$.
An ${\mathbb N}$-filtration $F=\{F_{n} H \mid n\geq 0\}$ of $H$ 
is called \emph{a Hopf filtration} if 
\begin{enumerate}
\item[(i)]
$G(H)\subset F_0 H$,
\item[(ii)]
$F_n H F_m H\subset F_{n+m} H$ for all $n,m$,
\item[(iii)]
$\Delta (F_n H)\subset \sum_{i=0}^{n} F_{i} H\otimes F_{n-i} H$ for
all $n$, and
\item[(iv)]
$S(F_n H)\subset F_n H$ for all $n$.
\end{enumerate}
In this case, the associated graded algebra $\gr_F H:=
\oplus_{i\geq 0} F_i H/F_{i-1} H$ is a Hopf algebra. Let $H$ be a 
pointed Hopf algebra with $G(H)=G$. Then the coradical filtration 
of $H$ is automatically a Hopf filtration and the associated 
graded algebra, denoted by $\gr H$, is a graded Hopf algebra. 

\begin{lemma}
\label{zzlem8.4} Let $F$ be a Hopf filtration of a Hopf algebra
$H$. 
\begin{enumerate}
\item[(1)]
For any $n\geq 0$ and $g,h\in G$, 
$$\dim \PP^n_{g,h}(H)\leq \dim \PP^n_{g,h}(\gr_F H).$$
As a consequence, 
$\PCdim H\le \PCdim (\gr_F H).$
\item[(2)]
Fix any $g,h\in G$. If there is an $n_0$ such that
$\PP^n_{g,h}(\gr_F H)=0$ for all $n\neq n_0$, then
$\PP^n_{g,h}(H)\cong \PP^n_{g,h}(\gr_F H)$ for all $n$.
\end{enumerate}
\end{lemma}

In general the inequalities are strict, see Remark \ref{zzrem9.4}(1).
If $F$ is the coradical filtration, it is unknown if
$\PP^n_{g,h}(H)\cong \PP^n_{g,h}(\gr H)$ for all $n$ and
$g$. 

\begin{proof}[Proof of Lemma \ref{zzlem8.4}]
(1) The proof uses a standard spectral sequence argument. For any 
fixed $g,h\in G$, the filtration $F$ on $H$ induces a filtration 
on the complex $T_{g,h}(H)$. Now the filtered complex $T_{g,h}(H)$ 
gives rise to a spectral sequence $E$ convergent to 
$\{\mathfrak{P}^n_{g,h}(H)\}_{n\geq 0}$ \cite[Theorem 5.5.1]{We}. 
Since the $E^1$ page is just $\{\mathfrak{P}^n_{g, h}(\gr_F H)_{p}\}$, 
the result follows.

(2) If $\PP^n_{g,h}(\gr_F H)=0$ for all $n\neq n_0$, then 
the spectral sequence in the proof of part (1) collapses. So 
the isomorphism follows.
\end{proof}

If $H$ is a graded pointed Hopf algebra with $kG$ being the degree 0 part, 
or more generally, there is a Hopf algebra projection $\pi: H\to kG$,
then a well-known result of Radford \cite{Ra} states that $H$ has a 
decomposition 
\begin{equation}
\label{E8.4.1}\tag{E8.4.1}
H\cong R\# kG,
\end{equation}
where $R$ is a (graded) braided Hopf algebra in the Yetter-Drinfeld 
category over $kG$. In fact,  $R$ is defined to be 
$$H^{co \; \pi}:=\{a\in H\mid (id\otimes \pi)\Delta(a)=
a\otimes 1\}.$$ 
Going back to the coradical filtration, when $H$ is pointed (not
necessarily graded), $\gr H$ has a decomposition like \eqref{E8.4.1}.
Notice that, as a coalgebra, $R$ is connected.

\begin{lemma}
\label{zzlem8.5}
Let $H$ be a graded Hopf algebra with $H_0=kG(H)$ and write
$H=R\# kG$ as in \eqref{E8.4.1}. Let $n$ be a natural number. Then 
$$\Cotor^n_R(k, k)\cong \Cotor^n_{H}(kG, k),$$
or equivalently, 
$$\PP^n_{1, 1}(R)\cong \bigoplus_{g\in G}\PP^n_{g, 1}(H).$$
As a consequence,
$$\PCdim R=\PCdim H.$$
\end{lemma} 

We conjecture that the isomorphism in Lemma \ref{zzlem8.5} induces 
an algebra isomorphism from $\bigoplus_{n\geq 0}\PP^n_{1, 1}(R)$
to $\bigoplus_{n\geq 0}\{\bigoplus_{g\in G}\PP^n_{g, 1}(H)\}$, 
once the algebraic structure is defined in Section \ref{zzsec10}.

\begin{proof}[Proof of Lemma \ref{zzlem8.5}]
Let $K=kG$ and 
$\overline{H}=H/HK^+$. By \cite[Proof of Theorem 3]{Ra}, 
$\overline{H}\cong R$ as coalgebras. 

We will use results and notation from \cite{FMS}. Please see  
\cite{Mo} or \cite{FMS} for some definitions if necessary.
It follows from \cite[Lemma 2.4]{FMS} that 
$K=\,^{co \overline{H}}H$. Hence, by \cite[Proposition 4.5 (b)]{FMS}, 
the category $\mathcal{M}^H_K$ is equivalent to 
$\mathcal{M}^{\overline{H}}$, where the functor 
$\Phi: \mathcal{M}^H_K\rightarrow \mathcal{M}^{\overline{H}}$ 
sends $V$ to $\overline{V}:=V/VK^+=V\otimes_K k$. In fact, 
any $V\in \mathcal{M}^H_K$ is a free $K$-module. This is 
because $V\cong \overline{V}\square_{\overline{H}}H$ by 
\cite[Proposition 4.5 (b)]{FMS} and 
$H\cong \overline{H}\otimes K$ as left 
$\overline{H}$-comodule and right $K$-module.

Let $\Psi$ and $\Gamma$ be the cotensor functors 
$-\square_{\overline{H}} k$ and $-\square_H k$, respectively.
It is clear that 
$$\Cotor^*_R(k, k)\cong \Cotor^*_{\overline{H}}(k, k)
=\HB^*\mathcal{R}\Psi(k)$$ and 
$$\Cotor ^n_{H}(kG,k)=\HB^*\mathcal{R}\Gamma(kG)
=\HB^*\mathcal{R}\Gamma(K),$$ 
where $\mathcal{R}\Gamma(-)$ and $\mathcal{R}\Psi(-)$ are 
right derived functors of $\Gamma$ and $\Psi$, respectively.
We claim that $\Gamma$ is naturally isomorphic to 
$\Psi\circ \Phi$. For any $V\in \mathcal{M}^H_K$,
$$\Gamma(V)=V\square_Hk\cong 
\overline{V}\square_{\overline{H}}H\square_Hk
\cong \overline{V}\square_{\overline{H}}k=\Psi\circ \Phi(V).$$
The naturality of the map is easy to check. Therefore we proved the claim.
Since $\Phi$ is exact, by the Grothendieck spectral sequence 
\cite[Theorem 5.8.3]{We}, we have that
$$\HB^*\mathcal{R}\Gamma(K)\cong 
\HB^*\mathcal{R}\Psi(\Phi(K))=\HB^*\mathcal{R}\Psi(\overline{K})
=\HB^*\mathcal{R}\Psi(k).$$
The assertion follows by combining these isomorphisms. The consequence
is easy.
\end{proof}

Lemma \ref{zzlem8.5} is helpful in the following computation.

\begin{corollary}
\label{zzcor8.6}
Let $H$ be $E_G(e,\chi, \ell, \lambda)$. 
\begin{enumerate}
\item[(1)]
$\dim (\oplus_{g\in G} \PP^n_{g,1}(H))=1$ for all $n$.
As a consequence, $\PCdim H= \infty$.
\item[(2)]
$$\PP^1_{g,1}(H)=\begin{cases} k \; \overline{z} &{\text{ if $g=e$,}}\\
0 &{\text{otherwise.}}
\end{cases}
$$
\item[(3)]
$$\PP^2_{g,1}(H)=\begin{cases} k \; \overline{[z]^{\ell}} 
&{\text{ if $g=e^{\ell}$,}}\\
0 &{\text{otherwise.}}
\end{cases}
$$
\end{enumerate}
\end{corollary}

\begin{proof} (1) By Lemma \ref{zzlem8.2}(2), $H$ has a $k$-linear 
basis of the form 
$$\{gz^n | g\in G, 0\le n< \ell\}.$$
Moreover, by Lemma \ref{zzlem5.3},
\begin{equation}\label{E8.6.1}\tag{E8.6.1}
\Delta(gz^n)=\sum_{i=0}^n {n \choose i}_q ge^{n-i} z^i\otimes gz^{n-i},
\end{equation}
where $q=\chi(e)$. By using the basis, we see that $H$ is isomorphic 
to $E_G(e,\chi, \ell, 0)$ as coalgebras. Hence, without loss of 
generality, we assume that $H=E_G(e,\chi, \ell, 0)$. 
In this case, $H$ is, in fact, coradically graded in the sense of 
\cite[Definition 1.13]{AS1} with respect to the grading 
determined by $\deg g=0$ for any $g\in G$ and $\deg z=1$. 
Hence $H\cong R\#kG$, where $R$, by \eqref{E8.6.1}, is 
the subalgebra of $H$ generated by $z$. The coalgebra 
structure of $R$ is given by 
$$\Delta_R(z^n)=\sum_{i=0}^n {n \choose i}_q z^i\otimes z^{n-i},$$
where $0\le n< \ell$ and $q=\chi(e)$. Consequently, the graded dual 
of $R$, denoted by $R^\circ$, is isomorphic to $k[x]/(x^\ell)$ as 
algebras. By Lemmas \ref{zzlem3.6} and \ref{zzlem8.5} and Proposition 
\ref{zzpro3.7}, for each $n\geq 0$, 
$$\dim \{\bigoplus_{g\in G}\mathfrak{P}^n_{g, 1}(H)\}=
\dim \mathfrak{P}^n_{1, 1}(R)=\dim \Tor^{R^\circ}_n(k, k)=1,$$
and $\PCdim H=\PCdim R=\gldim R^\circ=\infty$. Therefore part 
(1) follows.

(2) By the proof of (1), we might assume that
$H=E_G(e,\chi, \ell, 0)$. Now by the 
defining relation of $H$ and \eqref{E8.6.1}, $H$ is a graded 
Hopf algebra by setting $\deg g=0$ for any $g\in G$ and $\deg z=1$. 
By a degree argument, the element $z$ is the only $(1, g)$-primitive 
element of degree $\ge 1$ for some $g\in G$. This implies part (2).

(3) Note that $[z]^{\ell}$ is an element in 
$Z^2_{g, 1}(H)$ of degree $\ell$, while all elements in $B^2_{g, 1}(H)$ 
have degree at most $\ell-1$. This shows that $\dim \bigoplus_{g\in G}
\PP^2_{g, 1}(H)\geq 1$. The assertion follows from part (1).
\end{proof}

We now introduce more HOEs.

\begin{example}
\label{zzex8.7}
In this and later examples in the rest of this section
let $K$ be the Hopf algebra $E_G(e,\chi, \ell, \lambda)$ defined 
as in Example \ref{zzex8.1}. Here $\lambda$ can be chosen to be $1$
or $0$ after rescaling. Recall from Example \ref{zzex5.4} that
$e$ is in the center of $G$. Then there is a unique algebra 
automorphism, denoted by $\theta$, of $K$ such that 
\begin{equation}
\label{E8.7.1}\tag{E8.7.1}
{\text{$\theta(z)=z$ and $\theta(g)=\chi(g^\ell)g$ for
all $g\in G$.}}
\end{equation} 
All HOEs in the rest of this section will be of the form 
$K[w;\theta,\delta]$ for different choices of $\delta$. 
Note that we are using symbols different from \cite[Theorem 2.4]{BOZZ}.
For comparison, the $(x,a,w)$ on \cite[pp. 2414-2415]{BOZZ} match 
with $(w,e^{\ell},[z]^{\ell})$ in the rest of this section. Hence
\cite[(19) on page 2414]{BOZZ} becomes
\begin{equation}
\label{E8.7.2}\tag{E8.7.2}
\Delta(w)=w\otimes 1+e^{\ell}\otimes w+[z]^{\ell}.
\end{equation}

Now, in this example, we further assume that
\begin{equation}
\label{E8.7.3}\tag{E8.7.3}
\lambda=0\quad {\text{and}}\quad
\delta=0
\end{equation}
and let $H:=K[w; \theta]$ be a HOE of 
$K$. We need to verify \cite[Theorem 2.4(d,e,f)]{BOZZ} before we
say that $K[w;\theta]$ is a Hopf algebra, but verification is 
routine. For example,
\cite[(23) on page 2415]{BOZZ} is Lemma \ref{zzlem5.6}(2),
\cite[Theorem 2.4(e)]{BOZZ} follows from Lemma \ref{zzlem8.8}(1)
below, and the other parts are straightforward and therefore 
omitted. This Hopf algebra is denoted by $F_G(e,\chi, \ell)$.
\end{example}

Keep in mind that all later examples in this section satisfy
\eqref{E8.7.1} and \eqref{E8.7.2}.

\begin{lemma}
\label{zzlem8.8}
Let $K$ be the Hopf algebra $E_G(e,\chi,\ell,0)$. 
\begin{enumerate}
\item[(1)]
Let $\theta$ be defined as in Example \ref{zzex8.7}. Then 
$$[z]^{\ell}\Delta(r)=\Delta (\theta(r)) [z]^{\ell}$$
for all $r\in K$. 
\item[(2)]
Suppose $\chi(g^{\ell})=1$ for all $g\in G$. Then 
$[z]^{\ell}\Delta(r)=\Delta (r) [z]^{\ell}$
for all $r\in K$.
\end{enumerate}
\end{lemma}

\begin{proof}
(1) Since $\theta$ is an algebra automorphism
and $\Delta$ is an algebra morphism, we only need to show the
assertion for the generators of $K$. By definition, $A_{G}(e,\chi)$,
and hence $E_G(e,\chi, \ell,0)$, is generated by 
$G$ and $z$. If $r=g\in G$, using the expression of $[z]^{\ell}$
\eqref{E5.5.2}, we have 
\begin{align}
\label{E8.8.1}\tag{E8.8.1}
\; [z]^{\ell} \Delta(g)&=[z]^{\ell} (g\otimes g)=(\sum_{i=1}^{\ell} b_i
e^{\ell-i}z^i\otimes z^{\ell-i})(g\otimes g)\\
\notag
&= \chi(g)^{\ell} (g\otimes g)(\sum_{i=1}^{\ell} b_i
e^{\ell-i}z^i\otimes z^{\ell-i})=\Delta(\theta(g))[z]^{\ell}.
\end{align}
So it remains to show $[z]^{\ell} \Delta(z)=\Delta(\theta(z))[z]^{\ell}=
\Delta(z)[z]^{\ell}$ since $\theta(z)=z$ in $K=E_G(e,\chi,\ell,0)$.
But this follows from Lemma \ref{zzlem5.6}(1) and the fact 
$z^{\ell}=0$ in $K$.

(2) This is a consequence of part (1) and the fact $\theta=Id$. 
\end{proof}

\begin{example}
\label{zzex8.9}
Let $(K,\theta)$ be as in Example \ref{zzex8.7} with 
$$\lambda=0.$$
Let $\eta: G\rightarrow k$ be a nonzero map such that $\eta(e)=0$ and that
$$\eta(gh)=\eta(g)+\chi^{\ell}(g)\eta(h)$$
for all $g,h\in G$. 
Then there is a unique $\theta$-derivation 
$\delta$ of $K$ such that 
\begin{equation}
\label{E8.9.1}\tag{E8.9.1}
{\text{$\delta(z)=0$ and $\delta(g)=\eta (g)g(e^{\ell}-1)$ 
for all $g\in G$.}}
\end{equation} 
Let $H:=K[w; \theta,\delta]$ 
be a HOE of $K$ with $\Delta(w)=w\otimes 1+e^{\ell}\otimes w+[z]^{\ell}$.

To verify that this is a HOE, let us only check \cite[Theorem 2.4(e)]{BOZZ} 
and leave other verifications (which are straightforward) to the reader. 
By Lemma \ref{zzlem8.8}(1), the equation in 
\cite[Theorem 2.4(e)]{BOZZ} is equivalent to 
$$X(r)=0$$
where
\begin{equation}
\label{E8.9.2}\tag{E8.9.2}
X(r):=\Delta (\delta(r))-\delta(r_1)\otimes r_2-e^{\ell} r_1\otimes 
\delta(r_2).
\end{equation}
If $r=g\in G$, we have
$$\Delta(\delta(r))=\Delta(\eta(g) g(e^{\ell}-1))
=\eta(g)(g\otimes g)(e^{\ell}\otimes e^{\ell}-1\otimes 1)$$
and 
$$\delta(r_1)\otimes r_2+e^{\ell} r_1\otimes 
\delta(r_2)=\eta(g)g(e^{\ell}-1)\otimes g+e^{\ell}g\otimes 
\eta(g)g(e^{\ell}-1)$$
which is equal to $\Delta(\delta(r))$. Hence $X(r)=0$ when
$r\in G$.
If $r=z$, $\Delta(\delta(r))=\Delta(0)=0$ and
$$\delta(r_1)\otimes r_2+e^{\ell} r_1\otimes 
\delta(r_2)=\delta(e)\otimes z+e^{\ell}z\otimes 
\delta(1)$$
which is zero as $\eta(1)=\eta(e)=0$. Hence $X(z)=0$. Since
$\delta$ satisfies a $\theta$-version of the Leibniz rule
$\delta(xy)=\delta(x)y+\theta(x)\delta(y)$ for all
$x,y\in K$, $X$ satisfies the following version of
the Leibniz rule
\begin{equation}
\label{E8.9.3}\tag{E8.9.3}
X(rs)=X(r) \Delta(s)+(\theta\otimes 1)\Delta(r) X(s)
\end{equation}
for all $r,s\in K$. The equation $X(r)=0$ 
(before \eqref{E8.9.2})
follows by using the above version of the Leibniz rule
since $K$ is generated
by $G$ and $z$. So we verified \cite[Theorem 2.4(e)]{BOZZ}.
We denote the Hopf algebra $K[w;\theta, \delta]$ by 
$L_G(e,\chi, \ell,\eta)$.
\end{example}


In the next four examples, $\theta$ is the identity of $K$ (or equivalently,
$\chi^{\ell}(g)=1$ for all $g\in G$). Since $\theta$ is the identity,
the Ore extension $K[w;\theta,\delta]$ is denoted by $K[w;\delta]$.

\begin{example}
\label{zzex8.10}
Suppose 
\begin{equation}
\label{E8.10.1}\tag{E8.10.1}
{\text{$e^{\ell}=1$ and $\chi^{\ell}(g)=1$ for all $g\in G$.}}
\end{equation}
Let $(K,\theta)$ be as in Example \ref{zzex8.7} with 
$$\lambda=0.$$ 
There is a unique derivation $\delta$ of $K$ such that 
\begin{equation}
\label{E8.10.2}\tag{E8.10.2}
{\text{$\delta(z)=\xi z$ and $\delta(g)=0$ for any $g\in G$.}}
\end{equation}
By \cite[Theorem 2.4]{BOZZ}, $H:=K[w;\delta]$ 
is a HOE of $K$ with 
$\Delta(w)=w\otimes 1+e^{\ell}\otimes w+[z]^{\ell}$. Denote this 
Hopf algebra by $N_G(e,\chi, \ell, \xi)$. 
\end{example}

\begin{example}
\label{zzex8.11}
Suppose 
\begin{equation}
\label{E8.11.1}\tag{E8.11.1}
{\text{$e^{\ell}\neq 1$ and $\chi^{\ell}(g)=1$ for all $g$.}}
\end{equation}
Let $(K,\theta)$ be as in Example \ref{zzex8.7} with 
$$\lambda=1.$$ 
Let $\eta$ be an additive character of $G$ such that 
$\eta(e)=q-1$ where $q=\chi(e)$. Define a derivation $\delta$ 
of $K$ by 
\begin{equation}
\label{E8.11.2}\tag{E8.11.2}
{\text{$\delta(z)=(q-1)e^{\ell}z$ and 
$\delta(g)=\eta(g)g(e^{\ell}-1)$ for all $g\in G$.}}
\end{equation}
To check \cite[Theorem 2.4(e)]{BOZZ}, we let, for all $r\in K$, 
$$Y(r):=[z]^{\ell} \Delta(r)-\Delta(r) [z]^{\ell}$$
which is the right-hand side of \cite[(21) on page 2415]{BOZZ}
in our setting, while $X(r)$ of \eqref{E8.9.2} is 
the left-hand side of \cite[(21) on page 2415]{BOZZ}. 
It is not difficult to see that $Y(r)$ satisfies the
equation
\begin{equation}
\label{E8.11.3}\tag{E8.11.3}
Y(rs)=Y(r)\Delta(s)+\Delta(r) Y(s)
\end{equation}
for all $r,s\in K$, which is similar to \eqref{E8.9.3}
where $\theta$ is the identity. When $r=g\in G$, it is easy
to see that $Y(r)=X(r)=0$. When $r=z$, Lemma \ref{zzlem5.6}(1)
shows that
$$\begin{aligned}
Y(z)&=-[\Delta(z), [z]^{\ell}]=
(1-q) ez^{\ell}\otimes z+(q-1) e^{\ell} z\otimes z^{\ell}\\
&=(1-q) e (e^{\ell}-1)\otimes z+(q-1) e^{\ell} z\otimes 
(e^{\ell}-1).
\end{aligned}
$$
On the other hand, by setting $\Delta(z)=z_1\otimes z_2$,
$$\begin{aligned}
X(z)&=\Delta (\delta(z))-\delta(z_1)\otimes z_2
-e^{\ell} z_1\otimes \delta(z_2)\\
&=(q-1) e^{\ell}z\otimes e^{\ell} -(q-1) e^{\ell}z\otimes 1
-(q-1) e^{\ell+1} \otimes z+(q-1) e\otimes z\\
&= Y(z).
\end{aligned}
$$
Since $K$ is generated by $G$ and $z$, the equations \eqref{E8.9.3}
and \eqref{E8.11.3} together with $X(r)=Y(r)$ for all $r=g\in G$ and
$r=z$ imply that $X(r)=Y(r)$ for all $r\in K$. Therefore 
\cite[(21) on page 2415]{BOZZ} holds. Other hypotheses of
\cite[Theorem 2.4(ii)]{BOZZ} are relatively easy to check.
Now, by \cite[Theorem 2.4(ii)]{BOZZ}, $H:=K[w;\delta]$ 
is a HOE of $K$ with 
$\Delta(w)=w\otimes 1+e^{\ell}\otimes w+[z]^{\ell}$. Denote this 
Hopf algebra by $O_G(e,\chi, \ell, \eta)$. 
\end{example}

\begin{example}
\label{zzex8.12}
Suppose 
\begin{equation}
\label{E8.12.1}\tag{E8.12.1}
{\text{$e^{\ell}\neq 1$ and $\chi^{\ell}(g)=1$ for all $g$.}}
\end{equation}
Let $(K,\theta)$ be as in Example \ref{zzex8.7} with 
$$\lambda=0.$$ 
Let $\eta$ be an additive character of $G$. 
Define a derivation $\delta$ of $K$ by 
\begin{equation}
\label{E8.12.2}\tag{E8.12.2}
{\text{$\delta(z)=-\eta(e)z$ and 
$\delta(g)=\eta(g)g(e^{\ell}-1)$ for all $g\in G$.}}
\end{equation}
There is no further restriction on $\eta$.
By \cite[Theorem 2.4]{BOZZ}, $H:=K[w; \delta]$ 
is a HOE of $K$ with 
$\Delta(w)=w\otimes 1+e^{\ell}\otimes w+[z]^{\ell}$. Denote this 
Hopf algebra by $P_G(e,\chi, \ell, \eta)$. 
\end{example}

One difference between Examples \ref{zzex8.11} and \ref{zzex8.12}
is the choice of $\lambda$.

\begin{example}
\label{zzex8.13}
Suppose 
\begin{equation}
\label{E8.13.1}\tag{E8.13.1}
{\text{$e^{\ell}\neq 1$, $e^{2\ell}=1$, and $\chi^{\ell}(g)=1$ for all $g$.}}
\end{equation} 
Let $(K,\theta)$ be as in Example \ref{zzex8.7} with 
$$\lambda=1.$$ 
Let $\eta$ be an additive character of $G$ and $q=\chi(e)$. Since
$e^{2\ell}=1$ and ${\rm{char}} \; k=0$, we have $\eta(e)=0$.
Define a derivation $\delta$ of $K$ by 
\begin{equation}
\label{E8.13.2}\tag{E8.13.2}
{\text{$\delta(z)=(q-1) z+(q-1)e^{\ell}z$
and $\delta(g)=\eta(g)g(e^{\ell}-1)$ for all $g\in G$.}}
\end{equation}
By \cite[Theorem 2.4]{BOZZ}, $H:=K[w;\delta]$ 
is a HOE of $K$ with 
$\Delta(w)=w\otimes 1+e^{\ell}\otimes w+[z]^{\ell}$. Denote this 
Hopf algebra by $Q_G(e,\chi, \ell, \eta)$. 
\end{example}

\begin{corollary}
\label{zzcor8.14}
Let $H$ be $F_G(e,\chi, \ell)$, $L_G(e,\chi, \ell,\eta)$, 
$N_G(e,\chi, \ell, \xi)$, $O_{G}(e,\chi,\ell,\eta)$, 
$P_{G}(e,\chi,\ell,\eta)$, or $Q_{G}(e,\chi,\ell,\eta)$
\begin{enumerate}
\item[(1)]
$\PP^1_{g,h}(H)=k \; \overline{hz}$ if $gh^{-1}=e$
and $\PP^1_{g,h}(H)=0$ otherwise.
\item[(2)]
$\PCdim H=1$.
\end{enumerate}
\end{corollary}

\begin{proof} (1) We use Theorem \ref{zzthm6.5} and 
use the notation there too. Set $a=e^{\ell}$ 
which replaces $e$ in Theorem \ref{zzthm6.5}. Further, 
$w$ replaces $z$, $[z]^{\ell}$ replaces $z_0$. In all 
cases, $\chi(a)=\chi(e^{\ell})=1$. So we can apply Theorem 
\ref{zzthm6.5} in the situation of $q=1$.  Since 
$[z]^{\ell}$ is not a coboundary in $Z^2_{a,1}(K)$, 
Theorem \ref{zzthm6.5}(2) shows that $\phi^1_{g,1}$ is 
an isomorphism for all $g\in G$. Hence 
$\PP^1_{g,1}(H)=\PP^{1}_{g,1}(K)$.
Now part (1) follows from Corollary \ref{zzcor8.6}.

(2) Note that, by using coradical filtration, 
$\gr H\cong F_G(e,\chi,\ell)$ for every $H$ in this 
corollary. By Lemma \ref{zzlem8.4}(1),
we only need to show the assertion for $H=F_G(e,\chi,\ell)$.

Now $H$ is a graded Hopf algebra with $\deg(g)=0$
for all $g\in G$, $\deg(z)=1$ and $\deg (w)=\ell$. By a
result of Radford \cite{Ra}, $H\cong R\# kG$ 
where $R$ is the coinvariant subring $\{a\in H\mid (id\otimes
\pi)\Delta(a)=1\otimes a\}$ where $\pi: H\to kG$ is the
projection. Further, it is easy to see that, for all $n\geq 0$, 
$H_n\cong kG a_n$ where $a_n=z^t w^{m}$ if $n=t+\ell m$ for 
some $0\leq t<\ell$ and $m\geq 0$. This implies that the Hilbert
series of $R$ is $1+t+t^2+\cdots=\frac{1}{(1-t)}$. By part (1)
and Lemma \ref{zzlem8.5}, $\PP^1_{1,1}(R)$ is 1-dimensional.
This is equivalent to say that the dual algebra $R^\circ$
is generated by 1 element. Since the Hilbert series of
$R^\circ$ is $\frac{1}{(1-t)}$, $R^{\circ}\cong k[x]$ for 
$\deg x=1$. It is clear that the global dimension of 
$R^{\circ}$ is 1, or equivalently, the PC dimension of $R$ is 1. 
The assertion follows from Lemma \ref{zzlem8.5}.
\end{proof}

In fact, all the Hopf algebras introduced above are 
pointed. To show this we need the following lemma, 
which is a folklore. 

\begin{lemma}
\label{zzlem8.15}
Let $H$ be a Hopf algebra. Let $C$ be a subcoalgebra 
of $H$ which generates $H$ as an algebra. If $C$ is 
pointed, then so is $H$. Moreover, the group $G(H)$ 
is generated by $G(C)$.
\end{lemma}

\begin{proof} With the algebra $C$, one can construct 
$\mathcal{H}(C)$, the free Hopf algebra over 
$C$ \cite{Ta}. By the universal property of 
$\mathcal{H}(C)$, there is Hopf algebra map 
$\mathcal{H}(C)\rightarrow H$, which is surjective 
since $C$ generates $H$ as an algebra. By 
\cite[Proposition 3.3]{Zh}, $\mathcal{H}(C)$ 
is pointed and $G(\mathcal{H}(C))$ is generated by 
$G(C)$. Now the result follows from 
\cite[Corollary 5.3.5]{Mo}.
\end{proof}

\begin{proposition}
\label{zzpro8.16}
Let $H$ be one of the following Hopf algebras: 
$A_G(e, \chi)$, $C_G(e, \tau)$, $E_G(e, \chi, \ell, \lambda)$, 
$F_G(e,\chi, \ell)$, $L_G(e,\chi, \ell, \eta)$, $N_G(e,\chi, \ell, \xi)$
$O_{G}(e,\chi,\ell,\eta)$, $P_{G}(e,\chi,\ell,\eta)$, or $Q_{G}(e,\chi,\ell,\eta)$.
Then $H$ is pointed with grouplike elements $G$. Moreover,
$$\GKdim H= \GKdim kG +1,$$
except for $H=E_G(e, \chi, \ell, \lambda)$, where in that case, 
$$\GKdim H= \GKdim kG.$$
\end{proposition}

\begin{proof}
For such a Hopf algebra $H$, we can find a pointed 
subcoalgebra $C$ which contains $G$ and generates 
$H$ as an algebra. By Lemma \ref{zzlem8.15}, $H$ 
is pointed. For $H=E_G(e, \chi, \ell, \lambda)$, 
the statement $G(H)=G$ follows from Remark \ref{zzrem8.3}. 
For all the other cases, it is a result of Lemma \ref{zzlem5.1}(3).

The statements about GK-dimension follows 
from \cite[Theorem 5.4]{Zh}.
\end{proof}

\section{Pointed Hopf algebras of $\PCdim$ one}
\label{zzsec9}
Throughout this section, let $H$ be a pointed Hopf algebra. 
We are going to classify pointed Hopf algebras $H$ 
satisfying the following condition
\begin{equation}
\label{E9.0.1}\tag{E9.0.1}
\dim \{\bigoplus_{g\in G}\PP^1_{1, g}(H)\}=1.
\end{equation}

As mentioned before Lemma \ref{zzlem8.5}, 
the associated graded algebra $\gr H$ has a 
decomposition $\gr H\cong R\# kG$.

\begin{lemma}
\label{zzlem9.1}
Let $H$ be pointed with $G=G(H)$.
Then the following $k$-spaces are isomorphic:
\begin{align*}
&\bigoplus_{g\in G}\PP^1_{1, g}(H),  \quad
\bigoplus_{g\in G}\PP^1_{g, 1}(H), \quad
\bigoplus_{g\in G}\PP^1_{1, g}(\gr H), \quad
\bigoplus_{g\in G}\PP^1_{g, 1}(\gr H), \quad
R_1, \\
&\bigoplus_{g\in G}P'_{1, g}(H), \quad
\bigoplus_{g\in G}P'_{g, 1}(H), \quad
\bigoplus_{g\in G}P'_{g, 1}(\gr H), \quad
\bigoplus_{g\in G}P'_{1, g}(\gr H).
\end{align*}
\end{lemma}

Note that by Lemma \ref{zzlem9.1}, the assumption 
\eqref{E9.0.1} is equivalent to say that one of 
the vector spaces in the lemma is one-dimensional.

\begin{theorem} 
\label{zzthm9.2}
Suppose that ${\rm{char}}\; k=0$. 
Let $H$ be pointed with $G(H)=G$. 
Assume that $H$ satisfies \eqref{E9.0.1}. Then $H$ is 
isomorphic to one the following.

\begin{enumerate}
\item
$A_G(e, \chi)$ in Example \ref{zzex5.4} where $\chi(e)$ is either 1 or 
not a root of unity.
\item
$C_G(e, \chi, \tau)$ in Example \ref{zzex5.5}.
\item
$E_{G}(e,\chi, \ell,\lambda)$ in Example \ref{zzex8.1} 
where $\lambda$ is either 1 or $0$.
\item
$F_G(e, \chi, \ell)$ in Example \ref{zzex8.7}.
\item
$L_G(e, \chi, \ell, \eta)$ in Example \ref{zzex8.9}.
\item
$N_G(e, \chi, \ell, \xi)$ in Example \ref{zzex8.10}.
\item
$O_G(e,\chi,\ell,\eta)$ in Example \ref{zzex8.11}.
\item
$P_G(e,\chi,\ell,\eta)$ in Example \ref{zzex8.12}.
\item
$Q_G(e,\chi,\ell,\eta)$ in Example \ref{zzex8.13}.
\end{enumerate}
\end{theorem}

\begin{proof} First of all, every Hopf algebra in
(a)--(i) satisfies \eqref{E9.0.1}. Now we show that these
are only Hopf algebras satisfying \eqref{E9.0.1}.

By Lemma \ref{zzlem9.1} and condition \eqref{E9.0.1}, 
there exists exactly one element $e\in G$ such that 
$P'_{1, e}(H)$ is one-dimensional. In other words,
\begin{equation}
\label{E9.2.1}\tag{E9.2.1}
\PP^1_{e,1}(H)=k \overline{u}\neq 0, \quad {\text{and}} \quad
\PP^1_{g,1}(H)=0, \quad {\text{for all $g\neq e$}},
\end{equation}
for some $u\in P_{1, e}(H)\setminus kG$. For any 
$g\in G$, $g^{-1}ug \in P_{1, g^{-1}eg}(H)\setminus kG$.
By condition \eqref{E9.2.1}, $g^{-1}e g=e$, namely, 
$e$ is in the center of $G$ as required by Examples 
\ref{zzex5.4}-\ref{zzex5.5}. 
Hence there exist a character $\chi : G\rightarrow k^{\times}$ 
and a map $\partial: kG\rightarrow kG$ such that
\begin{equation}
\label{E9.2.2}\tag{E9.2.2}
g^{-1}ug=\chi(g)u+\partial(g).
\end{equation}
Since $g^{-1}ug-\chi(g)u$ is in $P_{1,e}(H)$, $\partial(g)
\in P_{1,e}(kG)$. So there is a 
map $\tau: G\to k$ such that $\partial(g)=\tau(g) (e-1)$.
Using \eqref{E9.2.2}, it is easy to check that $\tau$ satisfies
$$\tau(gh)=\tau(g)+\chi(g)\tau(h), 
\quad {\text{for all $g,h\in G$}}$$
\eqref{E5.5.1}. In the rest, we use the pair
$(\chi, \tau)$ to divide the proof into several cases.

Case 1. Suppose that $\chi(e)=1$. Using the definition
of $C_{G}(e,\chi,\tau)$ [Example \ref{zzex5.5}], 
one sees that there is a Hopf 
algebra map $\phi: C_G(e, \chi, \tau)\rightarrow H$ 
sending $z$ to $u$ and restricting to the identity on 
$G$. Condition \eqref{E9.0.1} implies that 
the skew primitive elements in $H$ are in $kG\oplus 
kG u$. So $\phi$ restricts to an isomorphism on the space 
of skew primitive elements. Now it follows from 
\cite[Corollary 5.4.7]{Mo} that $\phi$ is injective. In
this case, $C_{G}(e,\chi,\tau)$ is the Hopf subalgebra
of $H$ generated by grouplikes and skew primitives.  
By Corollary \ref{zzcor6.9}(2), $\PCdim C_{G}(e,\chi,\tau)
=1$. By Proposition \ref{zzpro2.4}(2), the map $\phi$ is 
actually an isomorphism.  This is (b). If, further, $\tau=0$,
then $C_{G}(e,\chi,0)=A_{G}(e,\chi)$ [Example \ref{zzex5.4}]
(where $\chi(e)=1$). So the Hopf algebra is in part (a)
(or (b)). Note that Hopf algebras in parts (a) and (b) 
have some overlaps. 

Case 2. Suppose that $\chi(e)\neq 1$. Set 
$u':=u+\frac{\tau(e)}{\chi(e)-1} (e-1)$, one can easily
check that $u' g=\chi(g) g u'$ for all $g\in G$. 
See \cite[Case 1, p. 805]{WYC} for a proof.
Replacing $u$ by $u'$ one can assume that 
$\tau$ is zero. For the rest of the proof we assume that 
$\tau=0$ and $ug=\chi(g) gu$ for all $g\in G$. 

Case 2.1. Suppose that $\chi(e)$ is not a root of unity. 
In this case, there is a Hopf algebra map 
$\phi: A_G(e, \chi)\rightarrow H$ sending $z$ to $u$ and 
restricting to the identity on $G$. Condition \eqref{E9.0.1} 
implies that the skew primitive elements in $H$ are in 
$kG\oplus kG u$. So $\phi$ restricts to an isomorphism on 
the space of skew primitive elements. It follows from 
\cite[Corollary 5.4.7]{Mo} that $\phi$ is injective.  In
this case, $A_{G}(e,\chi)$ is the Hopf subalgebra
of $H$ generated by grouplikes and skew primitives.  
By Corollary \ref{zzcor6.10}(2), $\PCdim A_{G}(e,\chi)
=1$. By Proposition \ref{zzpro2.4}(2), the map $\phi$ is 
actually an isomorphism.  This is (a).

Case 2.2. Suppose that $q:=\chi(e)$ is an $\ell$th primitive 
root of unity for some $\ell\ge 2$.

By Lemma \ref{zzlem5.3}, $\Delta(u^{\ell})=u^{\ell}\otimes 1
+e^{\ell}\otimes u^{\ell}$. So $u^{\ell}$ is in $P_{1, e^{\ell}}(H)$. 
Since $e^{\ell}\neq e$ (otherwise $\chi(e)$ will be an 
$(\ell-1)$th root of unity, which is impossible),  
condition \eqref{E9.0.1} implies that $P_{1, e^{\ell}}(H)=
k(e^{\ell}-1)$. Consequently, there exists $\lambda\in k$ such that 
\begin{equation}\label{E9.2.3}\tag{E9.2.3}
u^{\ell}=\lambda(e^{\ell}-1).
\end{equation}
If $e^{\ell}=1$, one may further assume $\lambda=0$.
As mentioned before, we have a Hopf algebra map $\phi: 
A_G(e, \chi)\rightarrow H$ which restricts to the 
identity on $G$ and sends $z$ to $u$. By $\eqref{E9.2.3}$, 
the element $\omega_\lambda:=z^{\ell}-\lambda(e^{\ell}-1)$
is in $\ker \phi$. Let $I$ be the ideal generated by 
$\omega_\lambda$. Clearly, there is a Hopf algebra map 
$\psi: A_G(e, \chi)/I\rightarrow H$ such that $\phi= \psi\circ\pi$ 
where $\pi : A_G(e, \chi)\rightarrow A_G(e, \chi)/I$ is the 
canonical morphism. As mentioned in Remark \ref{zzrem8.3}, we 
would not use $E_G(e, \chi, \ell, \lambda)$ for $A_G(e, \chi)/I$ 
until we make sure that $\omega_\lambda$ is normal in $A_G(e, \chi)$. 
Let $E= A_G(e, \chi)/I$. By \cite[Corollary 5.3.5]{Mo}, $\pi$ 
restricts to a surjective map from $G$ to $G(E)$. Since $\phi$ 
restricts to the identity on $G$, $\pi$ is injective restricting on 
$G$. Hence we can identify $G(E)$ with $G$. Moreover, $u\in \im\psi$ 
so $kG\subsetneq \im\psi$, which implies that $kG\subsetneq E$. Now 
by Lemma \ref{zzlem8.2}, $\omega_\lambda$ is normal in $A_G(e, \chi)$ 
and we have a Hopf algebra map $\psi: E_G(e, \chi, \ell, \lambda)\rightarrow H$.

Similar to an earlier 
argument, by Corollary \ref{zzcor8.6}(2) and \cite[Corollary 5.4.7]{Mo}, 
$\psi$ is injective. 

Case 2.2.1: If $\psi$ is an isomorphism, then $H\cong 
E_{G}(e, \chi, \ell, \lambda)$, which is (c). 

Case 2.2.2: For the rest, we assume that $\psi$ is not an isomorphism.
After identifying $u$ with $z$, we may assume that $\psi$ is an 
inclusion since it is injective. Set $K=E_{G}(e, \chi, \ell, \lambda)$. 
Then $K\subsetneq H$ and $C_i(K)=C_i(H)$ for $i=0,1$,
where $C_i$ means the $i$th term of the coradical filtration. 
By Corollary \ref{zzcor8.6}(3), $[z]^{\ell}$ is the only element 
in $\PP^2_{g,1}(K)$, up to a boundary and up to a scalar. By Proposition 
\ref{zzpro2.4}(1), there is $v\in H$ such that 
$\Delta(v)=v\otimes 1+ e^{\ell}\otimes v+ [z]^{\ell}$. 
By a simple calculation, for any $g\in G$,
$$g^{-1}vg-\chi(g^{\ell})v\in P_{1, e^{\ell}}(H).$$
By Remark \ref{zzrem6.2}, $e^{\ell}\neq e$. So, by \eqref{E9.2.1}, 
$g^{-1}vg-\chi^{\ell}(g)v\in k(1-e^{\ell})$. Then there is a map 
$\eta: G\rightarrow k$ such that
\begin{equation}
\label{E9.2.4}\tag{E9.2.4}
vg=\chi^{\ell}(g)gv+\eta(g)g(e^{\ell}-1).
\end{equation}
Similar to \eqref{E5.5.1}, one can check that 
\begin{equation}
\label{E9.2.5}\tag{E9.2.5}
\eta(gh)=\eta(g)+\chi^{\ell}(g)\eta(h), \quad 
{\text{for all $g,h\in G$}}.
\end{equation}
Note that $\chi^{\ell}(e)=\chi(e^{\ell})=1$.

By \eqref{E9.2.3}, \eqref{E9.2.4} and Lemma \ref{zzlem5.6}(1), we have
$$\begin{aligned}
\Delta([v,z]) &=[\Delta(v),\Delta(z)]
=[v\otimes 1+e^{\ell}\otimes v+[z]^{\ell},\Delta(z)]\\
&=[v\otimes 1+e^{\ell}\otimes v,z\otimes 1+e\otimes z]+
[[z]^{\ell},\Delta(z)]\\
&=[v,z]\otimes 1+e^{\ell+1}\otimes [v,z]\\
&\qquad+ [v,e]\otimes z+ 
(1-q)ez^{\ell}\otimes z+(q-1) e^{\ell}z\otimes z^{\ell}\\
&=[v,z]\otimes 1+e^{\ell+1}\otimes [v,z]\\
&\qquad+ \eta(e)e(e^{\ell}-1)\otimes z+ 
(1-q)\lambda e(e^{\ell}-1)\otimes z+(q-1)\lambda e^{\ell}z\otimes 
(e^{\ell}-1).
\end{aligned}
$$
Hence,
$$\partial^1_{e^{\ell+1},1}([v,z])=a e(e^{\ell}-1)\otimes z+ b
e^{\ell}z\otimes (e^{\ell}-1)$$
where 
\begin{equation}
\label{E9.2.6}\tag{E9.2.6}
a=-\eta(e)+(q-1)\lambda \quad {\text{and}}\quad
b=(1-q)\lambda.
\end{equation} 
A direct computation shows that
$$\partial^1_{e^{\ell+1},1}(az-be^{\ell}z)=a e(e^{\ell}-1)\otimes z+ b
e^{\ell}z\otimes (e^{\ell}-1).$$
So $\partial^1_{e^{\ell+1},1}([v,z]-az+be^{\ell}z)=0$ or
$[v,z]-az+be^{\ell}z\in Z^1_{e^{\ell+1},1}(H)$.
If $e^{\ell}\neq 1$, then $\PP^1_{e^{\ell+1},1}=0$ implies
that $[v,z]-az+be^{\ell}z\in B^1_{e^{\ell+1},1}(H)=k(e^{\ell+1}-1)$.
Therefore
\begin{equation}
\label{E9.2.7}\tag{E9.2.7}
[v,z]=az-be^{\ell}z +c (e^{\ell+1}-1)
\end{equation}
for some $c\in k$. If $e^{\ell}=1$ (and we take $\lambda=0$ and $b=0$), 
then 
$$
[v,z]-az+be^{\ell}z\in Z^1_{e^{\ell+1},1}(H)=Z^1_{e,1}(H)=k z+k(e-1),
$$
which implies that 
\begin{equation}
\label{E9.2.8}\tag{E9.2.8}[v,z]=(a-d)z+c(e-1)
\end{equation}
for some $c,d\in k$. Combining these two cases, we have
\begin{equation}
\label{E9.2.9}\tag{E9.2.9}
vz=zv+az-d e^{\ell} z+c(e^{\ell+1}-1)
\end{equation}
for some $c,d\in k$. If $e^{\ell}\neq 1$, then $d=b= (1-q)\lambda$.

Let $K'$ be the Hopf subalgebra of $H$ generated by $K$ and $v$.
We claim that $K'$ is isomorphic to one of the algebras
listed in (d)-(i) of Theorem \ref{zzthm9.2}. If the claim is true, 
then, by Corollary \ref{zzcor8.14} and Proposition \ref{zzpro2.4}(2),
$K'=H$ as desired. It remains to show the claim.

By using \eqref{E9.2.9} and the fact that $[\ell]_{q}=0$,
we have
$$\begin{aligned}
\; [v,z^{\ell}]&=\sum_{i=0}^{\ell-1} z^i [v,z] z^{\ell-i-1}
=\sum_{i=0}^{\ell-1} z^i (az-d e^{\ell}z+c (e^{\ell+1}-1)) z^{\ell-i-1}\\
&=\ell (a-de^{\ell})z^{\ell}
+c[\ell]_{q} e^{\ell+1} z^{\ell-1}-\ell cz^{\ell-1}
=\ell (a-de^{\ell})z^{\ell}-\ell cz^{\ell-1}.
\end{aligned}
$$
By \eqref{E9.2.3} and \eqref{E9.2.4}, we obtain that
\begin{equation}
\label{E9.2.10}\tag{E9.2.10}
\lambda \eta(e^{\ell})e^{\ell}(e^{\ell}-1)=\ell
(a-de^{\ell})\lambda (e^{\ell}-1)-\ell c z^{\ell-1}.
\end{equation}
As a consequence of \eqref{E9.2.10}, we have $c=0$. Then
\eqref{E9.2.7} and \eqref{E9.2.8} become
$$[v,z]=\begin{cases} (a-d) z & e^{\ell}=1,\\
az-b e^{\ell} z & e^{\ell}\neq 1.\end{cases}$$
Next we need to consider
multiple cases again.

Case 2.2.2(1): Suppose $e^{\ell}=1$. And we can take $\lambda=b=0$.
Further we can take $\eta=0$ in \eqref{E9.2.4}.
By equation \eqref{E9.2.8},  we have $[v,z]=\xi z$ where 
$\xi=a-d\in k$. 

If $\xi=0$, then $\delta=0$ and we can define 
a Hopf algebra surjection $F_{G}(e,\chi,\ell)\to K'$. 
By \cite[Corollary 5.4.7]{Mo}, this is an isomorphism.
So we obtain that $K'$ is the algebra in (d).

Now assume that $\xi\neq 0$.
Using \eqref{E9.2.2} (with $\partial(g)=0$), \eqref{E9.2.4} and 
the relation $[v,z]=\xi z$, we have,
\begin{align}
\label{E9.2.11}\tag{E9.2.11}
v(zg)&=v(\chi(g) gz)
=\chi(g) (\chi^{\ell}(g) gv)z\\
\notag
&=\chi^{\ell+1}(g) g(zv+\xi  z)\\
\notag
&=\chi^{\ell+1}(g) gzv+\xi\chi^{\ell+1}(g)  g  z,\\
\notag
(vz)g&=(zv+\xi z) g=z(\chi^{\ell}(g) gv)
+\xi\chi(g)  gz\\
\notag
&=\chi^{\ell+1}(g) gzv+\xi\chi(g)gz.
\end{align}
Hence $\xi(\chi^{\ell}(g)-1)=0$ for all $g$. 
Since $\xi\neq 0$, $\chi^{\ell}(g)=1$ for all $g\in G$
\eqref{E8.10.1}.  In this case 
we obtain the algebra in (f).
 
Case 2.2.2(2): Suppose $e^{2\ell}\neq 1$.
By \eqref{E9.2.5}, $\eta(e^i)=i\eta(e)$ for
all $i$. 

If $\lambda=0$, by \eqref{E9.2.6}, 
$b=0$ and $a=-\eta(e)$. In this case,
$[v,z]=az$, and by \eqref{E9.2.4}, we have
a computation similar to \eqref{E9.2.11},
\begin{align}
\label{E9.2.12}\tag{E9.2.12}
v(zg)&=v(\chi(g) gz)
=\chi(g) (\chi^{\ell}(g) gv+\eta(g)g (e^{\ell}-1))z\\
\notag
&=\chi^{\ell+1}(g) g(zv+a z)+\chi(g)\eta(g)g(e^{\ell}-1) z\\
\notag
&=\chi^{\ell+1}(g) gzv+\chi^{\ell+1}(g) a g  z
+\chi(g)\eta(g) g(e^{\ell}-1) z,\\
\notag
(vz)g&=(zv+az) g=z(\chi^{\ell}(g) gv+\eta(g)g(e^{\ell}-1))
+a\chi(g)  gz\\
\notag
&=\chi^{\ell+1}(g) gzv+\chi(g)\eta(g)g(e^{\ell}-1)z+a\chi(g)gz.
\end{align}
Hence $a(\chi^{\ell}(g)-1)=0$ for all $g$. 

If $\chi^{\ell}(g)\neq 1$ for some $g\in G$, then $a=0$. 
As a consequence, $\eta(e)=0$.
So we obtain that $K'$ is the algebra in (e).

If $\chi^{\ell}(g)=1$ for all $g\in G$ \eqref{E8.12.1}, then $\eta$ is an additive
character of $G$ by \eqref{E9.2.5}.  We obtain the algebra in (h).

Now assume $\lambda\neq 0$. We might assume that $\lambda=1$.
Then \eqref{E9.2.10} implies that
$$\eta(e)e^{\ell}=a-b e^{\ell}$$
which is equivalent to $a=0$ and $\eta(e)=-b$. Hence 
\begin{equation}
\label{E9.2.13}\tag{E9.2.13}
vz=zv+\xi e^{\ell} z
\end{equation}
where $\xi=-b=\eta(e)=(q-1)\lambda=q-1\neq 0$. 
Using \eqref{E9.2.2} (with $\partial(g)=0$), \eqref{E9.2.4} and 
\eqref{E9.2.13}, we have a computation similar to \eqref{E9.2.12}
and 
obtain that $\xi(\chi^{\ell}(g)-1)=0$ for all $g$. 
Since $\xi\neq 0$,  $\chi^{\ell}(g)=1$ for all $g\in G$
\eqref{E8.11.1}.
Then $\eta$ is an additive
character of $G$ by \eqref{E9.2.5}. In this case $K'$ is the
algebra in (g).

Case 2.2.2(3): Suppose $e^{\ell}\neq 1$ and $e^{2\ell}= 1$.
If $\lambda=0$, the argument in the first half of Case 2.2.2(2) 
works, and $K'$ is either in (e) or (h). 
It remains to consider the case when $\lambda\neq 0$, 
or $\lambda=1$. 

Using the fact $e^{2\ell}=1$, \eqref{E9.2.10} implies that
\begin{equation}
\label{E9.2.14}\tag{E9.2.14}
\eta(e)=-(a+d)=-(a+b)
\end{equation}
which is equivalent to \eqref{E9.2.6}. Hence 
\begin{equation}
\label{E9.2.15}\tag{E9.2.15}
vz=zv+ az-b e^{\ell} z
\end{equation}
where $b=1-q\neq 0$. A computation similar to both
\eqref{E9.2.11} and \eqref{E9.2.12} shows that
$$a(\chi^{\ell}(g)-1)=b(\chi^{\ell}(g)-1)=0$$
for all $g$. Since $b\neq 0$, $\chi^{\ell}(g)=1$
for all $g$. In this case \eqref{E9.2.5} implies
that $\eta$ is additive. Since ${\rm{char}}\; k=0$ and
$e^{2\ell}=1$, we obtain that $\eta(e)=0$. Together with
\eqref{E9.2.14} and \eqref{E9.2.15}, we have
$$[v,z]=(q-1) z+(q-1)e^{\ell}z.$$
Now the data appeared in this subcase match up with that given in 
Example \ref{zzex8.13}. Hence $K'$ is the algebra in (i). This 
completes the proof of the claim and therefore the proof of the
entire theorem.
\end{proof}

Recall that a Hopf algebra $H$ is called to \emph{have rank one}
if $C_1(H)$ is a free left $C_0(H)$-module of rank two 
\cite{KR, WYC}. 

\begin{proof}[Proof of Theorem \ref{zzthm0.2}]
When $H$ is pointed, $H$ has rank one if and only if
$C_1(H)=kG\oplus kGz$ for some skew primitive element $z$. 
This is equivalent to condition \eqref{E9.0.1}.
So Theorem \ref{zzthm0.2} follows from Theorem 
\ref{zzthm9.2}.
\end{proof}

Let $H$ be a pointed Hopf algebra with $G(H)=G$. Define
$$pHil_{H}(t)=\sum_{n=0}^{\infty} (\dim \{\oplus_{g\in G}
\PP^{n}_{g,1}(H)\}) t^n.$$

By the classification the following is easy.

\begin{corollary}
\label{zzcor9.3} Suppose ${\rm{char}}\; k=0$. Let $H$ be a pointed
Hopf algebra of rank one. Then 
$$pHil_H(t)=\begin{cases} \frac{1}{1-t} & \emph{if} \quad 
H= E_{G}(e,\chi, \ell, \lambda)\\
1+t & {\rm{otherwise}},\end{cases}$$
and
$$\PCdim H=\begin{cases} \infty & \emph{if} \quad 
H= E_{G}(e,\chi, \ell, \lambda)\\
1 & {\rm{otherwise}}.\end{cases}$$
\end{corollary}

\begin{proof} This follows from Corollaries \ref{zzcor6.9},
\ref{zzcor6.10}, \ref{zzcor8.6} and \ref{zzcor8.14}. 
\end{proof}

We make a few remarks.

\begin{remark}
\label{zzrem9.4} 
\begin{enumerate}
\item[(1)]
Consider the Hopf filtration $F$ of $H:=F_G(e,\chi,\ell)$ by defining
$\deg (h)=0$ for all $h\in E_{G}(e,\chi,\ell,0)$ and $\deg w=1$.
Then $$\PCdim \gr_F H=\infty>1=\PCdim H.$$
\item[(2)]
It is a very interesting project to prove a version of Theorem
\ref{zzthm9.2} in positive characteristic. Scherotzke started 
this in \cite{Sc} and classified all Hopf algebras of rank one
that are generated by $C_1(H)$. The work of Wang \cite{W1}, Wang-Wang
\cite{WW} and Nguyen-Wang-Wang \cite{NWW} also
provides some very interesting examples. 
\item[(3)]
Let $R$ be the free algebra $k\langle x_1,\cdots,x_w\rangle$ with 
$x_i$ being primitive. Then $\gldim R=1$ and 
$$\Tor^R_{n}(k,k)=\begin{cases} k & n=0,\\
\oplus_{i=1}^w kx_i & n=1,\\ 0 & n\geq 2.
\end{cases}
$$
Let $H$ be the graded dual Hopf algebra of $R$. Note that $H$ is 
neither affine or noetherian. By Lemma \ref{zzlem3.6}(2), 
$\PP^n_{1,1}(H)\cong (\Tor^R_n(k,k))^\circ$.
As a consequence, $\PCdim H=1$ and $pHil_H(t)=1+w t$.
\item[(4)]
We conjecture that if a pointed Hopf algebra $H$ is affine 
(or noetherian) with $\PCdim$ 1, then $H$ has rank one. 
Consequently, $pHil_H(t)=1+t$. This conjecture holds in a special
case, see the next lemma.
\item[(5)]
It is impossible to get a reasonable classification of
pointed Hopf algebras of rank two without any extra condition,
see a related comment in Remark \ref{zzrem9.8}(3).
\end{enumerate}
\end{remark}

\begin{lemma}
\label{zzlem9.5} Let $H$ be pointed and coradically graded. If 
$\PCdim H=1$ and $H$ is left noetherian, then $H$ has rank one.
\end{lemma}

\begin{proof} Since $H$ is coradically graded, $H=R\# kG$. 
Since $H$ is left noetherian, so is $R$. By 
\cite[Theorem 0.1]{SZ}, $R$ is locally 
finite and has sub-exponential growth. So $R^{\circ}$
is locally finite and has sub-exponential growth.
By Lemma \ref{zzlem8.5},
$\gldim R^{\circ}=1$. Then $R^{\circ}$ is a free algebra.
The only free algebra with sub-exponential growth is $k[t]$.
Therefore $R^{\circ}\cong k[t]$. This implies that
$H$ has rank one.
\end{proof}

For the rest of this section we introduce some invariants
of $H$ which are related to the rank of $H$.
The signature of a connected Hopf algebra or of a coideal 
subalgebra of a connected Hopf algebra was introduced by
Gilmartin \cite{Gi} and Brown-Gilmartin \cite{BG}. 
Here we define a version of signature for pointed Hopf
algebras.

\begin{definition}
\label{zzdef9.6} 
Let $H$ be a pointed Hopf algebra and $\gr H$ be the 
associated graded Hopf algebra with respect to the coradical
filtration. Write $\gr H=R\# G(H)$ where $R=(\gr H)^{co \; \pi}$
where $\pi: \gr H\to G(H)$ is the projection. 
The \emph{signature space of $H$} is defined to be the generating space
of $R$, namely, $\sign(H)=R_{\geq 1}/(R_{\geq 1})^2.$
The \emph{signature series of $H$}, denoted by $sHil_H(t)$, is the 
Hilbert series of the graded space $\sign(H)$. 
The \emph{signature of $H$}, denoted by $S(H)$, is the sequence of the 
degrees (with 
multiplicity) of elements in any homogeneous basis of $\sign(H)$.
\end{definition}

\begin{example}
\label{zzex9.7}
We list some examples of signature series.
\begin{enumerate}
\item[(1)]
If $H=U(\fg)$, then $sHil_H(t)=(\dim \fg) t$.
\item[(2)]
Let $H$ be the Hopf algebras $A(1,\lambda,0)$ and
$B(\lambda)$ in \cite[Theorem 1.3]{Zh}, then 
$sHil_H(t)=2t+t^2$.
\item[(3)]
If $H=F_G(e,\chi,\ell)$, then $sHil_H(t)=
t+t^{\ell}$.
\end{enumerate}
\end{example}

\begin{remark}
\label{zzrem9.8} 
\begin{enumerate}
\item[(1)]
If $H$ is connected, then the dual space 
$\sign(H)^{\circ}$, which is called \emph{lantern of $H$}, 
has a natural Lie algebra structure. It is expected that
$\sign(H)^{\circ}$ in general has a braided Lie algebra
structure.
\item[(2)]
When $H$ is connected, the signature has to satisfy certain
conditions. For example, $t+t^{\ell}$ cannot be a signature series
of any connected Hopf algebra over a field of characteristic zero. 
But this is a signature series
for a pointed Hopf algebra [Example \ref{zzex9.7}(3)].
\item[(3)]
The rank of $H$ is the dimension of $R_1$, so it is a part 
of signature. By Theorem \ref{zzthm0.2}, if $H$ has rank one, 
then $sHil_H(t)$ is either $t$ or $t+t^{\ell}$. If $H$ has rank 
two, then there are infinitely many possible $sHil_H(t)$.
\end{enumerate}
\end{remark}

\section{Primitive cohomological ring}
\label{zzsec10}

In this section we consider some structures 
on primitive cohomology ring
$$\PP(H):=\bigoplus_{n\geq 0} \{\bigoplus_{g\in G(H)}  
\PP^n_{1,g}(H)\}.$$
Note that we choose $\PP^n_{1,g}(H)$ instead of $\PP^n_{g,1}(H)$
in this section, because $\PP^n_{1,g}(H)$ seems easier and more 
natural to use when we have both product and differential.

First we define a product on the primitive cohomology ring. Let $H$ 
be a bialgebra and let $G$ be the (semi)group of grouplike elements 
in $H$. Let $g$ and $h$ be two elements in $G$. Let $x\in T_{1,g}(H)^n$ 
and $y\in T_{1,h}(H)^m$ (see Section \ref{zzsec1}). The \emph{product} 
of $x$ and $y$ is defined to be 
\begin{equation}
\label{E10.0.1}\tag{E10.0.1}
x\odot y=x \otimes ((g\otimes \cdots \otimes g) y)
\in T_{1,gh}(H)^{n+m}.
\end{equation}
We will show that $\Omega_{1,G}:=
(\bigoplus_{g\in G} T_{1,g}(H), \odot, \partial)$ 
is a dg algebra where $\partial=(\partial_{1,g}^n)$.

We consider a slightly more general situation. Let $D$ be a 
subbialgebra of $H$, and define a complex 
$T_{1,D}(H):=X^{\geq 0}\square_H D$
where $X$ is the Hochschild complex for $H$ defined as in \eqref{E3.4.1}. 
It is easy to see that $T_{1,D}(H)^n= H^{\otimes n}\otimes D$ which we 
denote  by $H^{\otimes n}\boxtimes D$. We will define a 
product on $T_{1,D}(H)$ which generalizes \eqref{E10.0.1}. Every element 
$f\in T_{1,D}(H)^n$ can be written as
$f=\sum f_i\boxtimes d_i$
where $f_i\in H^{\otimes n}$ and $d_i\in D$. For every $f\boxtimes c
\in T_{1,D}(H)^n$ and $g\boxtimes d\in T_{1,D}(H)^m$ we define the 
{\it product} of these two elements by
\begin{equation}
\label{E10.0.2}\tag{E10.0.2}
\quad
\end{equation}
$$(f\boxtimes c)\odot (g\boxtimes d)=\sum (f\otimes (c_1\otimes \cdots
\otimes c_m)g)\boxtimes c_{m+1} d=\sum (f\otimes \Delta^{m-1}(c_1)g)
\boxtimes c_{2} d$$
where $\Delta^{m}(c)=\sum c_1\otimes \cdots \otimes c_{m+1}$.
Note that 
$$\Delta^m=\begin{cases} \epsilon & m=-1,\\
Id &m=0,\\
(\Delta\otimes Id^{m-1})\cdots \Delta & m\geq 1.
\end{cases}
$$
It is easy to see that \eqref{E10.0.2} is equivalent to \eqref{E10.0.1}
when $D=k G(H)$. 

Let $f\boxtimes c\in T_{1,D}(H)^n$. We define the differential $\partial$ by 
$$\partial (f\boxtimes c):=(1\otimes f)\boxtimes c-D_{n}(f)\boxtimes c
+(-1)^{n+1} \sum (f\otimes c_1)\boxtimes c_2$$
where $D_n$ is defined in \eqref{E1.0.2} and $\sum c_1\otimes c_2=\Delta(c)$.

Note that the product on $T_{1,D}(H)$ can be defined when $D$ is any left 
$H$-comodule algebra. In particular, $D$ can be taken to be a left coideal 
subalgebra of $H$.

\begin{lemma}
\label{zzlem10.1} Let $x\in T_{1,D}(H)^n, y\in T_{1,D}(H)^m$ and
$z\in T_{1,D}(H)^p$ for $n,m,p\in {\mathbb N}$. Then
the following hold.
\begin{enumerate}
\item[(1)]
$$(x\odot y)\odot z=x\odot (y\odot z)\in T_{1,D}(H)^{n+m+p}.$$
\item[(2)]
$$\partial^{n+m}(x\odot y)=\partial^n(x)\odot y+
(-1)^{n} x\odot \partial^m(y).$$
\item[(3)]
$(T_{1,D}(H), \odot, \partial)$ is a dg algebra
where $\partial=(\partial^n)$.
\end{enumerate}
\end{lemma}

\begin{proof} (1) By linearity we may assume that $x=f\boxtimes c$,
$y=g\boxtimes d$ and $z=h\boxtimes e$ where $c,d,e\in D$. For 
simplicity we will omit $\sum$ in various places. 
$$\begin{aligned}
(x\odot y)&\odot z=
[(f\otimes (c_1\otimes \cdots \otimes c_n)g)\boxtimes c_{n+1} d ]\odot
[h\boxtimes e]\\
&=[(f\otimes (c_1\otimes \cdots \otimes c_n)g)\otimes (c_{n+1}d_1\otimes 
\cdots \otimes c_{n+m}d_{m})h]\boxtimes c_{n+m+1}d_{m+1}e\\
&=[f\boxtimes c]\odot [g\otimes (d_1\otimes 
\cdots \otimes d_{m})h]\boxtimes d_{m+1}e\\
&=x\odot (y\odot z).
\end{aligned}
$$

(2) Again we may assume that $x=f\boxtimes c\in T_{1,D}(H)^n$ and 
$y=g\boxtimes d\in T_{1,D}(H)^m$ and we write $\partial$ for all 
$\partial^n$.
$$\begin{aligned}
\partial(x\odot y)&=\partial 
[(f\otimes (c_1\otimes \cdots \otimes c_n)g)\boxtimes c_{n+1} d]\\
&=1\otimes [(f\otimes (c_1\otimes \cdots \otimes c_n)g)\boxtimes c_{n+1} d]\\
& \qquad -[(D_n(f)\otimes (c_1\otimes \cdots \otimes c_n)g)\boxtimes 
c_{n+1} d]\\
&\qquad +(-1)^{n+1} [(f\otimes (c_1\otimes \cdots \otimes c_n\otimes c_{n+1})
D_{m}g)\boxtimes c_{n+2} d]\\
&\qquad +(-1)^{n+m+1} 
[(f\otimes (c_1\otimes \cdots \otimes c_n)g)\otimes c_{n+1}d_1 \boxtimes 
c_{n+2} d_2]\\
&=[(1\otimes f)\boxtimes c]\odot [g\boxtimes d]
-[D_n(f)\boxtimes c] \odot [g\boxtimes d]\\
&\qquad +(-1)^{n+1} [f\boxtimes c]\odot [D_{m}g\boxtimes d]
+(-1)^{n+m+1} 
[f\boxtimes c]\odot [(g\otimes d_1) \boxtimes d_2]\\
&=\partial(f\boxtimes c)\odot (g\boxtimes d)
+(-1)^n [(f\otimes c_1)\boxtimes c_2]\odot (g\boxtimes d)\\
&\qquad +(-1)^{n} (f\boxtimes c)\odot \partial(g\boxtimes d)
-(-1)^n (f\boxtimes c)\odot (1\otimes g\boxtimes d)\\
&=\partial(f\boxtimes c)\odot (g\boxtimes d)+
(-1)^{n} (f\boxtimes c)\odot \partial(g\boxtimes d)\\
&=\partial(x)\odot y+(-1)^{n} x\odot \partial(y).
\end{aligned}
$$

(3) It is easy to check that $1_k \boxtimes 1$ is the 
identity of $T_{1,D}(H)$ and $\partial(1_k\boxtimes 1)=0$. 
By parts (1,2), $(T_{1,D}(H),\odot,\partial)$ is a 
dg algebra with identity. 
\end{proof}

\begin{definition}
\label{zzdef10.2} 
Let $H$ be a bialgebra and $D$ be a subbialgebra of $H$.
\begin{enumerate}
\item[(1)]
The {\it $D$-cobar construction} of $H$, or \emph{cobar construction
of $H$ based on $D$} is defined to be the dg algebra
$\Omega_{1,D}(H):=(T_{1,D}(H),\odot, \partial).$
\item[(2)]
The {\it $D$-cohomology ring} of $H$ is defined to 
be
$$\PP_{D}(H):={\text{H}}^*(\Omega_{1,D}(H))=
\bigoplus_{n\geq 0} \Cotor^n_H(k,D).$$
This is a connected graded $k$-algebra.
The degree $n$ piece $\Cotor^n_H(k,D)$ is also denoted by
$\PP_{D}^n(H)$.
\item[(3)]
If $D=k$, $\PP_{k}(H)$, also denoted by $\PP_{1}(H)$,  is 
called the {\it connected cohomology ring} 
(or simply {\it cohomology ring}) of $H$.
\item[(4)]
Let $G=G(H)$ and $D=kG$. The {\it primitive cohomology ring} 
of $H$ is defined to be $\PP(H)=\PP_{kG}(H)$, or equivalently,
$$\PP(H):={\text{H}}^*(\Omega_{1,G}(H))=
\bigoplus_{n\geq 0,g\in G} \PP^n_{1,g}(H).$$
The primitive cohomology ring $\PP(H)$ is a 
${\mathbb Z}\times G$-graded algebra. 
The ${\mathbb Z}$-grading is called \emph{cohomological grading}, 
and $G$-grading, which is induced from the internal structure 
of $H$, is just called \emph{$G$-grading}. This is a connected 
graded $k$-algebra if we forget about the $G$-grading. 
\end{enumerate}
\end{definition}


Next show that there are natural Hopf actions on the
$D$-cohomology ring. Let $H$ be a bialgebra and $D$ 
be a subbialgebra. Suppose $K\subset H$ is another Hopf subalgebra 
of $H$, namely, $K$ is a subbialgebra of $H$ with an antipode.
In most cases, $K=D$.  For every $a\in K$, the left adjoint action 
of $a$ on $D$ is defined  to be 
$$ad_l(a) (c)=\sum a_1 cS(a_2)\in H$$
for all $c\in D$. When $H$ is commutative, then the action $ad_l(a)$
is trivial. We say $D$ is $K$-normal if $ad_l(K)(D)
\subset D$. If all elements in $K$ commute with elements
in $D$, then $ad_l(a)(c)=c$. In this case $D$ is $K$-normal
and the left adjoint action of $K$ on $D$ is trivial. If
$K=H=kG$ and $D$ is a subgroup algebra, then $K$-normality
agrees with the normality in group theory. 

Let $\Delta^n$ denote $(\Delta\otimes Id^{n-1})(\Delta\otimes Id^{n-2})
\cdots \Delta$. Suppose now that $D$ is $K$-normal. For 
every $a\in K$ we define the action of $a$ on $T_{1,D}(H)$
as follows:
$$\begin{aligned}
ad_l(a) (f_1\otimes \cdots \otimes & f_n\boxtimes c)\\
&=\sum a_1 f_1 S(a_{2n+2})\otimes \cdots \otimes a_n f_n S(a_{n+3})\boxtimes 
a_{n+1} c S(a_{n+2})\\
&= \sum \Delta^{n}(a_1) (f_1\otimes \cdots \otimes f_n\boxtimes c)
\Delta^{n}(S(a_2))\\
&= \sum \Delta^{n}(a_1) (f_1\otimes \cdots \otimes f_n\boxtimes c)
S^{\otimes (n+1)}\Delta^{n}(a_2)
\end{aligned}
$$
for all $f_1\otimes \cdots f_n\boxtimes c\in T_{1,D}(H)^n$.

\begin{lemma}
\label{zzlem10.3} Suppose that $D$ is $K$-normal.
\begin{enumerate}
\item[(1)]
Let $a,b$ be elements in $K$ and let $x\in T_{1,D}(H)^n$.
Then 
$$ad_l(ab) (x)=ad_l(a)(ad_l(b)(x)).$$
This means that $ad_l$ is a natural left $K$-action on $T_{1,D}(H)$. 
\item[(2)]
Suppose that $K$ is cocommutative.
Let $a$ be an element in $K$ and let $x\in T_{1,D}(H)^n$
and $y\in T_{1,D}(H)^m$.
Then 
$$ad_l(a)(x\odot y)=\sum ad_l(a_1)(x)\odot ad_l(a_2)(y).$$
\item[(3)]
Suppose that $K$ is cocommutative.
$$\partial(ad_l(a)(x))=ad_l(a)(\partial (x))$$
for all $x\in T_{1,D}(H)$.
\end{enumerate}
\end{lemma}

\begin{proof} (1) By linearity, we may assume that 
$x=f_1\otimes \cdots \otimes f_n\boxtimes c$. 
Write $\Delta(a)=\sum a_1\otimes a_2$ and 
$\Delta(b)=\sum b_1\otimes b_2$. Then 
$\Delta(ab)=\sum a_1 b_1\otimes a_2 b_2$. Hence
$$\begin{aligned}
ad_l(ab) (x)&= \sum \Delta^n(a_1 b_1) x \Delta^n(S(a_2 b_2))\\
&=\sum \Delta^n(a_1)\Delta^n(b_1) x \Delta^n(S(b_2))\Delta^n(S(a_2))\\
&=ad_l(a)(ad_l(b)(x)).
\end{aligned}
$$

(2) Write $x=f\boxtimes c=f_1\otimes \cdots \otimes f_n\boxtimes c$ and
$y=g\boxtimes d=g_1\otimes \cdots \otimes g_m\boxtimes d$. Then
$$x\odot y=\sum (f\otimes \Delta^{m-1}(c_1)g)\boxtimes c_2d.$$
Hence we have, following the cocommutativity of $K$,
$$\begin{aligned}
&\sum ad_l(a_1)(x)\odot ad_l(a_2)(y)\\
&= \sum \Delta^{n-1}(a_1) f \Delta^{n-1}(S(a_4))
\boxtimes a_2 c S(a_3) \odot 
\Delta^{m-1}(a_5) g \Delta^{m-1}(S(a_8))
\boxtimes a_6 d S(a_7)\\
&=\sum \Delta^{n-1}(a_1) f \Delta^{n-1}(S(a_2))
\boxtimes a_3 c S(a_4) \odot 
\Delta^{m-1}(a_5) g \Delta^{m-1}(S(a_6))
\boxtimes a_7 d S(a_8)\\
&=\sum (\Delta^{n-1}(a_1) f \Delta^{n-1}(S(a_2))
\otimes \Delta^{m-1}((a_3 c S(a_4)_1)) 
\Delta^{m-1}(a_5) g \Delta^{m-1}(S(a_6)))\\
&\qquad\qquad\qquad \qquad\qquad\qquad 
\boxtimes (a_3 c S(a_4))_2 a_7 d S(a_8)\\
&=\sum (\Delta^{n-1}(a_1) f \Delta^{n-1}(S(a_2))
\otimes \Delta^{m-1}(a_3 c_1 S(a_4)) 
\Delta^{m-1}(a_5) g \Delta^{m-1}(S(a_6)))\\
&\qquad\qquad\qquad \qquad\qquad\qquad 
\boxtimes a_7 c_2 S(a_8) a_9 d S(a_{10})\\
&=\sum (\Delta^{n-1}(a_1) f \Delta^{n-1}(S(a_2))
\otimes \Delta^{m-1}(a_3 c_1) g \Delta^{m-1}(S(a_6)))\\
&\qquad\qquad\qquad \qquad\qquad\qquad 
\boxtimes a_7 c_2 d S(a_{10})\\
&=\sum (\Delta^{n-1}(a_1) f \Delta^{n-1}(S(a_2))
\otimes \Delta^{m-1}(a_3)\Delta^{m-1}(c_1) g \Delta^{m-1}(S(a_6)))\\
&\qquad\qquad\qquad \qquad\qquad\qquad 
\boxtimes a_7 c_2 d S(a_{10})\\
&=ad_l(a)(\sum f \otimes \Delta^{m-1}(c_1) g \boxtimes c_2 d )\\
&= ad_l(x\odot y).
\end{aligned}
$$

(3) Note that, for all $a,f_i\in H$,
$$D_n (\Delta^{n-1}(a) (f_1\otimes \cdots \otimes f_n))
=\Delta^{n}(a) D_n(f_1\otimes \cdots\otimes f_n)$$
and
$$D_n ((f_1\otimes \cdots \otimes f_n)\Delta^{n-1}(a) )
=D_n(f_1\otimes \cdots\otimes f_n)\Delta^{n}(a).$$
Since $K$ is cocommutative, 
$\sum a_1 S(a_n)\otimes a_2\otimes \cdots \otimes a_{n-1}=
\sum 1 \otimes a_2\otimes \cdots \otimes a_{n-1}$ where
$\Delta^{n-1}(a)=\sum a_1\otimes \cdots \otimes a_n$.
Write $x=f\boxtimes c=f_1\otimes \cdots \otimes f_n\boxtimes c$. 
Using the cocommutativity, we have
$$\begin{aligned}
ad_l(a) (\partial(x))&=
ad_l(a)(1\otimes f\boxtimes c-D_n(f) \boxtimes c+\sum (-1)^{n+1}
f\otimes c_1\boxtimes c_2)\\
&=\sum \Delta^n(a_1)(1\otimes f)\Delta^n(S(a_2))\boxtimes a_3c S(a_4)\\
&\qquad\quad 
-\Delta^n(a_1)D_n(f)\Delta^n(S(a_2))\boxtimes a_3c S(a_4)\\
&\qquad\quad 
+(-1)^{n+1}\Delta^n(a_1)(f\otimes c_1)\Delta^n(S(a_2))\boxtimes a_3c_2 S(a_4)
\\
&=\sum (1\otimes \Delta^{n-1}(a_1)f\Delta^{n-1}(S(a_2)))\boxtimes a_3c S(a_4)\\
&\qquad\quad 
-D_n(\Delta^{n-1}(a_1)f\Delta^{n-1}(S(a_2)))\boxtimes a_3c S(a_4)\\
&\qquad\quad 
+(-1)^{n+1}(\Delta^{n-1}(a_1)f\Delta^{n-1}(S(a_2))\otimes a_3c_1 S(a_4))
\boxtimes a_5c_2 S(a_6)\\
&=\partial(\sum \Delta^{n-1}(a_1)f\Delta^{n-1}(S(a_2))\boxtimes a_3c S(a_4))\\
&=\partial(ad_l(a)(x)).
\end{aligned}
$$
\end{proof}

Now we have the following result.

\begin{proposition}
\label{zzpro10.4} Let $H$ be a bialgebra and $D$ be a subbialgebra
of $H$. Let $K$ be a cocommutative subbialgebra of $H$ such that
(a) $K$ is a Hopf algebra and (b) $D$ is $K$-normal. Then there is a
natural left adjoint action of $K$ on the $D$-cohomology ring
$\PP_{D}(H)$ induced by the action of $K$ on $T_{1,D}(H)$.
\end{proposition}


\providecommand{\bysame}{\leavevmode\hbox to3em{\hrulefill}\thinspace}
\providecommand{\MR}{\relax\ifhmode\unskip\space\fi MR }
\providecommand{\MRhref}[2]{%

\href{http://www.ams.org/mathscinet-getitem?mr=#1}{#2} }
\providecommand{\href}[2]{#2}

\end{document}